\documentclass[english, 12pt]{article}
\usepackage[T1]{fontenc}
\usepackage{geometry}
\usepackage[latin9]{inputenc}
\usepackage{amsthm}
\usepackage{amsmath}
\usepackage{amssymb}
\usepackage{authblk}
\usepackage[margin=2cm, labelfont = bf]{caption}
\usepackage{setspace}
\usepackage{esint}
\usepackage{enumerate}
\usepackage{etoolbox}
\usepackage{url}
\usepackage{graphicx, epstopdf}
\usepackage[english]{babel}
\usepackage{xcolor}
\usepackage{soul}
\usepackage{tikz}
\usetikzlibrary{automata,arrows,positioning,calc}
\usepackage{float}

\numberwithin{equation}{section}

\makeatletter
\theoremstyle{plain}
\newtheorem{thm}{\protect\theoremname}[section]
\ifx\proof\undefined
\newenvironment{proof}[1][\protect\proofname]{\par
	\normalfont\topsep6\p@\@plus6\p@\relax
	\trivlist
	\itemindent\parindent
	\item[\hskip\labelsep
	\scshape
	#1]\ignorespaces
}{%
	\endtrivlist\@endpefalse
}
\providecommand{\proofname}{Proof}
\fi
\theoremstyle{plain}
\newtheorem{lem}[thm]{\protect\lemmaname}
\theoremstyle{plain}
\newtheorem{prop}[thm]{\protect\propositionname}
\theoremstyle{plain}

\theoremstyle{definition}
\newtheorem{defn}[thm]{\protect\definitionname}
\theoremstyle{remark}
\newtheorem{rem}[thm]{\protect\remarkname}
\theoremstyle{plain}
\newtheorem{cor}[thm]{\protect\corollaryname}
\theoremstyle{definition}
\newtheorem{example}[thm]{\protect\examplename}
\numberwithin{figure}{section}

\makeatother

\usepackage{babel}
\providecommand{\conjecturename}{Conjecture}
\providecommand{\corollaryname}{Corollary}
\providecommand{\definitionname}{Definition}
\providecommand{\examplename}{Example}
\providecommand{\lemmaname}{Lemma}
\providecommand{\propositionname}{Proposition}
\providecommand{\remarkname}{Remark}
\providecommand{\theoremname}{Theorem}

\renewcommand{\P}{\ensuremath{\mathbb{P}}}
\newcommand{\E}{\ensuremath{\mathbb{E}}}

\newtheorem{assump}[thm]{Assumption}
\usepackage{mathrsfs}

\begin{document}
		\date{}
	

	\title{\vspace{20pt}At the Mercy of the Common Noise: Blow-ups in a Conditional McKean--Vlasov Problem\vspace{40pt}}
	\author[1]{Sean Ledger}
	\author[2]{Andreas S{\o}jmark}
	\affil[1]{School of Mathematics, University of Bristol and Heilbronn Institute for Mathematical Research, BS8 1TW, UK }
	\affil[2]{Department of Mathematics, Imperial College London, London, SW7 2BU, UK\vspace{15pt}}

	\maketitle

	\begin{abstract} We extend a model of positive feedback and contagion in large mean-field systems, by introducing a common source of noise driven by Brownian motion. Although the driving dynamics are continuous, the positive feedback effect can lead to `blow-up' phenomena  whereby solutions develop jump-discontinuities. Our main results are twofold and concern the conditional McKean--Vlasov formulation of the model. First and foremost, we show that there are global solutions to this McKean--Vlasov problem which can be realised as limit points of a motivating particle system with common noise. Furthermore, we derive results on the occurrence of blow-ups, thereby showing how these events can be  triggered or prevented by the pathwise realisations of the common noise.
	\end{abstract}


\newpage
\section{Introduction} 
\label{Sect_Intro}

\noindent This paper studies a model of positive feedback and contagion in large mean-field systems, focusing on the interplay between positive feedback loops and a common source of noise in the driving dynamics. Specifically, we consider the convergence of the underlying finite-dimensional particle system as the number of particles tends to infinity, and we illuminate the emergence of `blow-up' phenomena in the limiting McKean--Vlasov problem
\begin{align} \label{MV} \tag{MV}
\begin{cases}
X_t = X_0 + \sqrt{1-\rho^2}B_t + \rho B_t^0 - \alpha L_t \\[2pt]
L_t = \mathbb{P}(\tau\leq t \,| \,B^{0}) \\[2pt]
\tau = \inf\{ t \geq 0 : X_t \leq 0 \}.
\end{cases}
\end{align}
Here the random drivers $B$ and $B^0$ are independent Brownian motions, and the start point $X_0$ is an independent random variable taking values in the positive half-line $(0,\infty)$, while $\alpha>0$ and $\rho\in[0,1)$ are constant parameters. By a solution to \eqref{MV} we understand an increasing process $L$ that is c\`adl\`ag with values in $[0,1]$ and zero at zero, i.e., $L_0=0$.

As we will see, the \emph{common noise}, $B^0$, plays a pivotal role: in certain cases it has the power to provoke or prevent a \emph{blow-up}, where a blow-up is defined as a jump discontinuity of  $t \mapsto L_t$. A key property of (\ref{MV}) is that these blow-ups occur endogenously: all the random drivers are continuous, yet jump discontinuities can develop through the positive feedback effect alone. 

The main technical results in this paper are twofold. Firstly, we analyse the blow-up phenomena of \eqref{MV} and derive some simple conditions for if and when such blow-ups occur, with the aim of highlighting the essential role played by the common noise (see Theorem \ref{Prop_BU_CommonNoiseBlowUp}). Secondly, we show that solutions to \eqref{MV} arise as large population limits of a `contagious' particle system, which provides the theoretical justification for the various applications discussed below. We establish this convergence result for more general drift and diffusion coefficients than the case (\ref{MV}), and we allow for a large class of natural initial conditions (see Theorem \ref{Thm:Existence}).

To get a better feel for the workings of (\ref{MV}), consider the conditional law of $X$, with absorption at the origin,  given the common noise $B^0$. This defines a flow of random sub-probability measures $t\mapsto \nu_t$, where \[
\nu_t:=\mathbb{P}(X_t\in\!\cdot\,, \, t<\tau \, | \, B^0), \qquad \text{for all } t\geq0,
\]
thus describing how the `surviving' mass of the system evolves. In particular, we can then write $L_t=1-\nu_t(0,\infty)$, which gives the accumulated loss of mass up to time $t$. From this point of view, a blow-up corresponds to a strictly positive amount of mass being absorbed at the origin in an infinitesimal period of time. A priori, the associated jump of $L$ is not uniquely specified by the equation (\ref{MV}) alone, however, we can fix a canonical choice in terms of the left limits $\nu_{t-}$ of $\nu_{t}$, at any given time $t>0$, according to the minimality constraint \eqref{eq:BU_PhysicalCondition}, which will be introduced in Section 2.

In Lemma \ref{density_process} below, we show that there is a density $V_t$  of $\nu_t$, as defined above, for all positive times $t>0$. Figure \ref{fig:theonlyfigure} illustrates the flow of these densities, for two given realisations of the common noise $B^0$---one of which leads to a blow-up, whose timing and magnitude is specified by the aforementioned minimality constraint.

\begin{rem}[Stochastic evolution equation]\label{rem:evo_eqn} As we  show in Proposition \ref{prop:SPDE_weak_form}, in a generalized sense, the flow of the densities $V_t$ is governed by the stochastic evolution equation
	 \begin{equation*}\label{eq:evo_eqn}
	\left\{
	\begin{aligned}
	dV_t &=  \tfrac{1}{2}\partial^2_{xx}V_tdt - \rho\partial_{x}V_t  dB_{t}^{0} + \alpha\partial_{x}V_t  d{L}^{c}_{t} +  \int_{\mathbb{R}}  [V_{t-}(\cdot+\alpha y)-V_{t-}]  J_L(dt,dy),
	\\[5pt]
	L^c_t &= L_t - \sum_{0<s\leq t}\Delta L_s, \quad L_t=1-\int_0^\infty V_t(x)dx,  \quad \text{and} \quad J_L=\sum_{0 < s<\infty} \delta_{(s,\Delta L_s)},
	\end{aligned}
	\right.
	\end{equation*}
	on the positive half-line with an absorbing boundary at the origin (where we recall that the jump sizes $\Delta L_t$ are specified by \eqref{eq:BU_PhysicalCondition} as mentioned above). Noting that the nonlinearity `$dL_t$' gives the flux of mass across the origin, this evolution equation provides the connection between \eqref{MV} and the mathematical neuroscience literature discussed below. We refer to Section \ref{subsec:SPDE} for further details, but we note here that the equation is quite different from typical stochastic PDEs treated in the literature, as it is nonlocal in space, and, in time, one cannot expect anything like absolute continuity of $L$ or, say, $dL_t=\ell_t dY_t$ for an exogenous driver $Y$.
\end{rem}

\begin{figure}[H]
	\begin{center}
		\hspace{-2cm} \includegraphics[width=0.6\textwidth]{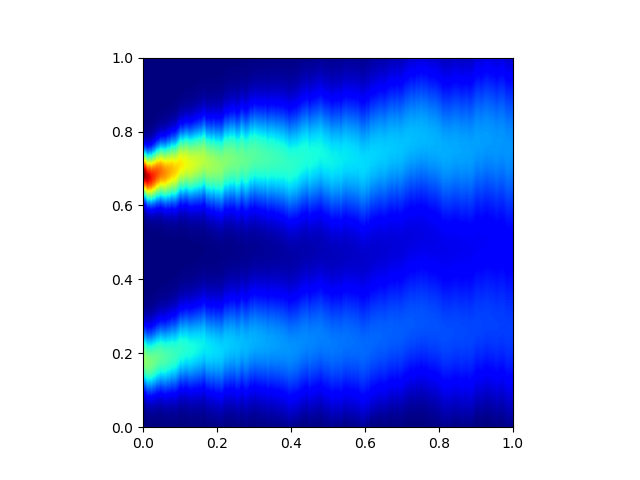} \hspace{-1.8cm} \includegraphics[width=0.6\textwidth]{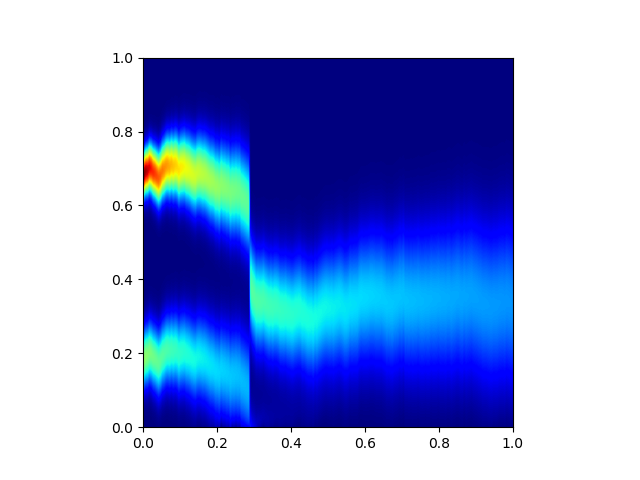} \hspace{-2cm}
		\caption{\label{fig:theonlyfigure} The plots show two different realisations of (\ref{MV}) with $\rho = 0.5$ and the same initial condition. Each pixel represents the value of the density at that space-time point, with space on the vertical and time on the horizontal. On the right, the common noise, $B^0$, decreases sufficiently quickly to cause a blow-up. See Section \ref{sec:numerics} for the numerical algorithm. }
	\end{center}
\end{figure} 

\subsection{Applications and related literature}\label{Sect:lit_review}
Our main interest in (\ref{MV}) stems from its potential to serve as a mean-field model for the interplay between common exposures and contagion in large financial systems, as also discussed in \cite{hambly_sojmark_2017}. In this setting, $X_t \in [0,\infty)$ gives the \emph{distance-to-default}
of a `typical' bank or credit-risky asset, $\rho$ captures the extent of common exposures, and $\alpha > 0$ imposes a contagious feedback effect from defaults. It then follows that a jump discontinuity of $t\mapsto L_t$ corresponds to a `default cascade' in which the feedback from default contagion explodes and a macroscopic proportion of the financial system is lost instantaneously. 

	The above mentioned model without the common noise $(\rho=0)$ motivates the analysis in \cite{hambly_ledger_sojmark_2018} and numerical schemes for this model have been explored in \cite{vadim_reisinger}. A similar financial framework also provides the motivation for \cite{nadtochiy_shkolnikov_2017, nadtochiy_shkolnikov_2018}---the latest paper extending the idiosyncratic ($\rho=0$) model to more generic random drivers and replacing the constant feedback parameter with a graph of interactions for a finite number of groups. At a more heuristic level, an analogous model for the emergence of macroeconomic crises was proposed in \cite{bouchaud}, based on a PDE $(\rho=0)$ version of the stochastic evolution equation in Remark \ref{rem:evo_eqn}, albeit without addressing the possibility of blow-ups and hence not accounting for the jump component.
	
	Concerning the presence of a common noise, more generally, we note that this has recently become a subject of great interest in the context of mean-field games---see e.g.~\cite{carmona2018,carmona2016} for a start. In particular, this also leads to the study of so-called conditional McKean--Vlasov type problems, however, the focus is quite different.

Another important motivation for studying \eqref{MV} comes from mathematical neuroscience, where (\ref{MV}) models the voltage level across a typical neuron in a large mean-field network. This application (with $\rho = 0$) is the focus of \cite{carrillo2011, carrillo2013, carrillo2015, DIRT_AAP, DIRT_SPA}, where the feedback term describes the jump in voltage level experienced when a neighbouring neuron spikes (i.e.,~hits a threshold, thus emitting an action potential and then being reset). The addition of a common noise ($\rho > 0$) is novel and can capture the effect of some systemic random stimulus influencing all neurons simultaneously, for example due to some external influence in the environment. While such common noise mean-field models are discussed in the neuroscience literature, see e.g.~\cite{brunel2000,Moreno-Bote_2010, torcini}, they have not been treated rigorously. 

We note that, in the mathematical neuroscience literature, the focus is not so much on the McKean--Vlasov formulation \eqref{MV}, but rather the corresponding deterministic $(\rho=0)$ or stochastic $(\rho>0)$ Fokker--Planck equation, where  `$dL_t$' gives the infinitesimal change in the proportion of spiking neurons over `$dt$' units of time. Since neurons are reset after spiking, one works with a mass preserving version of the evolution equation from Remark \ref{rem:evo_eqn}, obtained by adding a source term to reintroduce the mass flowing through the origin.

Forgetting about the blow-up component in Remark \ref{rem:evo_eqn}, various formulations of such stochastic Fokker--Planck equations appear in the mathematical neuroscience literature (see e.g.~\cite{brunel2000, brunel-hakim_1998, mattia_2002, torcini}), but so far there have been no attempts at attaching a rigorous meaning to neither the stochastic evolution equation itself nor how it emerges from a finite particle system with common noise. For further details, we refer to the aforementioned references as well as Proposition \ref{prop:SPDE_weak_form} and Remark \ref{rem:neuro_pde} in Section \ref{subsec:SPDE}.

\subsubsection*{Results on well-posedness}

Questions of existence and uniqueness for \eqref{MV} are delicate, not least because of the possibility of blow-ups. While the series of papers \cite{carrillo2011, carrillo2013, carrillo2015} studied the deterministic Fokker--Planck equation, working with a classical notion of solution that may cease to exist in finite time, the papers \cite{DIRT_AAP, DIRT_SPA} were the first to consider the probabilistic formulation \eqref{MV} with $\rho=0$. In particular, for  $\rho=0$, \cite{DIRT_SPA} showed that global solutions can be obtained as limit points of a suitable particle system, using ideas that have inspired our approach in Section 3 and which will be expanded upon there.

The present paper contributes to the literature on well-posedness by introducing a `relaxed' notion of solution for \eqref{MV} with $\rho>0$ that is shown to characterize the limit points of a `contagious' particle system with common noise. This provides the justification for the applications discussed above. We mention here that the `relaxation' of \eqref{MV} concerns the measurability of $L$ (or $\nu$) with respect to the common Brownian motion $B^0$, but further details are left to Section 3.

Prior to the first version of this paper, well-posedness results for \eqref{MV} were only available for the idiosyncratic $(\rho=0)$ version of the problem \cite{DIRT_AAP, DIRT_SPA, hambly_ledger_sojmark_2018, nadtochiy_shkolnikov_2017} with uniqueness only known before the first blow-up \cite{hambly_ledger_sojmark_2018, nadtochiy_shkolnikov_2017}. However, by ruling out blow-ups, the latter becomes global uniqueness if $\alpha>0$ is small enough \cite{DIRT_AAP,hambly_ledger_sojmark_2018}. Concerning $\rho>0$, it has now been shown in the second version of the preprint \cite{nadtochiy_shkolnikov_2018} that global solutions to \eqref{MV} can also be constructed from a generalised Schauder fixed point argument, thus complementing the results in Section 3 of the present paper. 

Additionally, there has been some interesting new developments on the question of uniqueness. For the case $\rho=0$, it was shown in  \cite{LS_unique} that there is local uniqueness when restarting solutions after a blow-up, and, even more recently, a complete well-posedness theory for \eqref{MV} with $\rho=0$ was then developed in the preprint \cite{delarue_sergey_shkolni}, giving global uniqueness under mild assumptions on the initial condition. Finally, \cite{LS_unique} also gives global uniqueness for $\rho>0$ under a smallness condition on the feedback parameter that rules out blow-ups.

\subsection{Overview of the paper}

The rest of the paper is split into three sections. In Section 2, immediately below, we begin our analysis by deriving a number of results about the possibility and timing of blow-ups for solutions to \eqref{MV}. The main results in this regard are collected in Theorem \ref{Prop_BU_CommonNoiseBlowUp}, which serves to illustrate the critical distinction between the idiosyncratic $(\rho=0)$ problem and the common noise $(\rho>0)$ problem. In Section 3, we proceed to formulate the `contagious' particle system that motivates our analysis of \eqref{MV} and analyse its convergence. The main result is Theorem \ref{Thm:Existence}, which shows that the limit points of the empirical measures for this particle system can be characterized as `relaxed' solutions of \eqref{MV} satisfying the minimality constraint \eqref{eq:BU_PhysicalCondition}. Finally, Section 4 provides some results on filtrations that are utilised in Section 3, and we also give a brief outline of the numerical scheme used to generate the simulations in Figure \ref{fig:theonlyfigure}.


\section{On the occurrence of blow-ups} 
\label{Sect_BU}

As emphasised in the introduction, a solution to \eqref{MV} comes with a flow of random sub-probability measures on the positive half-line, namely $t\mapsto \nu_t$, defined by
\[
\nu_t=\mathbb{P}(X_t\in\!\cdot\,, \, t<\tau \, | \, B^0), \qquad \text{for all } t\geq0.
\]
 We will return to the dynamics of this flow later, but for now we simply need it for a concise characterization of the jump sizes at blow-ups, which, due to the c\`adl\`ag nature of the system, depends on the left limits $ \nu_{t-}$, taking the form
 \[
 \nu_{t-}=\mathbb{P}(X_{t-}\in\!\cdot\,, \, t\leq \tau \, | \, B^0),\qquad \text{for all } t\geq0,
 \] with the convention that $\nu_{0-}=\nu_{0}$ (corresponding to $L_{0-}=0$ and $X_{0-}=X_0$).
 
 Given a sample path of $L$ undergoing a jump discontinuity at some time $t$, we can deduce from \eqref{MV} that the jump size $\Delta L_t$ must agree with the loss of mass resulting from a negative shift of the system by the amount $\alpha \Delta L_t$. That is, we have the constraint
\begin{equation}
\label{eq:jump_size}
\Delta L_t={\nu}_{t-}(\hspace{0.4pt}[ 0,\alpha \Delta L_t ]\hspace{0.4pt}),
\end{equation}
almost surely. While this places a restriction on the admissible jump sizes,  it is not sufficient to determine the jumps uniquely, and so it becomes necessary to have a selection principle. Therefore, we enforce the following minimality constraint, stipulating that the jump sizes should satisfy
\begin{equation}
\label{eq:BU_PhysicalCondition}
\Delta L_t = \inf\{  x > 0 : {\nu}_{t-}(\hspace{0.4pt}[0,\alpha x]\hspace{0.4pt}) < x \}, \qquad \text{for all } t\geq 0,
\end{equation}
where the equality holds almost surely (recall the conventions $\nu_{0-}=\nu_0$ and $L_{0-}=0$). As we prove in Proposition \ref{Existence_Prop_MinimalityForCommonNoise}, this constraint amounts to only considering c\`adl\`ag solutions and always selecting the minimal jump size satisfying the requirement \eqref{eq:jump_size}.

In the special case $\rho=0$, the constraints \eqref{eq:jump_size} and \eqref{eq:BU_PhysicalCondition} are of course deterministic, with $\nu_t$ taking the unconditional form
\[
\nu_{t}=\mathbb{P}(X_{t}\in \cdot \, , \, t<\tau) \qquad \text{for all } t\geq0,
\]
and we stress that they correspond precisely to the notion of a `physical solution' first introduced in \cite{DIRT_AAP} (see \cite{hambly_ledger_sojmark_2018} for a treatment closer to the present one, and see also \cite{nadtochiy_shkolnikov_2017} for a slightly different setup). A simple pictorial illustration of these constraints, focused on the special case $\rho=0$, can be found in \cite[Fig.~2.1]{hambly_ledger_sojmark_2018}. Crucially, we mention here that, in Section 3, the solutions arising as limit points of the motivating particle system will be shown to satisfy the minimality constraint \eqref{eq:BU_PhysicalCondition}.

\subsection{Idiosyncratic vis-\`a-vis common noise}

The purpose of Section 2 is to present a clear and simple juxtaposition of the phenomenon of blow-ups for \eqref{MV} with and without the common noise. In particular, we will illustrate how the addition of a common noise can determine whether or not the model undergoes a blow-up. We stress that the methods employed are not sharp, as they are only intended to guarantee the existence of non-trivial blow-up probabilities, rather than defining precise conditions under which they occur. Even in the idiosyncratic $(\rho = 0)$ setting, one should not expect to find a neat closed-form condition on the choice of initial law and feedback parameter in a way that determines if and when a blow-up occurs, and so we shall not pursue this here.

Nevertheless, we note that the proofs rely on constructive arguments (in Sections \ref{Subsec_BU_idio} and \ref{Sect_BU_Common}), thereby outlining specific events on which a blow-up does or does not occur for certain initial conditions. When constructing events on which a blow-up is excluded, the key ingredient will be the minimality constraint (\ref{eq:BU_PhysicalCondition}) introduced above. Conversely, in arguments about enforcing blow-ups, one of our main tools will be a variant of the moment method from \cite[Thm.~1.1]{hambly_ledger_sojmark_2018}.

The next theorem is the main result of this section and it serves to illustrate the decisive role played by the common noise in relation to blow-ups. As in the original contribution \cite{DIRT_AAP} for the idiosyncratic problem \eqref{MV} with $\rho=0$, we focus on a fixed starting point $X_0=x_0$. That is, the system is started from a Dirac mass $\nu_0=\delta_{x_0}$ at some $x_0>0$. In this setting, it follows from \cite[Thm.~2.3 \& 2.4]{DIRT_AAP} that, if $\alpha>0$ is small enough as a function of $x_0$, then \eqref{MV} with $\rho=0$ admits a unique (continuously differentiable) solution $L$ with no blow-ups on any given time interval. With the common noise, however, the situation is markedly different.

\begin{thm}[Blow-ups]
	\label{Prop_BU_CommonNoiseBlowUp} Consider the McKean--Vlasov problem \eqref{MV} with a fixed starting point $X_0=x_0$. We have the following dichotomy for the phenomena of blow-ups, depending on the presence $(\rho>0)$ or absence $(\rho=0)$ of the common noise.
	
	\begin{itemize}
		\item[(i)] Idiosyncratic  model: given any $\alpha>0$, $L$ has a blow-up for starting points $x_0\in(0,\infty)$ close enough to the origin, while there are no blow-ups for starting points $x_0\in(0,\infty)$ far enough from the origin.
		
		\item[(ii)] Common noise model: given any $\alpha>0$, blow-ups are random events with
		\[
		0< \mathbb{P}(L \emph{\textrm{ has a blow-up}}) < 1,
		\]
		for any starting point $x_0\in(0,\infty)$.
	\end{itemize}
\end{thm}

The proof of this theorem is the subject of the next two subsections, where we will shed further light on the blow-up behaviour for the idiosyncratic $(\rho=0)$ and common noise $(\rho>0)$ versions of \eqref{MV}. The first part of Theorem \ref{Prop_BU_CommonNoiseBlowUp} follows from Proposition \ref{Prop_BU_LinearBounds} in Section \ref{Subsec_BU_idio}, while the second and main part of the result
is a consequence of Proposition \ref{Thm_enforce_blow-up} in Section \ref{Sect_BU_Common}. The dichotomy pointed out by Theorem \ref{Prop_BU_CommonNoiseBlowUp} is illustrated in Figures \ref{fig:the2ndfigure} and \ref{fig:the3rdfigure} in Section \ref{Sect_BU_Common}. These figures give a simple pictorial account of how the common noise may provoke and/or prevent a blow-up, as compared to how the  idiosyncratic system evolves.
 
 Before proceeding with the analysis of blow-ups, we make the following general observation which was alluded to already in the introduction:~for all strictly positive times $t>0$, the flow $t\mapsto \nu_t$ is given by a flow of bounded densities $t\mapsto V_t$, regardless of the nature of the initial condition $\nu_0$. This is analogous to the situation for $\rho=0$ \cite[Prop.~2.1]{hambly_ledger_sojmark_2018}, except of course for the randomness of the flow in the $\rho>0$ setting.

 \begin{lem}[Existence of densities]\label{density_process} Suppose there is a solution to \eqref{MV}. For every $t>0$, the associated random sub-probability measure $\nu_t$ has a (random) density $V_t\in L^{\infty}(0,\infty)$.
 \end{lem}
 \begin{proof}
 	Fix any $t>0$ and let $p_t$ denote the density of the  Brownian motion $B_t$. Omitting the absorption at the origin, and using Tonelli's theorem to change the order of integration, we have
 	\begin{align*}
 	\nu_t(A) & \leq \mathbb{P}\bigl(X_0 + \sqrt{1-\rho^2} B_t + \rho B^0 -\alpha L_t \in A \mid B^0\bigr) \\
 	&= \int_A \int^\infty_0 p_{(1-\rho^2)t}(x - x_0 - \rho B^0 + \alpha L_t) \nu_0(dx_0) dx,
 	\end{align*}
 	for any Borel set $A\in\mathcal{B}(0,\infty)$, where we have used the independence of $X_0$, $B$, and $B^0$, as well as the $B^0$-measurability of $L$. Consequently,
 	\[
 	\nu_{t}(A)\leq(2\pi (1-\rho^2)t)^{-1/2}\text{Leb}(A),
 	\]
 	for all  Borel set $A\in\mathcal{B}(0,\infty)$, so the result follows from the Radon--Nikodym theorem.
 \end{proof}

As in the above, we will always use $\nu_0$ to denote the law of the initial random variable $X_0$ for \eqref{MV}, where it is understood that $\nu_0$ is a Borel probability measure on $\mathcal{B}(0,\infty)$. Furthermore, throughout the rest of the paper, we will say that $\nu_0$ is supported on a given Borel set $A\in \mathcal{B}(0,\infty)$ if $\nu_0$ assigns zero mass to the complement of $A$.

\subsection{The idiosyncratic model with  $\rho = 0$}\label{Subsec_BU_idio}

Without the common noise, the flow $t\mapsto \nu_t$ and the loss $t \mapsto L_t$  are deterministic. Hence the blow-ups for (\ref{MV}) with $\rho=0$ are completely deterministic events. Since we know from above that each $\nu_t$ has a density, a simple observation concerning blow-ups is the following. If, for a given $t_0>0$, we have $V_{t_0-}(x)<\alpha^{-1}$ in a right-neighbourhood of the origin, then $\Delta L_{t_0} = 0$ is fixed by (\ref{eq:BU_PhysicalCondition}), and so a blow-up cannot occur at that time. In the opposite direction, we can similarly observe that if, in stead, we have $V_{t_0-}(x)\geq\alpha^{-1}$ in a right-neighbourhood of the origin, then a blow-up must occur at time $t_0$.

Furthermore, from the proof of Lemma \ref{density_process} above, we can see that
\begin{equation}
\label{eq:BU_DensityControl}
\Vert V_{t} \Vert_{\infty} \leq \min\{ \Vert V_{0} \Vert_{\infty} , (2\pi t)^{-1/2} \},
	\qquad \textrm{for every } t \geq 0, 
\end{equation}
with $\Vert V_{0} \Vert_{\infty}$ interpreted as $+\infty$ if $\nu_0$ does not have a density. Thus, referring again to $\eqref{eq:BU_PhysicalCondition}$, we obtain the following result as an immediate consequence of the previous proof of Lemma \ref{density_process}.

\begin{cor}[Initial curbs on blow-up]
\label{Cor_BU_NoBlowUpEasy} 
For any initial condition, $\nu_0$, no blow-up can occur in \eqref{MV} strictly after time $\alpha^2 / 2\pi$. Furthermore, if $\nu_0$ has a density, $V_0$, satisfying $\Vert V_0 \Vert_\infty < \alpha^{-1}$, then a blow-up never occurs.
\end{cor}

By working harder, we can extend this approach to show that if the initial density is supported far enough from the boundary, then there is insufficient time for the mass to reach the boundary before the time decay in (\ref{eq:BU_DensityControl}) prevents a blow-up. By exploiting this, the next result gives a linear range in $\alpha$, which is clearly  not sharp, but it is nonetheless illustrative and useful for what follows.

\begin{prop}[Linear spatial bounds for blow-up]
\label{Prop_BU_LinearBounds}
If the initial condition $\nu_0$ is supported on $(0,\tfrac{1}{2}\alpha)$ then a blow-up must occur in \eqref{MV}. If $\nu_0$ is supported on $(\tfrac{5}{4}\alpha, +\infty)$ then a blow-up will never occur. 
\end{prop}

\begin{proof}
If $\nu_0$ is supported on $(0,\frac{1}{2}\alpha)$, then we have $\mathbb{E}X_0<\frac{1}{2}\alpha$, so the first part is immediate from \cite[Thm.~1.1]{hambly_ledger_sojmark_2018}. For the second part, notice that, by disregarding the absorption at the origin, we obtain the upper bound
\begin{align*}
\nu_{t}(\hspace{0.4pt}[0,\alpha x] \hspace{0.4pt} ) 
	&\leq \mathbb{P}(X_0 + B_t - \alpha L_t \in [0,\alpha x])\nonumber\\
	&= \int^\infty_0 \int_0^{\alpha x} \frac{1}{\sqrt{2\pi t}} \exp \Big\{ - \frac{(z-x_0 + \alpha L_t)^2}{2t} \Big\} dz \nu_0(dx_0) .
\end{align*}
Suppose $\nu_0$ is supported on $(\alpha b, +\infty)$ for some $b > 1$. Since $L_t \leq 1$, the previous bound gives
\begin{align*}
\nu_{t}( \hspace{0.4pt}[ 0,\alpha x ] \hspace{0.4pt} ) 
	&\leq \int^\infty_0 \int_0^{\alpha x} \frac{1}{\sqrt{2\pi t}} \exp \Big\{ - \frac{(z -\alpha (b - 1))^2}{2t} \Big\} dz \nu_0(dx_0) \\
	&\leq  \frac{\alpha x}{\sqrt{2\pi t}} \exp \Big\{ - \frac{\alpha^2(b - 1 - x)^2}{2t} \Big\},
\end{align*}
for every $x <b-1$. Noting also that $t \mapsto (2\pi t)^{-1/2} e^{-c^2 / 2t}$ has maximum value of $(2\pi e c^2)^{-1/2}$, we can thus deduce that
\[
	\nu_t (\hspace{0.4pt}[ 0, \alpha x] \hspace{0.4pt}) \leq \frac{\alpha x}{\sqrt{2\pi e \alpha^2 (b-1-x)^2}} < x
\]
for all $x <b-1 - (2\pi e)^{-1/2}$, provided $b>1+(2\pi e)^{-1/2}$. Since $2\pi e > 16$, taking $b = \tfrac{5}{4}$ is sufficient, and so the proof is complete.
\end{proof}

\subsubsection{Transformed loss processes}

A natural question, arising for example from \cite{nadtochiy_shkolnikov_2017}, is what happens when we replace the linear loss term in (\ref{MV}, $\rho=0$) with a general function of the loss process. This leads to the similar problem
\begin{align}
\label{eq:BU_TransformedLoss}
\begin{cases}
 X_t = X_0 + B_t - \alpha f( L_t ) \\
L_t = \mathbb{P}(\tau \leq t)\\
\tau = \inf\{ t \geq 0 : X_t \leq 0 \} ,
\end{cases}
\end{align}
for a given function $f : [0,1) \to \mathbb{R}$. The particular example $f(x) = -\log(1 - x)$ is the focus of \cite{nadtochiy_shkolnikov_2017}. To avoid mass escaping to infinity, we can restrict to $f$ bounded below---then we know that $L_t \to 1$ as $t \to \infty$. In particular, we note that this leads to a useful simplification of the right-hand side in Proposition \ref{BlowUp_Prop_BlowUpForTransformed} below.

While we focus on $f(x)=x$ in this as paper, as in the motivating papers \cite{DIRT_AAP,DIRT_SPA,hambly_ledger_sojmark_2018,hambly_sojmark_2017}, we note that our results extend immediately to (\ref{eq:BU_TransformedLoss}) for any Lipschitz continuous $f$. Also, it takes only minor adaptations (see e.g.~\cite[Sec.~7]{nadtochiy_shkolnikov_2017} for the kind of changes that are needed) to allow for functions whose Lipschitz constant explodes when $x\uparrow1$ such as $f(x)=-\log(1-x)$ mentioned previously or e.g.~$f(x)=(1-x)^{-1}$.

For a discontinuous $f$, one would have to be more careful. In particular, a first simple observation is that if $f$ has a positive jump at $x_0 < 1$, then $L$ will have a blow-up at $t_0 = L^{-1}(x_0)$, however, if $f$ has a negative jump, then $L$ might still be continuous.

In any case, assuming there is a solution to \eqref{eq:BU_TransformedLoss} for a given function $f$, the following result presents the natural extension of \cite[Thm.~1.1]{hambly_ledger_sojmark_2018} to this problem.
\begin{prop}[Blow-up for transformed loss]
\label{BlowUp_Prop_BlowUpForTransformed}
Suppose $L$ gives a solution to (\ref{eq:BU_TransformedLoss}) for some function $f : [0,1) \to \mathbb{R}$. If we have
\[
m_0 < \alpha\Big( \int^{L_{t}}_0 f(x)dx + (1 - L_{t})f(L_{t}) \Big),
\]
at some time $t>0$, where $m_0:= \mathbb{E}X_0 $, then a blow-up must occur before or at this time $t$. 
\end{prop}

\begin{proof}
For a contradiction, suppose that $X$ is continuous up to and including time $t$. Then, by stopping $X$ at the first time it reaches zero, we get
\[
0 \leq X_{t \wedge \tau} = X_0 + B_{t \wedge \tau} - \alpha f(L_{t \wedge \tau}),
\] 
since $L$ must be continuous up to this time, by our assumption. Taking expectations, and observing that $s\mapsto L_s$ gives the law of the stopping time $\tau$, we thus have
\[
0 \leq m_0 - \alpha \int^{\infty}_0 f(L_{t \wedge s}) dL_s = m_0 - \alpha \Big( \int^{t}_0 f(L_s) dL_s + (1-L_{t})f(L_{t}) \Big).
\]
The proof is now complete by noting that $L$ is of finite variation, since it is increasing, and continuous up to and including time $t$, by assumption, so the integral  on the right-hand side can be re-written as in the statement of the result.  
\end{proof}

By taking $f(x) = x$ and sending $t \to \infty$ in Proposition \ref{BlowUp_Prop_BlowUpForTransformed}, we recover \cite[Thm.~1.1]{hambly_ledger_sojmark_2018}. Notice that, as soon as the integral of $f$ over $(0,1)$ is positive, then for any initial condition we can take $\alpha>0 $ large enough to cause a blow-up. In addition, if $f$ is chosen so that its integral over $(0,1)$ is positively infinite (for example, $f(x) = (1-x)^{-1}$), then a blow-up is guaranteed for any initial condition and any feedback parameter $\alpha > 0$. Note, however, that the case $f(x) = -\log(1-x)$ from \cite{nadtochiy_shkolnikov_2017} has finite integral, so, in this case, Proposition \ref{BlowUp_Prop_BlowUpForTransformed} does not determine whether or not a blow-up must occur for any value of the (strictly positive) feedback parameter.

\subsection{The common noise model with $\rho > 0$}
\label{Sect_BU_Common}

For the idiosyncratic ($\rho=0$) model, the only parameters controlling blow-ups are the feedback strength, $\alpha$, and initial condition, $\nu_0$. In this regard, an important novelty in the common noise model ($\rho>0$) is that there are choices of initial condition for which the realisations of the common noise determine whether a blow-up occurs or not. Furthermore, as the solution $L$ is no longer deterministic, the occurrence of a blow-up is now a random event.

In the case of a Dirac initial condition $\nu_0=\delta_{x_0}$, Theorem \ref{Prop_BU_CommonNoiseBlowUp} says that, if the support $x_0>0$ is far enough from the origin, then we do \emph{not} have a blow-up in the idiosyncratic model, while there is always a non-zero probability of blow-up with common noise ($\rho > 0$).

The proof of this uses an explicit construction to show that a blow-up is forced if the sample path of $B^0$ is sufficiently negative so as to quickly transport the initial mass towards the boundary without loosing too much mass along the way.

With the same amount of work, the arguments can be phrased in a slightly more general way that also applies to a uniform distribution concentrated near a point $x_0$ away from the origin. The latter case is well suited for illustrative purposes, and so we rely on it for Examples  \ref{Example1} and \ref{Example2}  below. Both examples are accompanied by a simplified pictorial account in Figures  \ref{fig:the2ndfigure} and \ref{fig:the3rdfigure}, respectively. In addition to illustrating the interplay between the common noise and blow-ups, these figures also give an intuitive idea of the mechanisms underlying the proofs of Propositions \ref{Thm_avoid_blow-up} and  \ref{prop:Dirac_blow-up}. However, we stress that these figures are only meant as simplified `cartoons' (in particular, the axes are not in the same scale, and the shape and size of the densities are not precise).

\begin{figure}
	\begin{center}
		\includegraphics[width=\textwidth]{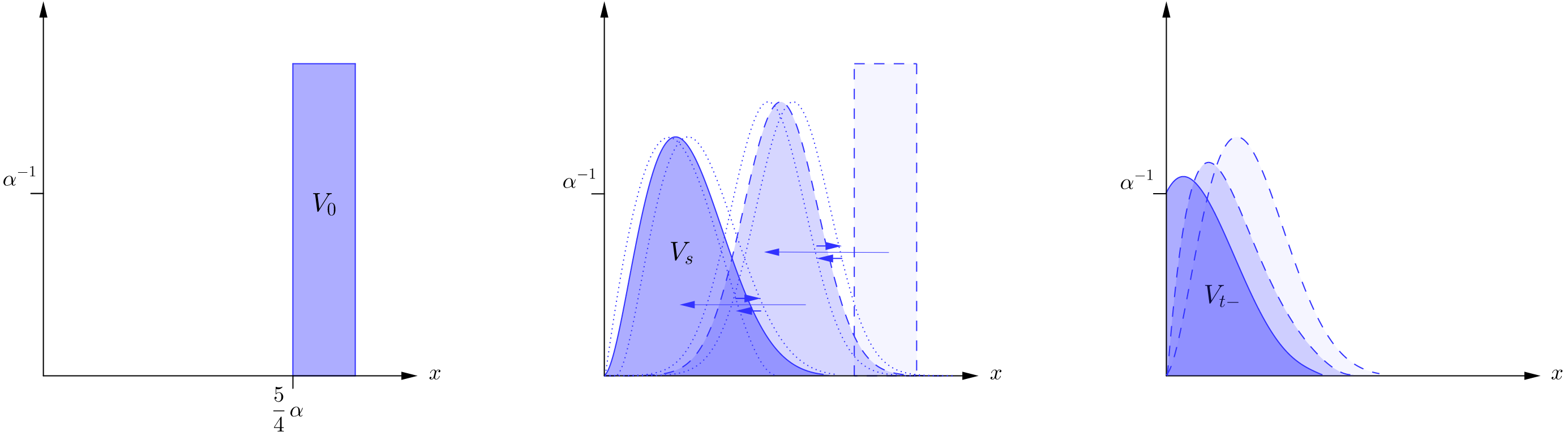}\caption{\label{fig:the2ndfigure} The figure displays the emergence of a blow-up. In the middle picture, the overall effect of the common noise is a prompt transportation of mass towards the origin, moving the system from $V_0$ to $V_s$. Since $V_s$ is concentrated near the origin, an adaptation of \cite[Thm.~1.1]{hambly_ledger_sojmark_2018} shows that there must be a blow-up, provided too much mass does not start escaping in the other direction. In this regard, the rightmost picture illustrates the nonlinear feedback becoming so strong that, after taking a limit as $s\uparrow t$, the resulting left-limit density $V_{t-}$ lies above $\alpha^{-1}$ near zero and hence $\Delta L_t \neq0$ in line with \eqref{eq:BU_PhysicalCondition}.}
	\end{center}
\end{figure} 

\begin{example}[Forcing blow-up]\label{Example1} Let the initial condition $\nu_0$ be a uniform distribution with density $V_0:=\delta^{-1}\mathbf{1}_{(c,c+\delta)}$ for constants $c\geq \frac{5}{4}\alpha$ and $0<\delta<\alpha$, as in the leftmost picture of Figure \ref{fig:the2ndfigure}. For such initial conditions, Proposition \ref{Prop_BU_LinearBounds} guarantees there is no blow-up in the idiosyncratic model. On the other hand, Proposition \ref{prop:Dirac_blow-up} below shows that there is a non-zero probability of blow-up in the common noise model. This happens when the common Brownian motion transports mass quickly towards the origin, thus creating a strong enough concentration of mass near zero so that a blow-up is forced to occur, essentially by comparison with \cite[Thm~1.1]{hambly_ledger_sojmark_2018} for the idiosyncratic model.
\end{example}

The previous example dealt with how the common noise can provoke a blow-up. In the next example, we consider the converse situation, illustrating how the common noise may prevent blow-ups, while starting from an initial condition for which the idiosyncratic model blows up.

\begin{example}[Averting blow-up]\label{Example2} This time, we let the initial condition $\nu_0$ be a uniform distribution with density $V_0:=\delta^{-1}\mathbf{1}_{(c-\delta,c)}$ for positive constants $c< \alpha$ and $\delta < c$ such that $c < \frac{1}{2}(\alpha+\delta)$, as in the leftmost picture of Figure \ref{fig:the3rdfigure}. Since $\int \!xd\nu_0(x)=c-\frac{1}{2}\delta<\frac{1}{2}\alpha$, the idiosyncratic model must have a blow-up for such initial conditions by \cite[Thm.~1.1]{hambly_ledger_sojmark_2018}. In contrast, it follows from Proposition \ref{Thm_avoid_blow-up} below that the common noise model has a strictly positive probability of never blowing up. Briefly, the common noise can keep the mass of the system sufficiently away from the origin, so that the diffusive effect can do its work and eventually rule out a blow-up from ever emerging, since the mass has become too dispersed. This is illustrated in Figure \ref{fig:the3rdfigure}.
\end{example}

\begin{figure}
	\begin{center}
		\includegraphics[width=\textwidth]{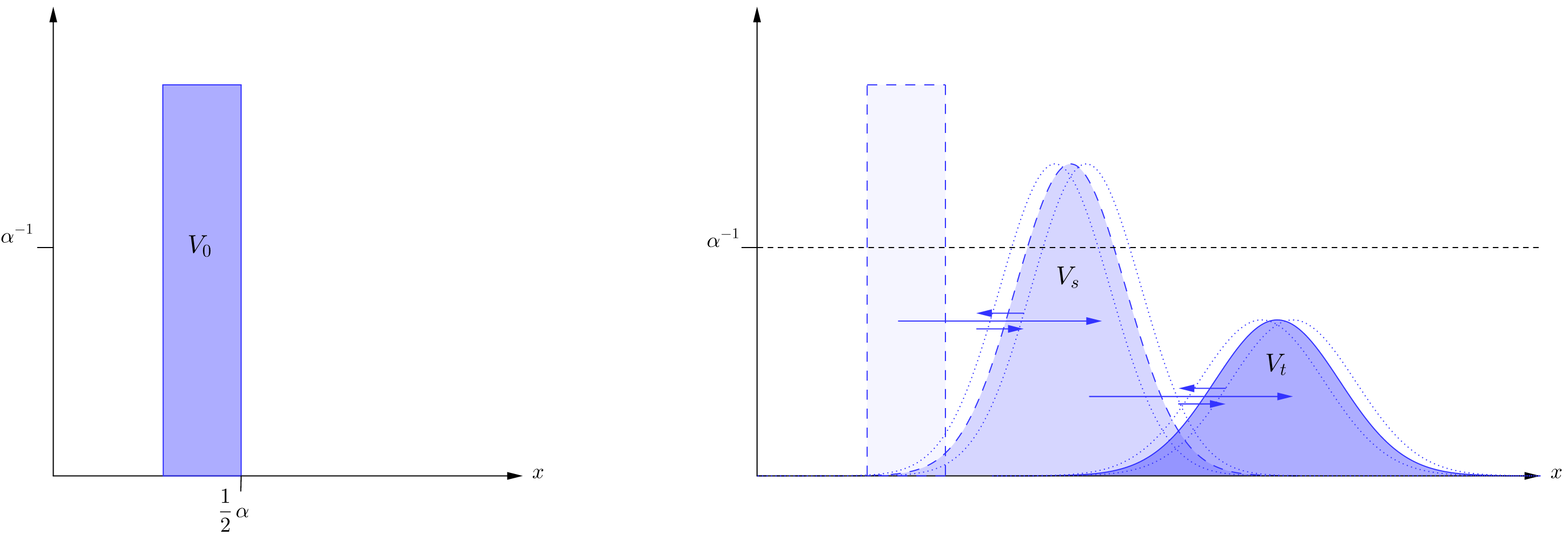}\caption{\label{fig:the3rdfigure} The figure displays an event where a blow-up is avoided, because the common noise keeps the mass away from the origin for an adequate amount of time. Specifically,  the picture on the right illustrates how the common noise can allow sufficient elbow room for the diffusive effect to spread out the mass, so that the density $V_t$ ends up lying everywhere strictly below $\alpha^{-1}$. From here, a blow-up cannot occur at any later times, for this realisation of the common noise, since we can no longer get into a situation like the rightmost picture of Figure \ref{fig:the2ndfigure}.}
	\end{center}
\end{figure}

 In the arguments that follow, we are interested in bounds on the loss of mass that are uniform over a given family of realisations of the common noise. This begins with a simple comparison argument.

\begin{lem}[No-crossing lemma]
	\label{Lem_BU_NoCrossing}
	Let $\nu_0$ be a probability measure supported on $[x_\star, \infty)$, for some $x_\star > 0$, and let $\tilde{f}, \bar{f} : [0,\infty) \to \mathbb{R}$ be two continuous deterministic functions satisfying $\tilde{f}(0), \bar{f}(0) > -x_\star$. Let $\tilde{L}$ and $\bar{L}$ be solutions to 
	\[
	\begin{cases}
	\tilde{X}_t = Z_0 + \sqrt{1-\rho^2} B_t + \tilde{f}(t) - \alpha \tilde{L}_t \\
	\tau = \inf\{ t \geq 0 : \tilde{X}_t \leq 0 \}  \\
	\tilde{L}_t = \mathbb{P}(\tilde{\tau} \leq t ),
	\end{cases}
	\quad
	\begin{cases}
	\bar{X}_t = Z_0 + \sqrt{1-\rho^2} B_t + \bar{f}(t) - \alpha \bar{L}_t \\
	\bar{\tau} = \inf\{ t \geq 0 : \bar{X}_t \leq 0 \}  \\
	\bar{L}_t = \mathbb{P}(\bar{\tau} \leq t),
	\end{cases}
	\]
	respectively, where $Z_0$ is distributed according to $\nu_0$. If $\tilde{f}(t) > \bar{f}(t)$ for all $t\in[0,t_0)$ and $\tilde{L}$ is continuous on $[0,t_0)$, then $\tilde{L}_t \leq \bar{L}_t$ for all $t \in [0,t_0)$.  If $\tilde{L}$ is also continuous at $t_0$, then $\tilde{L}_t \leq \bar{L}_t$ for all $t \in [0,t_0]$.
\end{lem}

\begin{proof} As in the statement, we have $\tilde{f}(t) > \bar{f}(t)$ for all $t\in[0,t_0)$, where the time $t_0>0$ is given, and we assume $\tilde{L}$ is continuous on $[0,t_0)$. Observe that if $\tilde{L}$ is continuous on $[0,t]$, for some $t>0$, and $\tilde{X}_s \geq \bar{X}_s$ for all $s < t$, then 
\begin{equation}
\label{eq:BU_NoCrossing}
	\tilde{L}_t = \tilde{L}_{t-} = \mathbb{P}( \, \inf_{s < t} \tilde{X}_s < 0 \, ) \leq \mathbb{P}( \, \inf_{s < t} \bar{X}_s < 0 \, ) = \bar{L}_{t-} \leq \bar{L}_{t}.
\end{equation}
	In particular, the claim of the lemma follows if only we can show that $\tilde{X}_t \geq \bar{X}_t$ for all $t \in [0,t_0)$. To this end, we let $t_\star$ be the first time $t$ such that $\tilde{X}_t \leq \bar{X}_t$. By definition of a solution, we have $ \tilde{L}_0=\bar{L}_0=0$, so the assumption $\tilde{f}(0) > \bar{f}(0)$ gives $\tilde{X}_0 > \bar{X}_0$. Combining this with the right-continuity of solutions, it follows that $t_\star > 0$.
	
	Now suppose, for a contradiction, that we have $t_\star < t_0$. Then $\tilde{L}$ is continuous on $[0,t_\star]$, by the assumption that it is continuous on $[0,t_0)$, so we also have continuity of  $\tilde{X}$ on $[0,t_\star]$. Noting that $\bar{X}$ only has downwards jumps ($\bar{L}$ has only upwards jumps), we must therefore have $\tilde{X}_{t_\star} = \bar{X}_{t_\star}$, and this forces
	\[
	\alpha(\bar{L}_{t_\star} - \tilde{L}_{t_\star}) = \bar{f}(t_\star) -  \tilde{f}(t_\star) < 0.
	\]
But this is a contradiction, because the left-hand side is non-negative by (\ref{eq:BU_NoCrossing}), since $L$ is continuous on $[0,t_\star]$ and $\tilde{X}_{t} > \bar{X}_t$ for $t < t_\star$, by construction.  This completes the proof.
\end{proof}

\subsubsection{Averting blow-ups}

Keeping in mind the case of a Dirac mass and the example in Figure \ref{fig:the3rdfigure}, we now turn to a precise construction of events on which blow-ups are excluded for all feedback parameters $\alpha>0$, when starting the system from a general initial condition $\nu_0$ supported away from the origin.

\begin{prop}[Non-certain blow-up]
	\label{Thm_avoid_blow-up} Let the initial condition $\nu_0$ be supported on $(x_\star,\infty)$, for some $x_\star>0$. Then the conditional McKean--Vlasov problem \eqref{MV} with common noise $(\rho>0)$ satisfies
	\[
	\mathbb{P}(L \emph{\textrm{ has a blow-up}}) < 1
	\]
	for all feedback parameters $\alpha > 0$.
\end{prop}

\begin{proof}
	Let $x_\star>0$ be as in the statement of the proposition, and fix a $\delta>0$ such that $\delta\leq x_\star/4$. Then define $\bar{L}$ to be a solution to the idiosyncratic problem
	\[
\begin{cases}
\bar{X}_t = (X_0 - \delta) + \sqrt{1-\rho^2} B_t  - \alpha \bar{L}_t \\
\bar{\tau} = \inf\{ t \geq 0 : \bar{X}_t \leq 0 \}  \\
\bar{L}_t = \mathbb{P}(\bar{\tau} \leq t ),
\end{cases}
\]
noting that $X_0-\delta \geq x_\star-\delta >0$. By the right-continuity of $\bar{L}$ as well as $\bar{L}_0=0$, by our notion of solution, we then have $\bar{L}_t \to 0$ as $t \downarrow 0$. In particular, we can take $t_0 > 0$ small enough so that $\bar{L}_{t_0} \leq \tfrac{1}{4}x_\star \alpha^{-1}$. Since $L$ is increasing, we have  $x_\star-\alpha \bar{L}_t \geq \tfrac{3}{4}x_\star$ for $t\in[0,t_0]$, so an adaptation of the argument in Proposition \ref{Prop_BU_LinearBounds} shows that we can take $t_0>0$ small enough such that $\bar{L}$ is continuous on $[0,t_0]$. Indeed, this amounts precisely to the bound \eqref{eq:bound_bo_BU} below for $\rho=0$.

For a general $\rho\in(0,1)$, we can exploit that $B$ is independent of $B^0$ and $L$, with $L$ being $B^0$-measurable, and this way we obtain the bound
	\begin{align}
	\label{eq:BU_CommonNoiseBU_Escape}
	\nu_{t-}( \hspace{0.4pt} [ 0,\alpha x] \hspace{0.4pt} ) 
	&\leq \mathbb{P}(X_{t-} \in [ 0, \alpha x]\mid B^0) \nonumber \\
	&= \mathbb{P}(X_0 + \sqrt{1-\rho^2}B_t + \rho B^0_t - \alpha L_{t-} \in[ 0, \alpha x] \mid B^0 ) \nonumber \\
	&= \int^\infty_0 \int^{\alpha x}_0 \frac{1}{\sqrt{2\pi (1-\rho^2)t }} \exp\Big\{ - \frac{(z - x_0 - \rho B^0_t + \alpha L_{t-})^2}{2(1-\rho^2)t} \Big\}  dz d\nu_0(x_0),
	\end{align}
	for any $t>0$, where we have simply disregarded the absorption at the origin.
	
	For the rest of the proof, we restrict attention to the event $(\rho B^0_t)_{t\in[0,t_0]}\in \mathcal{A}_{t_0,\delta,m}$, where the latter is a family of paths defined by
	\[
	\mathcal{A}_{t_0,\delta,m}:= \bigl\{  f:[0,t_0]\rightarrow \mathbb{R} \;\text{s.t.}\; |f(t)-mt| \leq \delta \; \text{for all}\;t\in[0,t_0]  \bigr\},
	\]
	for a positive constant $m>0$ to be determined later. On this event, for any given $m>0$, we have 
	$-\rho B_t^0 \leq \delta \leq \frac{1}{4}x_\star$ for all $t\in[0,t_0]$. In order to also control the loss of mass, we stop $L$ before it exceeds $\frac{1}{4} x_\star \alpha^{-1}$, which amounts to looking at the stopping time
	\[
	\tau_0:=\inf \{ t\geq 0 : L_{t} >  x_\star \alpha^{-1} /4  \}.
	\]
	By the right-continuity of $L$ with $L_0=0$, we have $\tau_0>0$, and we also stress that $\tau_0$ of course inherits $B^0$-measurability from $L$.
	
	Since $\nu_0$ is supported on $(x_\star,\infty)$, we can deduce that, on the event $(\rho B_t^0)_{t\in[0,t_0]} \in \mathcal{A}_{t_0, m ,\delta}$, the bound \eqref{eq:BU_CommonNoiseBU_Escape} yields
	\begin{align*}
	\nu_{t\land \tau_0-}( \hspace{0.4pt} [ 0,\alpha x] \hspace{0.4pt} )
	& \leq \int^\infty_0 \int^{\alpha x}_0 \frac{1}{\sqrt{2\pi (1-\rho^2)(t\land \tau_0) }} \exp\Big\{ - \frac{(x_0 -3x_\star/4)^2}{2(1-\rho^2)(t\land \tau_0)} \Big\}  dz d\nu_0(x_0) \\
	& \leq \frac{\alpha }{ \sqrt{2\pi (1-\rho^2)(t\land \tau_0) } } \exp\Big\{ - \frac{x_\star^2}{32(1-\rho^2)(t\land \tau_0)} \Big\} \cdot x,
	\end{align*}
	for every $x<\frac{1}{4}x_\star \alpha^{-1}$. By lowering $t_0>0$ if necessary, it follows that there is a constant $C_\star<1$ such that
	\begin{equation}\label{eq:bound_bo_BU}
	\nu_{t\land \tau_0-}( \hspace{0.4pt} [ 0,\alpha x] \hspace{0.4pt} ) \leq C_\star x, \qquad \text{for all }t\leq t_0.
	\end{equation}
	
	Based on the above, we now argue that $L$ cannot have a blow-up on the full interval $[0,t_0]$, when restricting to the event $(\rho B_t^0)_{t\in[0,t_0]} \in \mathcal{A}_{t_0, m ,\delta}$. First of all, restricting to this event, the bound \eqref{eq:bound_bo_BU} and the minimality constraint \eqref{eq:BU_PhysicalCondition} already ensures that $L$ is continuous on $[0,\tau_0\land t_0]$ with $\tau_0\land t_0>0$. Now, if $L$ has a blow-up on $[0,t_0]$, then there must be a realisation of $(\rho B_t^0)_{t\in[0,t_0]} \in \mathcal{A}_{t_0, m ,\delta}$ and some time $s_0\in(0,t_0)$ such that $L$ is continuous on $[0,s_0)$ while failing to be continuous on any right-neighbourhood $[s_0,s)$ with $s>s_0$. Noting that
	\[
	\rho B^0_t \geq mt_0 -\delta > -\delta, \qquad \text{for all } t\in(0,t_0],
	\]
	on the event $(\rho B_t^0)_{t\in[0,t_0]} \in \mathcal{A}_{t_0, m ,\delta}$, so the conditions of the no-crossing lemma are satisfied. Therefore, the continuity of $L$ on $[0,s_0)\subseteq [0,t_0)$ gives
	\begin{equation}\label{eq:L-s0-bound}
	L_{s_0-}\leq \bar{L}_{t_0}\leq \frac{1}{4}x_\star \alpha^{-1}
	\end{equation}
	precisely by the no-crossing lemma (Lemma \ref{Lem_BU_NoCrossing}), and thus the bound \eqref{eq:bound_bo_BU} now reads as
	\[
	\nu_{s_0-}( \hspace{0.4pt} [ 0,\alpha x] \hspace{0.4pt} ) \leq C_\star x \qquad \text{for all } x \text{ sufficiently close to zero}.
	\]
	This, in turn, implies $\Delta L_{s_0}=0$, by the minimality constraint \eqref{eq:BU_PhysicalCondition}, and, in particular, we then get $L_{s_0}=L_{s_0-}\leq \frac{1}{4}x_\star \alpha^{-1}$ from \eqref{eq:L-s0-bound}. By right-continuity, it follows from here that, for any $\varepsilon>0$, we have  $L_{s}\leq \frac{1}{4}x_\star \alpha^{-1}+\varepsilon$ for $s>s_0$ close enough to $s_0$. Thus, we can conclude that there is a bound of the form $\nu_{s_0-}( \hspace{0.4pt} [ 0,\alpha x] \hspace{0.4pt} ) \leq C x$ with $C<1$, which continues to hold (for $x$ sufficiently close to zero) on a small enough right-neighbourhood of $s_0$, only we may need to take $C>C_\star$ (but of course still strictly less than one). This means that $L$ must remain continuous on $[s_0,s)$ for an $s>s_0$ close enough to $s_0$, and so it is indeed the case that, on the event $(\rho B_t^0)_{t\in[0,t_0]} \in \mathcal{A}_{t_0, m ,\delta}$, the loss $L$ does not have a blow-up on $[0,t_0]$.

	Given the above, we fix a later time $t_1>t_0$ by setting \[
	t_1:=\alpha^2/2\pi(1-\rho^2) + t_0.
	\]
	Since $t_0>0$, it follows from the bound in the proof of Lemma \ref{density_process} that
	\[
\nu_{t}([0,\alpha x])\leq \alpha(2\pi (1-\rho^2)t)^{-1/2} x < x, \qquad \text{for all } t\geq t_1,
	\]
for every $x>0$. Therefore, the minimality constraint \eqref{eq:BU_PhysicalCondition} gives us that $L$ cannot have a blow-up after time $t_\star$. It remains to ensure that there is no blow-up on the intermediate time interval $(t_0,t_\star)$. To this end, we can consider events where the realisation of the common Brownian motion belong to a family of paths of the form
		\[
	\mathcal{B}_{t_0,t_\star,\epsilon} :=\{ f:[t_0,t_\star] \to \mathbb{R} \;\;\textrm{s.t.}\;\; |f(t) - f(t_0)| \leq \epsilon \textrm{ for all } t \in [t_0, t_\star]  \}.
	\]
	The work so far has been independent of $m>0$, so we can certainly take it to be given by $m = t_0^{-1} K$, where we are free to choose the constant $K>0$. This choice of $m$ gives us that
	\[
	\rho B^0_t \geq K - x_\star - \delta -\alpha, \qquad \text{for all } t\in[t_0,t_1],
	\]
	On the event $(\rho B^0_t)_{t\in[0,t_1]} \in \mathcal{A}_{t_0, \delta, m} \cap \mathcal{B}_{t_0, t_1, \alpha}$. Insisting also on $K+x_\star>2(\alpha + \delta) $, we can therefore deduce from \eqref{eq:BU_CommonNoiseBU_Escape} that 
	\[
	\nu_{t-}( \hspace{0.4pt} [0, \alpha x] \hspace{0.4pt} ) \leq \frac{\alpha}{ \sqrt{2\pi (1-\rho^2) t_0}} \exp\Big\{ -\frac{(K+x_\star-2\alpha - 2\delta)^2}{2(1-\rho^2)t_1} \Big\} \cdot x,
	\]
	for every $x\in(0,\delta)$, for all $t\in[t_0,t_1]$, on the event $(\rho B^0_t)_{t\in[0,t_1]} \in \mathcal{A}_{t_0, \delta, m} \cap \mathcal{B}_{t_0, t_1, \alpha}$. Taking $K>0$ large enough, we can in turn ensure that $\nu_{t-}(\hspace{0.4pt}[0, \alpha x]\hspace{0.4pt}) < x$, for any $x>0$ near the origin, for all $t\in[t_0, t_1]$. By the minimal jump constraint \eqref{eq:BU_PhysicalCondition}, this means that $L$ cannot have a blow-up on $[t_0,t_1]$, and hence we conclude that
	\[
	0 < \mathbb{P}\bigl( (\rho B^0_s)_{s\in [0,t_1]} \in \mathcal{A}_{t_0,\delta,m } \cap \mathcal{B}_{ t_0, t_1, \alpha}  \bigr) \leq \mathbb{P}(L \textrm{ does not have a blow-up}),
	\]
	as desired. This completes the proof.
\end{proof}

\subsubsection{Provoking blow-ups}

The next result establishes the most interesting direction of part (ii) in  Theorem \ref{Prop_BU_CommonNoiseBlowUp}. Namely that, if \eqref{MV} is started from a fixed start point $X_0=x_0>0$, then there is a strictly positive probability of blow-up for all $\alpha>0$.
 While the former concerns Dirac initial conditions, the statement below is phrased slightly more generally in order to also include the initial conditions considered in Examples \ref{Example1} and  \ref{Example2} above.

\begin{prop}[Non-trivial  blow-up]\label{prop:Dirac_blow-up}
	\label{Thm_enforce_blow-up}  Let the initial condition $\nu_0$ be supported on $(x_\star,\infty)$, for some $x_\star>0$, and suppose we have $\int_{x_\star}^\infty (x-x_\star)\nu_{0}(dx) < \frac{1}{2}\alpha $.  In this case, the conditional McKean--Vlasov problem \eqref{MV} with common noise $(\rho>0)$ satisfies
	\[
	0<\mathbb{P}(L\;\mathrm{\emph{has a blow-up}})<1,
	\]
	for the given feedback parameter $\alpha>0$.
\end{prop}
\begin{proof} The first claim, that the probability of a blow-up is strictly less than one, is a consequence of Proposition \ref{Thm_avoid_blow-up}, so we only need to consider the second claim. In order to construct an event on which $L$ has a blow-up, we fix $x_\star>0$ as in the statement, and we then fix a small $\varepsilon_0>0$ such that $\int_{x_\star}^\infty (x-x_\star)\nu_{0}(dx) \leq \frac{1}{2}\alpha - \varepsilon_0 $. As in the proof Proposition \ref{Thm_avoid_blow-up}, we write
	\[
	\mathcal{A}_{t_0,\delta_0,m}= \bigl\{  f:[0,t_0]\rightarrow \mathbb{R} \;\text{s.t.}\; |f(t)-mt| \leq \delta \; \text{for all}\;t\in[0,t_0]  \bigr\},
	\]
	but this time we take
	\[
	m= -\frac{x_\star - \varepsilon_0/2}{t_0}\qquad \text{and}\qquad \delta_0 < \frac{\varepsilon_0}{8},
	\] 
	where $t_0>0$ and $\delta_0>0$ are to be determined. Furthermore, we let $\bar{L}$ be a solution to the idiosyncratic problem
	\[
	\begin{cases}
	\bar{X}_t = ( X_0+3\varepsilon_0/8 - x_\star) + \sqrt{1-\rho^2} B_t  - \alpha \bar{L}_t \\
	\bar{\tau} = \inf\{ t \geq 0 : \bar{X}_t \leq 0 \}  \\
	\bar{L}_t = \mathbb{P}(\bar{\tau} \leq t )
	\end{cases}
	\]
	where $X_0 $ is distributed according to $\nu_0$, which we recall is supported in $(x_\star,\infty)$ by assumption. Notice that, on the event $(\rho B^0_t)_{t\in[0,t_0]} \in \mathcal{A}_{t_0,m,\delta_0}$, we have
	\[
	\rho B^0_t \geq mt - \delta_0 >  3\varepsilon_0/8 -x_\star, \qquad \text{for all } t\in[0,t_0).
	\]
	Consequently, it follows from the no-crossing result of Lemma \ref{Lem_BU_NoCrossing} that either $L$ has a blow-up on $[0,t_0]$ or else $L$ is continuous on this interval and then $L \leq \bar{L}$ on $[0,t_0]$.
	
	In the former case, the claim is immediate, so for the rest of the proof we restrict to the latter case, where we then have
	\[
	L \leq \bar{L} \qquad \text{on } [0,t_0].
	\]
	Note that we are free to tune $t_0$ in the above. Since $L$ is right-continuous with $\bar{L}_0=0$, we have $\bar{L}_t \rightarrow 0$ as $t\downarrow 0$, so we can take $t_0$ small enough such that $\bar{L}_{t_0} \leq (1 \land \alpha^{-1})\delta_0$. On the event $(\rho B^0_t)_{t\in[0,t_0]} \in \mathcal{A}_{t_0, \delta_0,m}$, we therefore have
	\begin{equation}\label{eq:L-bound}
    L_{t}\leq (1 \land \alpha^{-1})\delta_0 \qquad \text{for all } t\in[0,t_0],
	\end{equation} by the above comparison with $\bar{L}$. Moreover, as $\bar{L}$ is increasing, we can always decrease $t_0>0$, so it is no loss of generality to insist on $t_0 \leq \delta_0^2$ for the remainder of the proof.
	
	In the remaining part of the proof, we modify the moment argument from \cite[Thm 1.1]{hambly_ledger_sojmark_2018} to show that a blow-up must occur after time $t_0$, provided the common Brownian motion stays close to its value at $t_0$. First of all, we can observe that, if $L$ is continuous up to and including some time $t_1>0$, then $X$ is also continuous there, so stopping $X$ at its first hitting time of the origin, namely $\tau$, we get
	\begin{equation}\label{eq:stop_X}
	0 \leq X_{t \wedge \tau} = X_{t_0 \wedge \tau} +  \sqrt{1-\rho^2} \Delta_{t_0, t} B_{\cdot \wedge \tau} + \rho  \Delta_{t_0, t} B^0_{ \cdot \wedge \tau} - \alpha \Delta_{t_0, t} L_{\cdot \wedge \tau},
	\end{equation}
	for $t\in[0,t_1]$, where we have used the notation $\Delta_{u,v}f := f(v) - f(u)$ for increments. Noting that $L$ is of finite variation and that it is precisely the distribution function for the hitting time $\tau$ conditional on $B^0$, we obtain the explicit expression
	\begin{align*}
	\mathbb{E}[\Delta_{t_0, t} L_{\cdot \wedge \tau} \,| \, B^0 ] &= \int_{t_0}^\infty  (L_{t \wedge s}  - L_{t_0})dL_s  \\
&= \frac{1}{2}(L_t - L_{t_0})^2 + (L_t - L_{t_0}) (1- L_t) \\
	& = \frac{1}{2}(L_t-L_{t_0})(2-L_t-L_{t_0}),
	\end{align*}
provided  $L$ is continuous on $[0,t]$, where we have also used that $L_s\rightarrow 1$ as $s\rightarrow \infty$, by comparison with Brownian motion. On the event $(\rho B^0_t)_{t\in[0,t_0]} \in \mathcal{A}_{t_0, \delta_0,m}$ we certainly have $L_{t_0} \leq \delta_0 $ by \eqref{eq:L-bound}, and of course $L \leq 1$, so taking conditional expectations in \eqref{eq:stop_X} and rearranging yields the inequality
	\begin{equation}
	\label{eq:BU_CommonNoiseBU_ModifiedMomentMethod}
	\frac{1}{2} \alpha (L_t-\delta_0)(1-\delta_0) \leq \mathbb{E} [X_{t_0 \wedge \tau} \,| \, B^0 ] + \rho \mathbb{E}[\Delta_{t_0, t} B^0_{ \cdot \wedge \tau} \,| \,B^0 ],
	\end{equation}
	for $t>t_0$, on the aforementioned event, provided $L$ is continuous there.
	
	Now fix a time $t_1>t_0$, to be determined, and consider the event  $(\rho B^0_t)_{t\in[0,t_1]} \in \mathcal{A}_{t_0, \delta_0,m} \cap \mathcal{B}_{t_0, t_1, \delta_0}$, where we recall the definition
	\[
	\mathcal{B}_{t_0, t_1, \delta} = \{ \, f:[0,t_1] \to \mathbb{R} \textrm{ s.t. } |f(t) - f(t_0)| \leq \delta \textrm{ for all } t \in [t_0, t_1]  \, \}.
	\]
	We will show that $L$ cannot be continuous on $[0,t_1]$ for any realisation of $B^0$ on the event $(\rho B^0_t)_{t\in[0,t_1]} \in \mathcal{A}_{t_0, \delta_0,m} \cap \mathcal{B}_{t_0, t_1, \delta_0}$, as this would lead to a contradiction based on \eqref{eq:BU_CommonNoiseBU_ModifiedMomentMethod}. To estimate the first term on the right-hand side of \eqref{eq:BU_CommonNoiseBU_ModifiedMomentMethod}, notice that continuity of $X$ implies
	\[
	X_{t_0 \wedge \tau}
	\leq |X_{t_0}|
	\leq |X_{0} + \rho B^0_{t_0} | + \sqrt{1-\rho^2} \, |B_{t_0}| + \alpha L_{t_0}.
	\]
	Furthermore, we have $\alpha L_{t_0} \leq \delta_0$ on the event $(\rho B^0_t)_{t\in[0,t_1]} \in \mathcal{A}_{t_0, \delta_0,m}$ by \eqref{eq:L-bound}, and a simple computation gives $\mathbb{E}[|B_{t_0}| \,| \, B^0]=\sqrt{2t_0/\pi}$ as $B$ is independent of $B^0$. Therefore,
	\[
	\mathbb{E}[X_{t_0 \wedge \tau} \,| \, B^0 ]
	\leq \mathbb{E}[|X_{0} + \rho B^0_{t_0} |  \,| \, B^0 ] + 2\delta_0,
	\]
	since we fixed $t_0 \leq \delta_0^2$. Next, we can observe that, on the event  $(\rho B^0_t)_{t\in[0,t_0]} \in \mathcal{A}_{t_0, \delta_0, m}$, we have $X_0 + \rho B^0_{t_0} > 0$ and $\rho B^0_{t_0} \leq -x_\star + \varepsilon_0/2 + \delta_0$, so we get
	\[
	\mathbb{E}[|X_{0} + \rho B^0_{t_0} |  \,| \, B^0 ] \leq \int_{x_\star}^\infty(x-x_\star)\nu_0(dx) + \varepsilon_0/2 + \delta_0 \leq \frac{1}{2}\alpha -  \varepsilon_0/2 + \delta_0.
	\]
	and hence
	\[
	\mathbb{E}[X_{t_0 \wedge \tau} \,| \, B^0 ]
	\leq  \frac{1}{2}\alpha -  \varepsilon_0/2 + 3\delta_0.
	\]
	Concerning the second term on the right-hand side of \eqref{eq:BU_CommonNoiseBU_ModifiedMomentMethod}, we of course have
	\[
	\rho \mathbb{E}[\Delta_{t_0, t} B^0_{ \cdot \wedge \tau} \,| \,B^0 ] \leq \delta_0
	\]
	on the event $(\rho B^0_t)_{t\in[t_0,t_1]} \in \mathcal{B}_{t_0, t_1, \delta_0}$, and thus the continuity of $L$ on $[0,t_1]$ would imply that the inequality (\ref{eq:BU_CommonNoiseBU_ModifiedMomentMethod}) becomes
	\begin{equation}
	\label{eq:BU_CommonNoiseBU_Contradiction}
	\frac{1}{2} \alpha (L_t-\delta_0)(1-\delta_0) 
	\leq \frac{1}{2}\alpha -  \varepsilon_0/2 + 4\delta_0 ,
	\end{equation}
	for $t\in[0,t_1]$, on the event $(\rho B^0_t)_{t\in[0,t_1]} \in \mathcal{A}_{t_0, \delta_0,m} \cap \mathcal{B}_{t_0, t_1, \delta_0}$.
	
	It remains to utilize that we are free to tune the variable $\delta_0\in (0, \varepsilon_0/8)$. In view of Lemma \ref{Lem_BU_NoCrossing}, we can bound $L$ from below on $[0,t_1]$, on the given event, by an idiosyncratic problem whose loss of mass goes to one when time goes to infinity, since Brownian motion hits every level almost surely. Therefore, we can certainly take $\delta_0>0$ small enough and $t_1>t_0$ large enough in order to get a contradiction in  \eqref{eq:BU_CommonNoiseBU_Contradiction}, as soon as $L$ is continuous on $[0,t_1]$ for any given realisation of the common noise contained in the event $(\rho B^0_t)_{t\in[0,t_1]} \in \mathcal{A}_{t_0, \delta_0,m} \cap \mathcal{B}_{t_0, t_1, \delta_0}$. In other words, we have 
	\[
	0 < \mathbb{P}( (\rho B^0)_{t\in[0,t_1]} \in \mathcal{A}_{t_0, \delta_0,m} \cap \mathcal{B}_{t_0, t_1, \delta_0}) \leq \mathbb{P}(L \textrm{ has a blow-up} ),
	\]
	as desired, and thus the proof is complete.
\end{proof}

We end this section by concluding the proof of Theorem \ref{Prop_BU_CommonNoiseBlowUp}. First, we can recall that part (i) of the theorem is an immediate consequence of Proposition \ref{Prop_BU_LinearBounds} from the previous subsection. Next, we can note that, for any given $\alpha>0$, every Dirac initial condition $\nu_0=\delta_{x_0}$ with $x_0>0$ satisfies the conditions of the above Proposition \ref{prop:Dirac_blow-up}. Indeed, we can simply take $x_\star:=x_0-\varepsilon$, for a small $\varepsilon>0$, and then we have $\int_{x_\star}^\infty (x-x_\star)\nu_0(dx)=\varepsilon <\frac{1}{2}\alpha$, so part (ii) of Theorem \ref{Prop_BU_CommonNoiseBlowUp} follows from Proposition \ref{prop:Dirac_blow-up}.


\section{From particles to the mean-field} 
\label{Sect_Particle}
In what follows we will show how solutions to a `relaxed' version of (\ref{MV}) emerge as the limit points of a finite-dimensional particle system with positive feedback. It is this particle system that motivates the applications described in the introduction, and hence it is a salient point that `relaxed' solutions to (\ref{MV}) indeed describe the large population mean-field limit. 

\subsection{The finite particle system}

For the rest of Section \ref{Sect_Particle} we shall be interested in characterizing the large population limit of the particle system
\begin{equation}
\label{eq:Particle_System}
\begin{cases}
dX_{t}^{i,N}=b(t,X_{t}^{i,N})dt+\sigma(t)\bigl(\sqrt{1-\rho(t)^{2}}dB_{t}^{i}+\rho(t)dB_{t}^{0}\bigr)-\alpha dL_{t}^{N} \\[3pt]
L_{t}^{N}=\displaystyle{\frac{1}{N} \sum_{i=1}^{N}}\mathbf{1}_{t\geq\tau^{i,N}}\;\;\;\text{with}\;\;\;\tau^{i,N}=\inf\{t>0:X_{t}^{i,N}\leq0\},
\end{cases}
\end{equation}
for $i=1,\ldots,N$, where $B^{0},\ldots,B^{N}$ is a family of independent Brownian motions, and $X_0^{1},\ldots,X_0^{N}$ are i.i.d.~copies of a given initial condition $X_0$ (independent of the Brownian motions). Here the coefficients for the drift, volatility, and correlation of the diffusive dynamics are given by deterministic functions $(t,x)\mapsto b(t,x)$, $t\mapsto \sigma(t)$, and $t\mapsto \rho(t)$, respectively, with the basic formulation \eqref{MV} corresponding to constant coefficients $b= 0$, $\sigma = 1$, and $\rho\in [0,1)$.

Effectively, one should think of the particles in \eqref{eq:Particle_System} as being absorbed when they first reach (or jump below) the origin, at which point they contribute to the current value of the loss process $L^N$. However, to avoid clouding the notation unnecessarily, we define the particles globally by the dynamics in \eqref{eq:Particle_System} for all times $t\geq 0$.

Some care has to be taken in specifying what is meant by $L^N$ in (\ref{eq:Particle_System}). Clearly, it only jumps when $X^{j,N}_{t-}=0$ for some particle $j$ that has not already crossed the origin at an earlier time (in other words, the particle $j$  was previously unabsorbed). However, if taking $\Delta L_t^N=1/N$ means that $X_{t-}^{i,N}-\alpha\Delta L_t^N \leq0$ for  more previously unabsorbed particles $i\neq j$, then these particles also jump to or below the origin at time $t$, hence contributing further to $L^N$, and so forth. That is, a cascade is initiated, which is resolved by demanding that $\Delta L_t^N$ equals the smallest $k/N$ such that a downward shift of the particles by $\alpha k/N$  results in no more than $k$ previously unabsorbed particles reaching (or crossing) the origin at time $t$, based on the positions of the particles at `time $t-$'.

The above can be formulated more succinctly as saying that we require $L^N$ to be the unique piecewise constant c\`adl\`ag process satisfying
\begin{equation}\label{eq:discrete-jump-condition}
\Delta L^N_t = \inf \bigl\{ \tfrac{k}{N} \geq 0 : \boldsymbol{\nu}^N_{t-}\bigl(\hspace{0.4pt}[0,\alpha \tfrac{k}{N}]\hspace{0.4pt}\bigr) \leq \tfrac{k}{N} \bigr\}, \quad \text{for all } t\geq 0,
\end{equation}
where we have introduced the finite-dimensional flows of sub-probability measures
\begin{equation}\label{eq:finite_flow}
 \boldsymbol{\nu}_{t}^{N}:=\frac{1}{N}\sum_{i=1}^{N}\mathbf{1}_{t<\tau^{i,N}}\delta_{X_{t}^{i,N}},
\end{equation} for each $N\geq 1$, which have left-limits
\[
\boldsymbol{\nu}_{t-}^{N}:=\frac{1}{N}\sum_{i=1}^{N}\mathbf{1}_{t\leq\tau^{i,N}}\delta_{X_{t-}^{i,N}}.
\]
The condition \eqref{eq:discrete-jump-condition},  of course, is nothing but the discrete analogue of the physical jump condition presented in (\ref{eq:BU_PhysicalCondition}).

From here, the existence and uniqueness of the finite particle system (\ref{eq:Particle_System}) becomes immediate, under  Assumption \ref{Assumptions}, as a unique strong solution can be constructed recursively on the intervals of constancy for $L^N$ with \eqref{eq:discrete-jump-condition}-\eqref{eq:finite_flow} dictating when and by how much it jumps.

We note that we have included a spatial drift $(t,x)\mapsto b(t,x)$ with up to linear growth in space to e.g.~allow for Ornstein--Uhlenbeck type dynamics, as this is of interest for the motivating financial \cite{hambly_sojmark_2017} and neuroscience \cite{brunel2000} applications. Our arguments for such a drift extend immediately to a drift also depending on the empirical mean in a Lipschitz way, but we leave this out for simplicity of presentation.

Throughout what follows, we let $D_{\mathbb{R}}$ be the space of real-valued c\`adl\`ag paths, and we let $\hat{\eta}$ denote the canonical process on $D_{\mathbb{R}}$ with $\hat{\tau}$ denoting its first hitting time of zero, that is, $\hat{\tau}=\inf \{t\geq 0:\hat{\eta}_t\leq 0 \}$. In particular, we can then write
\[
L_t^N=\mathbf{P}^N( \hat{\tau} \leq t) \qquad \text{and} \qquad \boldsymbol{\nu}_t^N=\mathbf{P}^N(\hat{\eta} \in \hspace{-1pt} \cdot \hspace{1pt}, \hspace{1pt} t<\hat{\tau}),
\]
where
\begin{equation}\label{eq:empirical_measure}
\quad \mathbf{P}^N:= \frac{1}{N} \sum_{i=1}^N \delta_{X^{i,N}_\cdot}, \qquad \text{for } N\geq1,
\end{equation}
noting that the empirical measures, $\mathbf{P}^N$, should be understood as random probability measures on $D_{\mathbb{R}}$. In terms of notation, we will consistently use the boldface font (e.g.~$\mathbf{P}$ and $\boldsymbol{\nu}$) for \emph{random} measures, whereas the blackboard font (e.g.~$\mathbb{P}$ and its expectation operator $\mathbb{E}$) is reserved for fixed background spaces.

\subsection{Limit points of the particle system}

Heuristically, one could hope that, along a subsequence and in a suitable topology, the pair $(\mathbf{P}^N,B^0)$ will converge in law to some limit $(\mathbf{P}^*,B^0)$, where $B^0$ is again a Brownian motion and $\mathbf{P}^*=\text{Law}(X\,|\, B^0)$ is the conditional law of $X$ given $B^0$, for a c\`adl\`ag process $X$ satisfying (\ref{MV}) with $L_t=\mathbf{P}^*(\hat{\tau} \leq t)$. However, the idea that $\mathbf{P}^*$, and hence also $L$, should automatically turn out to be $B^0$-measurable is asking too much of the weak limit point $(\mathbf{P}^*,B^0)$. Anyhow, this property can be relaxed vis-\`a-vis the original formulation of (\ref{MV}) without altering the essential features of the problem. We define this relaxation next, noting that it is an essential part of our endeavours to recover (\ref{MV}) as a large population limit of the finite system (\ref{eq:Particle_System}). 

\begin{defn}[Relaxed solutions]\label{def:rMV}
Let the coefficient functions $b(t,x)$, $\sigma(t)$, and $ \rho(t)$ be given along with the initial condition $X_0 \sim \nu_0$ and feedback parameter $\alpha >0$. Then we define a relaxed solution to (\ref{MV}) as a family $(X,B,B^0,\mathbf{P})$ on a probability space $(\Omega, \mathbb{P},\mathcal{F})$ of choice such that
\begin{equation}
\begin{cases}
\label{eq:Relaxed McKean-Vlasov} \tag{rMV}
dX_{t}=b(t,X_{t})dt+\sigma(t)\bigl(\sqrt{1-\rho(t)^{2}}dB_{t}+\rho(t)dB_{t}^{0}\bigr)-\alpha dL_{t}\\[2pt]
L_t=\mathbb{P}(\tau\leq t \, | \, B^0 , \mathbf{P}) \;\; \text{and} \;\; \mathbf{P}=\text{Law}(X\,| \, B^{0},\mathbf{P}), \\[2pt]
\tau = \inf \{ t>0 : X_t \leq 0 \} \;\; \text{and} \;\; (B^{0},\mathbf{P})\perp B,
\end{cases}
\end{equation}
with $L_0=0$, $X_0 \sim \nu_0 $ and $X_0 \perp (B,B^0, \mathbf{P})$, where $(B, B^0)$ is a $2d$ Brownian motion, $X$ is a c\`adl\`ag process, and $\mathbf{P}$ is a random probability measure on the space of c\`adl\`ag paths $D_\mathbb{R}$. By the notation $\emph{\text{Law}}(X\,| \, B^{0},\mathbf{P})$, we mean the conditional law of $X$ given $(B^0,\mathbf{P})$ which indeed defines a random $(B^0,\mathbf{P})$-measurable probability measure on $D_\mathbb{R}$.
\end{defn}

We stress that the fixed point type constraint $\mathbf{P}=\text{Law}(X\,| \, B^0, \mathbf{P})$ is a crucial part of our relaxed notion of solution, which characterizes the extra information in the conditioning and  encodes the fact that $L_t=\mathbf{P}(\hat{\tau} \leq t)$, as should be required for a limit point $\mathbf{P}$ of $\mathbf{P}^N$ with $L^N=\mathbf{P}^N(\hat{\tau} \leq t)$. Observe also that, if the random measure $\mathbf{P}$ is $B^0$-measurable, then the above simplifies to the exact form of (\ref{MV}). In the absence of $B^0$-measurability, however, the independence of $(\mathbf{P},B^0)$ and $B$ constitutes the most important point, which allows \eqref{eq:Relaxed McKean-Vlasov} to retain all the essential aspects of (\ref{MV}). In particular, this independence ensures that the results from Section 2 are equally valid for \eqref{eq:Relaxed McKean-Vlasov}, by conditioning on the pair $(\mathbf{P},B^0)$ whenever we conditioned on $B^0$. The nature of relaxed solutions and their agreement with (\ref{MV}) is explored further in  Section \ref{subsec:SPDE}. 
 
We now present the main result of this section, showing that solutions to \eqref{eq:Relaxed McKean-Vlasov} arise as limit points of the finite particle system \eqref{eq:Particle_System}.

\begin{thm}[Existence and convergence] \label{Thm:Existence}
Suppose Assumption \ref{Assumptions} below is satisfied, and let $\mathbf{P}^N$ be the empirical measure \eqref{eq:empirical_measure} for the finite particle system \eqref{eq:Particle_System}. Then any subsequence of the pair $(\mathbf{P}^N,B^0)$ has a further subsequence, which converges in law to some $(\mathbf{P},B^0)$, where $B^0$ is a Brownian motion and $\mathbf{P}$ is a random probability measure $\mathbf{P}:\Omega \rightarrow \mathcal{P}(D_{\mathbb{R}})$. Given this limit $(\mathbf{P},B^0)$, there is a background space $(\Omega,\mathcal{F},\P)$, which carries another Brownian motion $B$ and a stochastic process $X$ such that $(X,B,B^0,\mathbf{P})$ is a solution to \eqref{eq:Relaxed McKean-Vlasov} according to Definition \ref{def:rMV}. Moreover, if we define
\begin{equation}\label{eq:fat_nu}
	\boldsymbol{\nu}_t:=\mathbf{P}(\hat{\eta}_{t} \in \cdot\, , \, t<\hat{\tau} )=\mathbb{P}(X_{t}\in \cdot\, , \, t<\tau \,| \,B^{0},\mathbf{P}),
\end{equation}
then we have the minimality constraint
\begin{equation}\label{thm_PJC}
\Delta L_{t}=\inf\{ x\geq0:\boldsymbol{\nu}_{t-}(\hspace{0.4pt}[0,\alpha x]\hspace{0.4pt})<x\},
\end{equation}
a.s., for all $t\geq 0$, which determines the jump sizes of $L$.
\end{thm}

As in Section 2, it is understood that $\boldsymbol{\nu}_{0-}=\nu_0$ and $L_{0-}=0$. The full list of structural conditions that we impose on \eqref{eq:Particle_System} and \eqref{eq:Relaxed McKean-Vlasov} are collected in Assumption \ref{Assumptions} immediately below. However, the proof of the theorem is postponed, as it requires several auxiliary results which form the subject of Sections \ref{subsec:estimates} and \ref{Sec:Continuity}. For now, we pause instead to explore the contents of the theorem a little further.

Note, in particular, that the requirement of right-continuity and $L_0=0$ in Definition \ref{def:rMV} corresponds to not having a jump at time $t=0$, so the minimality constraint \eqref{thm_PJC} says that the initial condition $\nu_0$ should satisfy $\inf\{ x\geq0:\nu_0(\hspace{0.4pt}[0,\alpha x]\hspace{0.4pt})<x\}=0$.

\begin{rem}[Schauder fixed point approach] While this paper focuses on convergence of the particle system, one can of course also frame \eqref{MV} as a fixed point problem. Such an approach has now been implemented in Remark 2.5 of \cite{nadtochiy_shkolnikov_2018}, based on a Schauder fixed point argument generalised to Skorokhod's M1 topology on $D_\mathbb{R}$. This yields a direct proof of existence for solutions to \eqref{MV} rather than \eqref{eq:Relaxed McKean-Vlasov}. However, the results in \cite{nadtochiy_shkolnikov_2018} do not address a condition for the jump sizes such as the minimality constraint \eqref{thm_PJC}, which is an integral part of Theorem \ref{Thm:Existence} (see also Proposition \ref{Existence_Prop_MinimalityForCommonNoise} below). Finally, we note in passing that the M1 topology is also key to our arguments in the subsections that follow (as in \cite{DIRT_SPA} for $\rho=0$).
\end{rem}

We do not address uniqueness in this paper, but we reiterate here, as in the introduction, that there has been some important new developments on this front (at least in the constant coefficient case). If the law of $X_0$ has a density $V_0\in L^\infty(0,\infty)$, then global uniqueness has been shown for $\rho \in [0,1)$ under the smallness condition $\Vert V_0 \Vert<\alpha^{-1}$ \cite{LS_unique} and for $\rho =0$ under the general condition that $V_0$ does not change monotonicity infinitely often on compacts \cite{delarue_sergey_shkolni}.
In these situations, we have full convergence in law of the particle system, and the relaxed limiting solutions \eqref{eq:Relaxed McKean-Vlasov} for $\rho>0$ simplify to the form \eqref{MV}, as explained in \cite[Thm. 2.3]{LS_unique}.

For general initial conditions, uniqueness remains an open problem when $\rho>0$, and we note that Theorem \ref{Thm:Existence} does not yield full propagation of chaos in the absence of such uniqueness.
 
 \begin{assump}[Structural conditions]
 	\label{Assumptions}
 	The drift, $b(t,x)$, is assumed to be Lipschitz in space with linear growth bound $\left|b(t,x)\right|\leq C(1+\left|x\right|)$. The volatility, $\sigma(t)$, and the correlation, $\rho(t)$, are taken to be deterministic functions in  $\mathcal{C}^{\kappa}(0,T)$, for some $\kappa>1/2$, satisfying the non-degeneracy conditions $0\leq\rho(t)\leq1-\epsilon$
 	and $\epsilon\leq\sigma(t)\leq \epsilon^{-1}$, for some $\epsilon \in (0,1)$. Furthermore, we assume that $\int\!x^8d\nu_0<\infty$, and that $\nu_0(\hspace{0.4pt}[0,x]\hspace{0.4pt} )\leq\frac{1-\epsilon}{2}\alpha^{-1}x$ for all $x > 0$ sufficiently small, for some $\epsilon\in(0,1)$, where $\nu_0$ is the distribution of the initial condition $X_0$ which takes values in $(0,\infty)$.
 \end{assump}
 Notice that, following a standard Gr\"onwall argument, the moment assumption on the initial law $\nu_0$ guarantees that $ \E [\sup_{t\leq T}|X^{1,N}_t|^8]$ is bounded uniformly in $N\geq1$, for any given $T<\infty$. When making use of this observation in later subsections, we will simply refer to Assumption \ref{Assumptions} without writing out the details.

Our next result touches on an essential aspect of Theorem \ref{Thm:Existence} that was in fact  anticipated already in Section 2, when we motivated the minimality constraint \eqref{eq:BU_PhysicalCondition}. Indeed, we can observe that \eqref{eq:BU_PhysicalCondition} is precisely the property \eqref{thm_PJC} enjoyed by the limiting solutions obtained in Theorem \ref{Thm:Existence}.
As we pointed out in Section 2, this property amounts to picking out the smallest possible jump sizes whilst also insisting on the system being c\`adl\`ag. This claim is verified by the next result, based on an adjustment to the argument in \cite[Prop.~1.2]{hambly_ledger_sojmark_2018} which established the corresponding result for the idiosyncratic model. As usual, we employ the conventions $\boldsymbol{\nu}_{0-}=\nu_0$ and $L_{0-}=0$.

\begin{prop}[Minimality of jumps]
	\label{Existence_Prop_MinimalityForCommonNoise}
	Consider any solution to \eqref{eq:Relaxed McKean-Vlasov} given by Definition \ref{def:rMV}, where the coefficient functions $b(t,x)$, $\sigma(t)$, and $\rho(t)$ are as in Assumption \ref{Assumptions}. In order for the solution to be c\`adl\`ag, as required by Definition \ref{def:rMV}, the loss process $L$ must satisfy
	\begin{equation}\label{eq:jump_size_ineq}
	\Delta {L}_t \geq \inf\{ x > 0 : \boldsymbol{\nu}_{t-}(\hspace{0.4pt} [ 0,\alpha x] \hspace{0.4pt}) < x \}
	\end{equation}
	with probability one, for all $t\geq 0$.\end{prop}

\begin{proof}
	We will proceed by contradiction, so suppose $L$ is a c\`adl\`ag solution for which the inequality \eqref{eq:jump_size_ineq} is violated at some time $t \geq 0$.

	By following the strategy in the proof of \cite[Prop.~1.2]{hambly_ledger_sojmark_2018} and \cite[Remark 2.6]{hambly_ledger_sojmark_2018}, we can always consider the restarted system, so it is no loss of generality to take $t=0$ with $L_0=0$ and $\Delta L_0=0$. Thus, assuming for a contradiction that \eqref{eq:jump_size_ineq} is violated at $t=0$, we can find $x_0 > 0$ such that 
	\begin{equation}\label{eq:x0_contradiction}
		\nu_0(\hspace{0.4pt}[0,x]\hspace{0.4pt}) \geq \alpha^{-1} x \qquad \text{for all } x < \alpha x_0. 
	\end{equation}
	We will show that this leads to a violation of the right-continuity of $L$ at $t=0$, by achieving a bound on $L_h$ from below as $h\downarrow 0$. 
	
	Compared to \cite[Prop.~1.2]{hambly_ledger_sojmark_2018}, we must account for the common noise and the drift term in the dynamics of $X$. To this end, we introduce the notation
	\[
	Y^0_t := \int^t_0 \sigma(s) \rho(s) dB^0_s,
	\qquad
	Y_t := \int^t_0 \sigma(s) \sqrt{ 1 - \rho(s)^2 } dB_s.
	\]
	Using the linear growth bound on the drift $b(t,x)$ from Assumption \ref{Assumptions}, we then have
	\[
	X_t \leq X_0 + Ct + C\int^t_0 X_s ds + Y^0_t + Y_t - \alpha L_t,
	\]
	for all $t \leq \tau$. In turn, applying Gr\"onwall's lemma gives
	\[
	X_t \leq e^{Ct} X_0 + Ce^{Ct} t + \bar{Y}^0_t + \bar{Y}_t - \alpha \bar{L}_t
	\leq e^{Ct} X_0 + Ce^{Ct} t + \bar{Y}^0_t + \bar{Y}_t - \alpha L_t
	\]
	for $t \leq \tau$, where we have used the short-hand notation
	\[
	\bar{f}_t := \int^t_0 e^{C(t-s)}df_s,
	\] and we have furthermore used that, with this definition, we have $\bar{L} \geq L$. It follows, by contradiction, that the event $X_0 + Ch + \bar{Y}^0_h + \bar{Y}_h - \alpha L_h \leq 0$ is a subset of the event $\inf_{s\leq h}X_s \leq0 $, for any $h>0$.
	
	Since $e^{Ch} = 1 + O(h)$ as $h \downarrow 0$, we can take a $B^0$-random subsequence of $h \downarrow 0$ for which $\bar{Y}^0_h \leq -h^{1/2}$ (recall that $\bar{Y}^0$ is a deterministic time-change of the Brownian motion $B^0$). In turn, we have
	\[
	 X_0 + Ch + \bar{Y}^0_h + \bar{Y}_h - \alpha L_h \leq X_0 + \bar{Y}_h - \alpha L_h,
	\]
	along that subsequence for $h > 0$ sufficiently small. For the rest of the proof, we restrict to this ($B^0$-random) subsequence. It then follows that

\begin{equation}
\label{eq:MinimalityForCommonNoise_I}
L_h = \mathbb{P}(\tau \leq h \,| \, B^0, \mathbf{P} ) \geq 
\int_{0}^\infty \mathbb{P}( x + \bar{Y}_h - \alpha L_h \leq 0 \,| \, B^0, \mathbf{P}) \nu_0(dx).
\end{equation}

Clearly, the integrand on the right-hand side of \eqref{eq:MinimalityForCommonNoise_I} is differentiable and decreasing as a function of $x\in(0,\infty)$. Thus, returning to \eqref{eq:x0_contradiction}, we can repeat the argument from the start of the proof of
\cite[Prop.~1.2]{hambly_ledger_sojmark_2018} to deduce that 
	\[
	L_h \geq \alpha^{-1} \int^{\alpha x_0}_0 \mathbb{P} (\bar{Y}_h \leq -x + \alpha L_h  \,| \, B^0, \mathbf{P}) dx,
	\]
	where $x_0>0$ was fixed as part of \eqref{eq:x0_contradiction}.

Next, we can observe that $\bar{Y}_{h}$ is independent of $(B^0, \mathbf{P})$, and that, conditionally on $(B^0, \mathbf{P})$, it is thus  normally distributed with mean $0$ and variance  $v(h)$,  where we have defined
	\[
	v(h):=\int^h_0 \sigma(t)(1 - \rho(t)^2)e^{C(h-t)}dt \asymp h.
	\] 
	Therefore, after a change of variables, we get
	\[
	\alpha v(h)^{-\tfrac{1}{2}} L_h \geq \int^{\alpha v(h)^{-1/2} x_0 - \alpha v(h)^{-1/2} L_h }_{ - \alpha v(h)^{-1/2} L_h } \Phi(-y)dy,
	\]
	for $h>0$ along the aforementioned subsequence,
	with $\Phi$ denoting the standard normal cdf. At this point the remainder of the proof is identical to that of \cite[Prop.~1.2]{hambly_ledger_sojmark_2018} (for each realization of the $B^0$-dependent sequence $h\downarrow 0$), which gives us the desired contradiction.
\end{proof}

By definition of $\boldsymbol{\nu}$, we can write $L_t=1-\boldsymbol{\nu}_t([0,\infty))$, for all $t\geq0$, with the jump sizes $\Delta L_t$ fully determined by the left limit $\boldsymbol{\nu}_{t-}$ via the minimality constraint \eqref{thm_PJC}. Hence it is natural to consider  $(X,B,B^0,\boldsymbol{\nu})$ as the leanest way of writing the solution, where $\boldsymbol{\nu}=(\boldsymbol{\nu}_t)_{t\in[0,T]}$ is viewed as a measure-valued stochastic process. In line with this point of view, we have the following result, whose proof is postponed to Section \ref{subsec:filtration}.

\begin{prop}[Filtration]\label{filtration}
	Define $\boldsymbol{\nu}:=(\boldsymbol{\nu}_t)_{t\in[0,T]}$ as the stochastic process whose marginals are the random measures $\boldsymbol{\nu}_t:\Omega \rightarrow \text{\emph{M}}(\mathbb{R})$ given by \eqref{eq:fat_nu} in Theorem  \ref{Thm:Existence}, where $\text{\emph{M}}(\mathbb{R})$ is the space of measures on $\mathbb{R}$ equipped with its Borel $\sigma$-algebra under the total variation norm. Then there is a filtration $(\mathcal{F}_t)_{t\in[0,T]}$, forming a filtered space $(\Omega,\mathcal{F},\mathcal{F}_t,\mathbb{P})$, for which $(X,B,B^0,\boldsymbol{\nu})$ is adapted and $(B, B^0)$ is a 2d Brownian motion.
\end{prop}

In the formulation of Definition \ref{def:rMV} and Theorem \ref{Thm:Existence} it is not necessary to refer to a filtration, but the previous proposition emphasises that one can indeed work with a filtration to which $\boldsymbol{\nu}$ and $L$ are adapted without affecting the Brownian dynamics. This observation leads us naturally to the next subsection, where we derive a stochastic evolution equation for the flow of the adapted stochastic process $\boldsymbol{\nu}$.

\subsection{Stochastic evolution equation}\label{subsec:SPDE}

Following on from the above, we can think of the McKean--Vlasov system \eqref{eq:Relaxed McKean-Vlasov} in terms of the adapted stochastic process $\boldsymbol{\nu}$. However, instead of working with the full filtration $\mathcal{F}_t$ from Proposition \ref{filtration}, we now wish to consider the subfiltration $\mathcal{F}^{B^0\!\!,\,\boldsymbol{\nu}}_t$ generated only by the common noise $B^0$ and the stochastic process $\boldsymbol{\nu}$ itself. The explicit construction of this is given in Section \ref{subsec:filtration} (along with the proof of Proposition \ref{filtration}), verifying that $B^0$ is indeed a standard Brownian motion in $\mathcal{F}^{B^0\!\!,\,\boldsymbol{\nu}}_t$ and that this subfiltration is independent of $B$.

The next result shows that $ \boldsymbol{\nu}$ satisfies a generalized stochastic evolution equation driven by the common noise $B^0$. While we work with the relaxed formulation \eqref{eq:Relaxed McKean-Vlasov}, we stress that identical arguments reveal that \eqref{MV} leads to the very same evolution equation. The only difference is whether or not the solution is automatically adapted to the driving noise $B^0$, which provides another point of view on the distinction between \eqref{eq:Relaxed McKean-Vlasov} and \eqref{MV}. Aside from this, our main aim here is to make more clear the connection between the McKean--Vlasov formalism of the present paper and the focus on Fokker--Planck equations in the related PDE and mathematical neuroscience literature, as further discussed in Remark \ref{rem:neuro_pde} below.

\begin{prop}[Stochastic evolution equation]\label{prop:SPDE_weak_form}Let $\boldsymbol{\nu}$ be defined by \eqref{eq:fat_nu}  for a given solution to \eqref{eq:Relaxed McKean-Vlasov} under Assumption \ref{Assumptions}. Then each marginal  $\boldsymbol{\nu}_t$ is supported on the positive half-line $(0,\infty)$, for all $t\geq0$, and the flow $t\mapsto \boldsymbol{\nu}_t$ satisfies the nonlinear stochastic evolution equation
\begin{equation}\label{eq:measure_SPDE}
\left\{
\begin{aligned}
	d\langle \boldsymbol{\nu}_t , \phi \rangle &=  \bigl\langle \boldsymbol{\nu}_t, \tfrac{1}{2}\sigma(t)^2\phi^{\prime\prime} + b(t,x)\phi^\prime \bigr\rangle dt+\bigl\langle \boldsymbol{\nu}_t,\rho(t)\sigma(t)\phi^\prime \bigr\rangle dB_{t}^{0} \\[2pt]
	&\quad\; -\langle \boldsymbol{\nu}_t,\alpha\phi^\prime \rangle d{L}^{c}_{t} +  \int_{\mathbb{R}} \langle \boldsymbol{\nu}_{t-}, \phi(\cdot - \alpha x)- \phi \rangle J_L(dt,dx),
	 \\[6pt]
	 L^c_t &= L_t - \sum_{0<s\leq t}\Delta L_s  \quad \text{and} \quad J_L=\sum_{0<s<\infty} \delta_{(s,\Delta L_s)}, \\[6pt]
	 	L_{t}&=1-\boldsymbol{\nu}_t([0,\infty)) \quad \text{and} \quad \Delta L_{t}=\inf\{ x\geq0:\boldsymbol{\nu}_{t-}(\hspace{0.4pt}[0,\alpha x]\hspace{0.4pt})<x\},
	\end{aligned}
	\right.
\end{equation}
	for test functions $\phi \in \mathcal{D}_0$, where $\mathcal{D}_0:=\{ f\in \mathcal{C}^2_b(\mathbb{R}) : f\vert_{(-\infty, 0]}=0 \}$.
\end{prop}

Before turning to the proof, we take a moment to interpret the equation \eqref{eq:measure_SPDE}. Most importantly, we know that $L$ is of finite variation, so it is clear that all the terms in \eqref{eq:measure_SPDE} make sense. Now suppose there is a blow-up at some time $t>0$. Then the dynamics in \eqref{eq:measure_SPDE} yield
		\[
	\Delta \langle \boldsymbol{\nu}_{t} , \phi \rangle = \langle \boldsymbol{\nu}_{{t}-}, \phi(\cdot - \alpha \Delta L_t) \rangle - \langle \boldsymbol{\nu}_{t-} , \phi \rangle,\quad \text{for all } \phi \in \mathcal{D}_0.
	\]
From Lemma \ref{density_process}, we know that $\boldsymbol{\nu}_t$ has a density $V_t$, for all $t\geq 0$, so we deduce that a blow-up can be described as the system restarting from the shifted density
\[
V_{t}(x)=V_{t-}(x+\alpha \Delta L_{t}), \quad \text{with} \quad \Delta L_{t}=\inf\{ x\geq0:\textstyle\int_0^{\alpha x}V_{t-}(y)dy<x\}.
\]
Furthermore, by performing a formal integration by parts,  we see that \eqref{eq:measure_SPDE} gives a generalized notion of solution for the evolution equation presented  in Remark \ref{rem:evo_eqn} phrased for the flow of the densities $t\mapsto V_t$. Finally, we recall that Figure \ref{fig:theonlyfigure} shows two heat plots illustrating how these densities evolve in time, with the right-hand plot displaying a blow-up given by a shift of the density as described above.

\begin{rem}[Integrate-and-fire models]\label{rem:neuro_pde}
	In the mathematical neuroscience literature, stochastic Fokker--Planck equations analogous to \eqref{eq:measure_SPDE} appear e.g.~in \cite[Eqn.~(32)]{brunel2000}, \cite[Eqn.~(3.24)]{brunel-hakim_1998}, \cite[Eqn.~(2.29)]{mattia_2002}, and \cite[Eqn.~(4)]{torcini} as `integrate-and-fire' models for networks of electrically coupled excitatory neurons. However, there is no discussion of a jump component in these works, while Theorem \ref{Prop_BU_CommonNoiseBlowUp} shows that blow-ups are an inevitable part of the stochastic evolution starting from a Dirac mass. Moreover, it is heuristically assumed that $t\mapsto L_t$ has a derivative, which is then referred to as the `firing rate' and is identified with the flux across the boundary. Since $t\mapsto L_t$ is increasing, it is indeed differentiable almost everywhere, and the equation \eqref{eq:measure_SPDE} formally gives the flux condition \[
	L^\prime_t=\frac{1}{2}\sigma(0)^2\partial_x V(0), \qquad \text{for all } t>0,
	\]
	in line with \cite[Eqn.~(2.6)]{mattia_2002}. Nevertheless, for $\rho>0$, one can no longer expect $L$ to be absolutely continuous even in between blow-ups, nor will $\partial_x V(0)$ be well-defined (due to the roughness of $t\mapsto B^0_t$). In this respect, it is clear that the McKean--Vlasov formulation is better suited for working with blow-ups, and it entirely avoids the aforementioned regularity issues.
\end{rem}
	
\begin{proof}[Proof of Proposition \ref{prop:SPDE_weak_form}]  Let $\boldsymbol{\nu}$ be defined by \eqref{eq:fat_nu}, for a solution to \eqref{eq:Relaxed McKean-Vlasov} given by Theorem \ref{Thm:Existence}, and consider $(X,B,B^0,\boldsymbol{\nu})$ on the filtered space from Proposition \ref{filtration}. By construction, $L_t=1-\boldsymbol{\nu}_t([0,\infty))$. Since $L$ is  of finite variation, the sum of jumps is convergent, and hence the continuous part $L^c$ is given by the well-defined expression
	\[
	L^c_t= L_t - \sum_{0<s\leq t}\Delta L_s,\qquad \text{for all } t\geq 0.
	\]
Moreover, this means that integrals against both $L$ and $L^c$ are well-defined in the Lebesgue--Stieltjes sense. As concerns the final condition on the jump sizes $\Delta L$, this is just the minimality constraint \eqref{thm_PJC}, which is satisfied by Theorem \ref{Thm:Existence}. In order to derive \eqref{eq:measure_SPDE}, we let $\mathcal{F}^{B^0\!\!,\,\boldsymbol{\nu}}_t$ be the corresponding filtration generated by $B^0$ and $\boldsymbol{\nu}$, as detailed in Section \ref{subsec:filtration}. The adaptedness of $\boldsymbol{\nu}$ to this filtration implies that $\boldsymbol{\nu}$ and $L$ can be written as
		\[
		\boldsymbol{\nu}_t  = \mathbb{P}( X_t\in \cdot\,,\;t<\tau\,|\,\mathcal{F}_t^{B^0\!\!,\,\boldsymbol{\nu}}) \quad \text{and} \quad L_t=\mathbb{P}({\tau} \leq t \, | \, \mathcal{F}_t^{B^0\!\!,\,\boldsymbol{\nu}}),
		\]
	for $t\geq 0$, with 	$\boldsymbol{\nu}_0= \mathbb{P}( X_0\in \cdot\,)$.
	As explained in Section \ref{subsec:filtration}, $\mathcal{F}_t^{B^0\!\!,\,\boldsymbol{\nu}}$ is a subfiltration of the filtration $\mathcal{F}_t$ from Proposition \ref{filtration}, $B^0$ remains a Brownian motion with respect to this filtration, and, crucially, the filtration is independent of the other Brownian motion $B$. In view of the above expression for $\boldsymbol{\nu}_t$, we see that its action on test functions $\phi \in \mathcal{D}_0$ is given by
		\[
	\langle \boldsymbol{\nu}_t ,\phi \rangle := \int \phi(x)  \boldsymbol{\nu}_t (dx) = \mathbb{E}[\phi(X_t)\mathbf{1}_{t<\tau}\,|\,\mathcal{F}_t^{B^0\!\!,\,\boldsymbol{\nu}}]
	\]
	for $t \geq 0$.	Applying It\^o's formula for general (c\`adl\`ag) semimartingales to $X_{\cdot\land \tau}$, we obtain
\begin{align}\label{Ito_nu}
\phi(X_{t\land \tau}) = \; & \phi(X_{0}) + \frac{1}{2}\int_{0}^t \textbf{1}_{s<\tau} \phi^{\prime\prime}(X_{s})\sigma(s)^2ds  + \int_{0}^t \textbf{1}_{s<\tau} \phi^\prime(X_{s})b(s,X_s)ds \nonumber \\
&  + \int_{0}^t \textbf{1}_{s<\tau} \phi^\prime(X_s)\sigma(s)\sqrt{1-\rho(s)^2}dB_s + \int_{0}^t \textbf{1}_{s<\tau} \phi^\prime(X_s)\sigma(s)\rho(s)dB_s^0 \nonumber \\
& - \alpha \int_{0}^t \textbf{1}_{s<\tau} \phi^\prime(X_{s})dL^c_s + \sum_{0<s\leq t}  1_{s\leq \tau}\bigl( \phi(X_{s-}-\alpha \Delta L_t) - \phi(X_{s-}) \bigr ),
\end{align}
 for $t\geq 0$,	for all $\phi\in\mathcal{D}_0$, where we have used that $L$ is of finite variation. From here, we note that \[
 \phi(X_t)\mathbf{1}_{t<\tau}=\phi(X_{t \land \tau}),\qquad \text{for all } t\geq 0,
 \]
 by definition of $\mathcal{D}_0$, so $\langle \boldsymbol{\nu}_t ,\phi \rangle $ equals the conditional expectation given $\mathcal{F}_t^{B^0\!\!,\,\boldsymbol{\nu}}$ of the right-hand side of \eqref{Ito_nu}. Since the filtration $\mathcal{F}_t^{B^0\!\!,\,\boldsymbol{\nu}}$ is independent of $B$, the integral against $B$ is killed under this operation of taking conditional expectation, whereas the $\mathcal{F}_t^{B^0\!\!,\,\boldsymbol{\nu}}$-adaptedness of $L^c$ and $B^0$ means that this operation commutes with the integrals against these processes (see e.g.~\cite[Lemma 8.9]{hambly_ledger_2017}).  Therefore, by the previous two observations and Fubini's theorem, we can deduce \eqref{eq:measure_SPDE} from \eqref{Ito_nu}.
\end{proof}

In addition to the formulation \eqref{eq:measure_SPDE} being weak in the sense of generalized functions, we stress that the solution given by the flow $t\mapsto\boldsymbol{\nu}_t$ for \eqref{eq:Relaxed McKean-Vlasov} is also probabilistically weak in the sense that the filtration $\mathcal{F}^{B^0,\boldsymbol{\nu}}_t$ is allowed to be strictly larger than that generated by the driving noise $B^0$ alone. For a similar situation, see also the classical paper \cite{dawson}. In other words, the fact that the loss process $L$ in \eqref{eq:Relaxed McKean-Vlasov} is conditional on $(B^0,\mathbf{P})$, rather than just $B^0$,  simply amounts to the corresponding solution  $\boldsymbol{\nu}$ of the stochastic evolution equation \eqref{eq:measure_SPDE} not necessarily being adapted to the driving noise $B^0$.

\subsection{Preliminary estimates}\label{subsec:estimates}

In this subsection we establish several important estimates concerning the behaviour of $L^N$,  $\boldsymbol{\nu}^N$, and a tagged particle $X^{i,N}$ as $N\rightarrow \infty$. These estimates are then used to establish suitable tightness and continuity results in the next subsection. Throughout both subsections, and indeed for the remainder of the paper, we shall always be working on the premise that Assumption \ref{Assumptions} is satisfied.

Due to our particular setup and the choice of topology in Section 3.2 below, the key prerequisite for tightness is sufficient control over the loss process $L^N$ near the initial time $t=0$. To achieve this, we need to exploit some control on the initial condition near the origin. The assumption $\nu_0(\hspace{0.4pt}[0,x]\hspace{0.4pt} )\leq\frac{1-\epsilon}{2}\alpha^{-1}x$ is by no means optimal, but it is already quite general (for example, for $\rho=0$, \cite{DIRT_SPA,nadtochiy_shkolnikov_2017} assumes $\nu_0$ is supported away from the boundary) and it allows for an intuitive argument to control the loss of mass. The point of the constant $c=\frac{1-\epsilon}{2}<1$ is that we can utilise $c<1-c$. Moreover, we rely on the fact that $\rho(t)$ is bounded away from $1$, as we need to single out the contribution of the common noise and utilise the averaging of the independent Brownian motions in a law of large numbers argument.
\begin{prop}[Initial control of the loss]
	\label{prop: Loss process nice at 0} Consider the sequence of loss processes $L^N$, for $N \geq 1$, as given by the unique solution to the particle system \eqref{eq:Particle_System} of size $N$ under Assumption \ref{Assumptions}. Then we have that
	\[
\limsup_{N\rightarrow\infty}\mathbb{P}\bigl(L_{\delta}^{N}\geq\alpha^{-1}\log(\delta^{-1})\delta^{\frac{1}{2}}\bigr)=0.
	\]
	for all $\delta\in(0,\delta_0)$, for some small enough $\delta_0>0$.
\end{prop}
	\begin{proof} For clarity, we split the proof into three concise steps.
		
		\textbf{Step 1. } Given $\epsilon>0$ from Assumption \ref{Assumptions}, we fix two constants $\lambda>0$ and $\lambda^\prime>0$ such that $1-\epsilon<\lambda^{\prime}<\lambda<1$. Setting $\varepsilon=\varepsilon(\delta):=\delta^{\frac{1}{2}}\log(\delta^{-1})$, we then have
		\[
			\mathbb{P}(L_{\delta}^{N}\geq\alpha^{-1}\varepsilon)\leq\mathbb{P}\bigl(L_{\delta}^{N}\geq\alpha^{-1}\varepsilon,\;\nu_{0}^{N}( \hspace{0.4pt} [ 0,\varepsilon ]\hspace{0.4pt} )\leq{\textstyle \frac{\lambda^{\prime}}{2}}\alpha^{-1}\varepsilon\bigr)+\mathbb{P}\bigl(\nu_{0}^{N}(\hspace{0.4pt} [ 0,\varepsilon ] \hspace{0.4pt} )\geq{\textstyle \frac{\lambda^{\prime}}{2}}\alpha^{-1}\varepsilon\bigr),
		\]
		for each $N\geq 1$, where we have defined the initial empirical measures
		\[ \nu_{0}^{N}:=\frac{1}{N}\sum_{i=1}^{N}\delta_{X_{0}^{i,N}}.
		\]
		Observe that, for all $\varepsilon>0$ sufficiently small, we must have
		\[
			\lim_{N\rightarrow\infty}\mathbb{P}\bigl(\nu_{0}^{N}(\hspace{0.4pt} [0,\varepsilon] \hspace{0.4pt} )\geq{\textstyle \frac{\lambda^{\prime}}{2}}\alpha^{-1}\varepsilon\bigr) = \mathbf{1}_{\{\nu_{0}( \hspace{0.4pt} [ 0,\varepsilon] \hspace{0.4pt} )\geq\frac{\lambda^{\prime}}{2}\alpha^{-1}\varepsilon\}}=0,
		\]
		by the law of large numbers and Assumption \ref{Assumptions}.
		Note also that $\nu_{0}^{N}(\hspace{0.4pt} [ 0,\varepsilon]\hspace{0.4pt} )\leq\frac{\lambda^{\prime}}{2}\alpha^{-1}\varepsilon$
		if and only if $\#\{i:X_{0}^{i,N}\in [0,\varepsilon] \}\leq N\frac{\lambda^{\prime}}{2}\alpha^{-1}\varepsilon$.
		Thus, letting $\mathcal{I}$ denote the random index set 
		\[
		\mathcal{I}=\{i\in\{1,\ldots,N\}:X_{0}^{i,N}\geq\varepsilon\},
		\]
		for our $\varepsilon=\varepsilon(\delta)>0$, we get the estimate
	\[
		\mathbb{P}\bigl(L_{\delta}^{N}\geq \alpha^{-1}\varepsilon,\,\nu_{0}^{N}( \hspace{0.4pt} [ 0,\varepsilon ] \hspace{0.4pt} )\leq{\textstyle \frac{\lambda^{\prime}}{2}}\alpha^{-1}\varepsilon\bigr)\leq \sum_{\left|\mathcal{I}_{0}\right|\geq N(1-\frac{\lambda^{\prime}\varepsilon}{2\alpha})}\!\mathbb{P}(L_{\delta}^{N}\geq\alpha^{-1}\varepsilon\mid\mathcal{I}=\mathcal{I}_{0})\mathbb{P}(\mathcal{I}=\mathcal{I}_{0}).
	\]
		
		\textbf{Step 2. }Next we define
		\[
		\tau:=\inf\{t\geq0:L_{t}^{N}\geq{\textstyle \frac{\lambda}{2}}\alpha^{-1}\varepsilon\} \quad\text{and}\quad\mathcal{J}:=\bigl\{ i:\inf_{s<\tau\land\delta}(X_{s}^{i,N}-X_{0}^{i,N})\leq-{\textstyle \frac{1}{2}}\varepsilon\bigr\}.
		\]
		Then we claim that, for every $\mathcal{I}_{0}$ with $|\mathcal{I}_{0}|\geq N(1-\frac{\lambda^{\prime}}{2}\alpha^{-1}\varepsilon)$,
		we have
		\begin{equation}
		\label{eq:loss_inc_2}
		\mathbb{P}(L_{\delta}^{N}\geq
		\alpha^{-1}\varepsilon\mid\mathcal{I}=\mathcal{I}_{0})\leq\mathbb{P}\bigl(|\mathcal{I}_{0}\cap\mathcal{J}|\geq N{\textstyle \frac{\lambda-\lambda^{\prime}}{2}}\alpha^{-1}\varepsilon\mid\mathcal{I}=\mathcal{I}_{0}\bigr).
		\end{equation}
		Suppose, for a contradiction, that $L_{\delta}^{N}\geq \alpha^{-1}\varepsilon$
		but $|\mathcal{I}_{0}\cap\mathcal{J}|<N{\textstyle \frac{\lambda-\lambda^{\prime}}{2}}\alpha^{-1}\varepsilon$.
		Then	
		\[
		|\mathcal{I}_{0}\cap\mathcal{J}|+|\mathcal{I}_{0}^{\complement}|<N{\textstyle \frac{\lambda-\lambda^{\prime}}{2}}\alpha^{-1}\varepsilon+N{\textstyle \frac{\lambda^{\prime}}{2}}\alpha^{-1}\varepsilon=N{\textstyle \frac{\lambda}{2}}\alpha^{-1}\varepsilon,
		\]
		so even if all these particles are killed in the interval $[0,\tau\land\delta]$,
		their total contribution to $L_{\tau\land\delta}^{N}$ is
		less than ${\textstyle \frac{\lambda}{2}}\alpha^{-1}\varepsilon$.
		In particular, this means that, at time $\tau\land\delta$,
		the largest possible downward jump for the remaining particles (indexed
		by $\mathcal{I}_{0}\cap\mathcal{J^{\complement}}$) is $-{\textstyle \frac{\lambda}{2}}\varepsilon$.
		Notice, however, that these remaining particles all satisfy $X_{0}^{j,N}\geq\varepsilon$
		and \[
		\inf_{s<\tau\land\delta}(X_{s}^{j,N}-X_{0}^{j,N})>-{\textstyle \frac{1}{2}}\varepsilon,
		\]
		so we must have $\inf_{s<\tau\land\delta}X_{s}^{j,N}>\frac{1}{2}\varepsilon$.
		Consequently, a translation by $-{\textstyle \frac{\lambda}{2}}\varepsilon$
		cannot cause any of them to drop below the origin and hence $L_{\tau\land\delta}^{N}<{\textstyle \frac{\lambda}{2}}\alpha^{-1}\varepsilon$, so $\delta<\tau$, which certainly implies $L_{\delta}^{N}<\alpha^{-1}\varepsilon$. This yields the desired contradiction.
		
		\textbf{Step 3. }Note that $\alpha L_{s}^{N}<{\textstyle \frac{\lambda}{2}}\varepsilon$
		for $s<\tau\land\delta$. Therefore, using the linear growth assumption on the drift, we deduce that $\inf_{s<\tau\land\delta}(X_{s}^{i,N}-X_{0}^{i,N})\leq-{\textstyle \frac{1}{2}}\varepsilon$
		implies
		\begin{equation}\label{eq:particles_in_J}
		-\delta C(1+\sup_{s<\delta}|X_{s}^{i,N}|)+\inf_{s<\delta}\{ Y_{s}^{i}+Y_{s}^{0}\} -{\textstyle\frac{\lambda}{2}}\varepsilon\leq-{\textstyle\frac{1}{2}}\varepsilon,
		\end{equation}
		where
		\[
		Y_{s}^{0}:=\int_{0}^{s}\sigma(r)\rho(r)dB_{r}^{0}\quad \text{and} \quad Y_{s}^{i}:=\int_{0}^{s}\sigma(r)\sqrt{1-\rho(r)^{2}}dB_{r}^{i}.
		\]
		From here, we split the analysis on the intersection of the events
		\[
		\Big\{\sup_{s<\delta}|Y_{s}^{0}|\leq \delta^{\frac{1}{2}} \log\log(\delta^{-1}) \Big\} \quad \text{and} \quad \Big\{\sup_{s<\delta}|X_{s}^{i,N}|\leq  \delta^{-\frac{1}{2}} \log\log(\delta^{-1})\Big\},
		\]
		and their complements. It is easily verified that both complements occur with probability $o(1)$ as $\delta\downarrow0$, uniformly in $N\geq1$. On the intersection of the two events, it follows from the previous observation (\ref{eq:particles_in_J}) that
		\[\inf_{s<\delta}Y_s^i \leq -\frac{1}{2}\varepsilon+\frac{\lambda}{2}\varepsilon+\delta C+2\delta^{\frac{1}{2}}\log\log(\delta^{-1})
		\] for all $i\in \mathcal{J}$. 
 	Hence we can conclude from the final estimate in Step 1 and the claim (\ref{eq:loss_inc_2}) in Step 2 that
		\[
		\mathbb{P}(L_{\delta}^{N}\geq\alpha^{-1}\varepsilon)
			\leq
				\mathbb{P}\Bigl(\frac{1}{N}\sum_{i=1}^{N}\mathbf{1}_{\bigl\{\inf_{s<\delta}Y_{s}^{i}\leq-{\textstyle \frac{1-\lambda}{2}}\varepsilon+\delta C+2\delta^{\frac{1}{2}}\log\log(\delta^{-1})\bigr\}}\geq{\textstyle \frac{\lambda-\lambda^{\prime}}{2}}\alpha^{-1}\varepsilon\Bigr) + o(1)
		\]
		as $N\uparrow \infty$ and $\delta\downarrow0$, where the $o(1)$ terms in $\delta$ are uniform in $N \geq 1$. Recall that $\varepsilon=\delta^{\frac{1}{2}}\log(\delta^{-1})$. Since $\rho$ is bounded away from $1$ (by~Assumption
		\ref{Assumptions}), we can perform a time-change in each $Y^{i}$, and then the law of large numbers yields
		\[
		\limsup_{N\rightarrow\infty}\mathbb{P}(L_{\delta}^{N}
			\geq
				\alpha^{-1}\varepsilon)\leq\mathbf{1}_{\bigl\{{\textstyle \Phi}\bigl(-c[\frac{1-\lambda}{2}\log(\delta^{-1})-\delta^{\frac{1}{2}}C-2\log\log(\delta^{-1})]\bigr)\geq \frac{\lambda-\lambda^{\prime}}{2}\alpha^{-1}\log(\delta^{-1})\delta^{\frac{1}{2}}\bigr\}}+o(1)
		\]
	 as $\delta\downarrow0$, for some $c>0$,
		where $\Phi$ is the standard normal cdf.~Noting that $\Phi$ evaluated at the given point is of order $O(\exp\{ - \mathrm{const.} \times (\log \delta^{-1})^2 \})$, the indicator eventually becomes zero as $\delta \downarrow 0$ and thus the proof is complete.
	\end{proof}

The minimality constraint (\ref{eq:BU_PhysicalCondition}) will be recovered from the following lemma, together with Proposition \ref{Existence_Prop_MinimalityForCommonNoise}. The proof follows by essentially the same reasoning as in Proposition 5.3 of \cite{DIRT_SPA}, but we provide the full details for completeness. The main point is that, unlike the previous lemma, there is no need to isolate the effects of the common and idiosyncratic noise.

\begin{lem}
	\label{Minimal_jumps_particle}
	Fix any $t<T$. There is
	a constant $C>0$ such that, for every $\delta>0$ sufficiently small, we have
	\[
		\mathbb{P}\bigl(\boldsymbol{\nu}_{t-}^{N}(\hspace{0.4pt} [0,\alpha x+C\delta^{\frac{1}{3}}]\hspace{0.4pt} ) \geq x \ \ \forall x\leq L_{t+\delta}^{N}-L_{t-}^{N}-C\delta^{\frac{1}{3}}\bigr)\geq1-C\delta,
	\]
whenever $N\geq\delta^{-\frac{1}{3}}$.
\end{lem}

\begin{proof}
	Fixing $\delta>0$ and $N\geq\delta^{-\frac{1}{3}}$, we note that $N(L_{t+\delta}^{N}-L_{t-}^{N})$ equals the number of particles
	that are absorbed in the interval $[t,t+\delta]$. Let
	\[
	Y_{t,s}^{i}:=\int_{t}^{s}\sigma(r)\rho(r)dB_{r}^{0}+\int_{t}^{s}\sigma(r)\sqrt{1-\rho(r)^{2}}dB_{r}^{i},
	\]
	for $i=1,\ldots,N$, and define the events
	\[
	E_{1}^{i,k}:=\Bigl\{ X_{t-}^{i,N}-\alpha\frac{k}{N}-C\delta(1+\sup_{s\leq t+\delta}|X_{s}^{i}|)-\sup_{s\in[t,t+\delta]}|Y_{t,s}^{i}|\leq0,\;t\leq\tau^{i}\Bigr\},
	\]
	for a given $k$. Then it must be the case that
	\begin{equation}
	\frac{1}{N}\sum_{i=1}^{N}\mathbf{1}_{E_{1}^{i,k}}\geq\frac{k}{N}\quad\text{for}\quad k=0,1,\ldots,N(L_{t+\delta}^{N}-L_{t-}^{N}).\label{eq: E^1,i}
	\end{equation}
	Now fix an arbitrary $\lambda_{0}\leq L_{t+\delta}^{N}-L_{t-}^{N}-2\delta^{\frac{1}{3}}$
	and set
	\[
	k_{0}:=\bigl\lfloor N(\lambda_{0}+2\delta^{\frac{1}{3}})\bigr\rfloor\leq N(L_{t+\delta}^{N}-L_{t-}^{N}).
	\]
	This way, (\ref{eq: E^1,i}) applies to $k_{0}$ and we have $\lambda_{0}\geq\frac{k_{0}}{N}-2\delta^{\frac{1}{3}}$ as well as $\frac{k_0}{N}\geq\lambda_0+2\delta^{\frac{1}{3}}-\frac{1}{N}$.
	Introducing the additional events
	\[
	E_{2}^{i}:=\bigl\{C\delta(1+\sup_{s\leq t+\delta}|X_{s}^{i,N}|)+\sup_{s\in[t,t+\delta]}|Y^i_{t,s}|\geq\delta^{\frac{1}{3}}\bigr\},
	\]
	for $i=1,\ldots,N$, it follows that, on each event $E_{1}^{i,k_{0}}\cap(E_{2}^{i})^{\complement}$, we have
	$X_{t-}^{i}-\alpha\frac{k_{0}}{N}\leq\delta^{\frac{1}{3}}$,
	and hence 
	\[
	X_{t-}^{i}-\alpha\lambda_{0}\leq X_{t-}^{i}-\alpha \bigl(\frac{k_{0}}{N}-2\delta^{\frac{1}{3}} \bigr)\leq(1+2\alpha)\delta^{\frac{1}{3}}.
	\]
	Consequently,
	\begin{align*}
		\boldsymbol{\nu}_{t-}^{N}(\hspace{0.4pt} [ 0,\alpha\lambda_{0}+(1+2\alpha)\delta^{\frac{1}{3}} ] \hspace{0.4pt}) & =\frac{1}{N}\sum_{i=1}^{N}\mathbf{1}_{\{X_{t-}^{i}\leq\alpha\lambda_{0}+(1+2\alpha)\delta^{\frac{1}{3}},\;t\leq\tau^{i}\}}\geq\frac{1}{N}\sum_{i=1}^{N}\mathbf{1}_{E_{1}^{i,k_{0}}}\mathbf{1}_{(E_{2}^{i})^{\complement}}. 
	\end{align*}
	Defining the event $E:=\{\frac{1}{N}\sum_{i=1}^{N}\mathbf{1}_{E^i_2}\leq\delta^{\frac{1}{3}}\}$,
	we deduce that, on $E$,
	\begin{align*}
		\boldsymbol{\nu}_{t-}^{N}(\hspace{0.4pt} [ 0,\alpha\lambda_{0}+(1+2\alpha)\delta^{\frac{1}{3}} ] \hspace{0.4pt}) & \geq\frac{1}{N}\sum_{i=1}^{N}\mathbf{1}_{E_{1}^{i,k_{0}}}-\frac{1}{N}\sum_{i=1}^{N}\mathbf{1}_{E_{2}^{i}}\geq\frac{k_{0}}{N}-\delta^{\frac{1}{3}}\\
	& \geq \bigl(\lambda_{0}+2\delta^{\frac{1}{3}}-\frac{1}{N} \bigr)-\delta^{\frac{1}{3}}=\lambda_{0}+\delta^{\frac{1}{3}}-\frac{1}{N}.
	\end{align*}
	Since we are working with $N\geq\delta^{-\frac{1}{3}}$, we have $\delta^{\frac{1}{3}}-N^{-1}\geq0$,
	so we can finally conclude that
	\begin{equation}
		\boldsymbol{\nu}_{t-}^{N}(\hspace{0.4pt} [ 0,\alpha\lambda_{0}+(1+2\alpha)\delta^{\frac{1}{3}} ]\hspace{0.4pt})\geq\lambda_{0}\qquad\text{on } E,
	\label{eq:on_E}
	\end{equation}
	for our arbitrary choice of $\lambda_{0}\leq L_{t+\delta}^{N}-L_{t-}^{N}-2\delta^{\frac{1}{3}}$.
	Thus, it only remains to observe that there exists a constant $C>0$, independent of $\delta$,
	such that $\mathbb{P}(E^{\complement})\leq C\delta$. To see this,
	we can apply Markov's inequality twice, the Cauchy--Schwarz inequality, and the Burkholder--Davis--Gundy inequality (to $Y^i$) in order to find that
	\begin{align*}
	\mathbb{P}(E^{\complement}) & =\mathbb{P}\Bigl(\frac{1}{N}\sum_{i=1}^{N}\mathbf{1}_{E^i_2}>\delta^{\frac{1}{3}}\Bigr)\leq\delta^{-\frac{1}{3}}\frac{1}{N}\sum_{i=1}^{N}\mathbb{P}(E_{2}^{i})\\
	& \leq c\delta^{5}\mathbb{E}\Bigl[\bigl(1+\sup_{s\leq T}|X_{s}^{1,N}|\bigr)^{8}\Bigr]+c\delta^{-3}\mathbb{E}\Bigl[\sup_{s\in[t,t+\delta]}|Y^1_{t,s}|^{8}\Bigr]
	\leq c_1\delta^{5}+c_1\delta.
	\end{align*}
	Recalling (\ref{eq:on_E}) and taking $C:=\max\{ 2c_1,1+2\alpha\}$ finishes the proof.
\end{proof}

Finally, we note that the convergence of the loss processes will essentially come down to the following property of each particle, highlighting that the Brownian drivers cause the particles to dip strictly below any level immediately after hitting it.
\begin{lem}
	\label{prop:Crossing Prop} Let $X^{1,N}$ be given by \eqref{eq:Particle_System}. For any $t\leq T$ and $\delta>0$,
	we have
	\[	\lim_{\varepsilon\rightarrow0}\limsup_{N\rightarrow\infty}\mathbb{P}\Bigl(\,\inf_{s\leq\delta\land (T-t)}\bigl\{ X_{t+s}^{1,N}-X_{t}^{1,N}\bigr\}>-\varepsilon\Bigr)=0.
	\]
\end{lem}

\begin{proof}
Since $L^{N}$ can only cause the particles to jump down, we can simply neglect it in the dynamics. Using also the linear growth bound on the drift from Assumption \ref{Assumptions}, we can therefore deduce that probability in question is controlled by
	\begin{equation}\label{eq:crossing_estimate}
	\mathbb{P}\Bigl(\,\inf_{s\leq\delta\land(T-t)}\bigl\{ sC \lambda+Y^1_{s+t}-Y^1_{t} \bigr\}>-\varepsilon \Bigr)+\mathbb{P}\bigl(\sup_{s\leq T}|X_{s}^{1,N}|>\lambda \bigr),
	\end{equation}
	for any given $\lambda>0$, where we have set
	\[
	Y_s^1:=\int_{0}^{s} \sigma(r) \sqrt{1-\rho(r)^2} dB_r^1 +  \int_{0}^{s} \sigma(r)\rho(r) dB_r^0.
	\]
	Since $Y^1$ has the law of a time-changed Brownian motion, the first term in \eqref{eq:crossing_estimate} is of order $o(1)$ as $\varepsilon \rightarrow 0$, uniformly in $N \geq1$, for any fixed $\lambda>0$. Furthermore, sending $\lambda \rightarrow \infty$, the second term in \eqref{eq:crossing_estimate} vanishes uniformly in $N\geq 1$, and hence the result follows.
\end{proof}

\subsection{Compactness, Continuity, and Convergence}\label{Sec:Continuity}

To establish the convergence of the finite particle system, we follow the ideas of \cite{DIRT_SPA}. In particular, we extend the particles
from (\ref{eq:Particle_System}) to $[0,\bar{T}]$, for a fixed $\bar{T}>T$, by adding purely Brownian noise on $(T,\bar{T}]$. This amounts to replacing the empirical measures $\mathbf{P}^N$ by
\begin{equation}
\label{eq:empirical_measures}
\tilde{\mathbf{P}}^{N}:=\frac{1}{N}\sum_{i=1}^{N}\delta_{\tilde{X}_{\cdot}^{i,N}}\quad\text{where}\quad\tilde{X}_{t}^{i,N}:=\begin{cases}
X_{t}^{i,N}, & t\in[0,T]\\
X_{T}^{i,N}+B_{t}^{i}-B_{T}^{i}, & t\in(T,\bar{T}].
\end{cases}
\end{equation}
For simplicity of notation, we will drop the `$\sim$' and simply write $X^{i,N}_{t}$ and $\mathbf{P}^{N}$, with the understanding
that $X_{t}^{i,N}$ is given by $\tilde{X}_{t}^{i,N}$ when we
are working on the full interval $[0,\bar{T}]$. Notice that this
construction automatically extends Proposition \ref{prop: Loss process nice at 0}
to hold at the endpoint $\bar{T}$, and the choice of Brownian
noise on $(T,\bar{T}]$ ensures the validity of Lemma \ref{prop:Crossing Prop}
with $\bar{T}$ in place of $T$ (of course, we are only actually interested in the dynamics of the system up to time $T$).

As in \cite[Sec.~4.1]{DIRT_SPA}, we endow $D_{\mathbb{R}}=D_{\mathbb{R}}[0,\bar{T}]$ with Skorokhod's M1 topology because it will allow us to circumvent irregularities in the loss, $L^N$, by virtue of its monotonicity. For properties of this topology we will be referring to \cite{ledger_2016, whitt} (in \cite{whitt} it is called the strong M1 topology and denoted $SM_1$). Importantly, these properties rely on the members of $D_{\mathbb{R}}$ being left-continuous at the terminal time, which is the reason for the continuous extension of the particle system to $\bar{T}$. Furthermore, we emphasise that $(D_{\mathbb{R}},\text{M1})$ is a Polish space \cite[Thm.~12.8.1]{whitt} and that its Borel sigma algebra is generated by the marginal projections \cite[Thm.~11.5.2]{whitt}.

Similarly to the analysis in \cite[Lemmas 5.4 and 5.5]{DIRT_SPA}, once we have a result such as Proposition \ref{prop: Loss process nice at 0} at both endpoints $t=0$ and $t=\bar{T}$, the tightness of the empirical measures becomes an easy consequence of the properties of the M1 topology (see Proposition \ref{lem:tightness}). From Prokhorov's theorem \cite[Thm.~5.1]{billingsley}, we thus obtain a weakly convergent subsequence of the pair $(\mathbf{P}^N,B^0)$, and the goal is then to relate the resulting limit point to a solution of (\ref{eq:Relaxed McKean-Vlasov}).

Before seeking to characterize the limit points $(\mathbf{P}^*,B^0)$, our first task is to ensure that $L^N=\mathbf{P}^N(t\geq\hat{\tau})$ converges to $\mathbf{P}^*(t\geq\hat{\tau})$ whenever $\mathbf{P}^N$ converges $\mathbf{P}^*$. This is achieved through a continuity result (namely Lemma \ref{prop:loss_continuity}), which is an analogue of \cite[Lemma 5.6, Prop.~5.8, Lemma 5.9]{DIRT_SPA}, and its proof follows by similar arguments after observing that it suffices to rely on Lemma \ref{prop:Crossing Prop}.

Finally, we need to characterize $\mathbf{P}^*$ as the conditional law of a process $X$ satisfying the desired conditional McKean--Vlasov equation such that $(\mathbf{P}^*,B^0)$ is independent of the idiosyncratic Brownian motion $B$ (which is constructed as part of the solution). For this part, we rely on a martingale argument (Lemma 3.11 and Proposition 3.12) which---although close in spirit---differs from the approach in \cite{DIRT_SPA} and is useful for dealing with the common noise. Compared to \cite{DIRT_SPA}, another advantage of this argument is that it is not sensitive to the specific form of the volatility, $\sigma$, whereas the approach in \cite{DIRT_SPA} is tailored to a constant volatility. Aside from our approach to the independence result for $\mathbf{P}^*$ and $B$ in Lemma 3.13, all our arguments extend to a bounded and non-degenerate Lipschitz continuous $x\mapsto\sigma(t,x)$, so without the common noise we easily obtain such a generalization of the results in \cite{DIRT_SPA}, but we leave out the details of this. Instead we push ahead and implement the plan outlined above, starting with the tightness of the empirical measures in the $\text{M1}$ topology. 

\begin{prop}[Tightness of the empirical measures]\label{lem:tightness} Let $\mathfrak{T}^{wk}_{\emph{M1}}$ denote the topology of weak convergence on $\mathcal{P}(D_{\mathbb{R}})$ induced by the $\emph{M1}$ topology on $D_{\mathbb{R}}$. Then the empirical measures
	$(\mathbf{P}^{N})_{N\geq1}$ are tight on $(\mathcal{P}(D_{\mathbb{R}}),\mathfrak{T}_{\emph{M1}}^{wk})$ under Assumption \ref{Assumptions}.
\end{prop}

\begin{proof}
	Since $(D_{\mathbb{R}},\text{M1})$ is a Polish space, so is $(\mathcal{P}(D_{\mathbb{R}}),\mathfrak{T}_{\text{M1}}^{wk})$. Therefore, by a classical result \cite[Ch.I, Prop.~2.2]{sznitman}, it suffices to verify that $(X^{1,N})_{N\geq1}$ is tight on $(D_{\mathbb{R}},\text{M1})$. Let 
	\[
	H_{\mathbb{R}}(x_1,x_2,x_3):=\inf_{\lambda \in [0,1]}|x_2-(1-\lambda)x_1+\lambda x_3|.
	\]
Following \cite[Sec.~4]{ledger_2016}, we can deduce the tightness of $(X^{1,N})_{N\geq1}$ by showing that
	\begin{equation}
	\label{eq:condition_1}
\mathbb{P}\bigl(H_{\mathbb{R}}(X_{t_1}^{1,N},X_{t_2}^{1,N},X_{t_3}^{1,N})\geq \delta \bigr) \leq C \delta^{-4} |t_3-t_1|^{2}	
	\end{equation}
	for all $N\geq1$, $\delta>0$ and $0\leq t_1< t_2<t_3\leq \bar{T}$, along with
	\begin{equation}
	\label{eq:condition_2}
	\lim_{\delta\downarrow0}\lim_{N\rightarrow\infty}\mathbb{P}\Bigl(\sup_{t\in(0,\delta)}\bigl| X_t^{1,N}-X_0^{1,N} \bigr|+\sup_{t\in(\bar{T}-\delta,\bar{T})}\bigl| X_{\bar{T}}^{1,N}-X_t^{1,N} \bigr|\geq\varepsilon\Bigr)=0
	\end{equation}
	for every $\varepsilon>0$. For the first condition, observe that
	\[
	H_{\mathbb{R}}(X_{t_1}^{1,N},X_{t_2}^{1,N},X_{t_3}^{1,N}) \leq| Z_{t_1}-Z_{t_2}| + | Z_{t_2}-Z_{t_3}|+\inf_{\lambda\in[0,1]} | L_{t_2}^N -(1-\lambda)L_{t_1}^N -\lambda L_{t_3}^N|,
	\]
	where $Z$ is given by
	\[
	dZ_t= b(t,X_t^{1,N})dt + \sigma(t) \sqrt{1-\rho(t)^2}dB_t^{1}+\sigma(t) \rho(t)dB_t^{0}.
	\]
	Since $L^N$ is increasing, the final term on the right-hand side is zero and hence (\ref{eq:condition_1}) follows easily from Markov's inequality by controlling the increments of $Z$. Finally, (\ref{eq:condition_2}) holds by virtue of Proposition \ref{prop: Loss process nice at 0} and the fact that it also applies at the artificial endpoint $\bar{T}$, as remarked above.
\end{proof}

 Recall that each $\mathbf{P}^{N}$ is a random probability
 measure on $(D_{\mathbb{R}},\mathcal{B}(D_{\mathbb{R}}))$, and that $L_t^N=\mathbf{P}^{N}(t\geq\hat{\tau})$,
 where $\hat{\tau}$ is the first hitting time of zero for the canonical process on $D_{\mathbb{R}}$. That is, written as a random variable on $D_{\mathbb{R}}$, we have
 \[
 \hat{\tau}(\zeta)=\inf \{ t>0 : \zeta_t \leq 0 \}, \qquad \text{for every } \zeta\in D_{\mathbb{R}}.
 \]
 The next lemma provides a continuity result for $L^N$ with respect to limit points of $\mathbf{P}^{N}$. It will play a central role in our further weak convergence analysis.
\begin{lem}
	[Continuity result for the loss]\label{prop:loss_continuity}Suppose $\mathbf{P}^{N}$
	\negthinspace{}converges weakly to $\mathbf{P}^{*}$ on $(\mathcal{P}(D_{\mathbb{R}}),\mathfrak{T}_{\text{\emph{M1}}}^{wk})$
	and let $\mathbb{Q}^{*}:=\text{\emph{Law}}(\mathbf{P}^{*})$. At $\mathbb{Q}^{*}$-a.e.~$\mu$
	in $\mathcal{P}(D_{\mathbb{R}})$, the mapping $\mu\mapsto\mu(t\geq\hat{\tau})$
	is continuous with respect to $\mathfrak{T}_{\text{\emph{M1}}}^{wk}$,
	for all $t\in\mathbb{T}^{\mu}$, where $\mathbb{T}^\mu$
	is the set of continuity points of $t\mapsto\mu(t\geq\hat{\tau})$.
\end{lem}

\begin{proof}We set
	\[
	Y_{t}:=\int_{0}^{t}\sigma(s)\sqrt{1-\rho(s)^{2}}dB_{s}^{1}+\int_{0}^{t}\sigma(s)\rho(s)dB_{s}^{0},\quad t\geq0,
	\]
and introduce the joint laws
	\begin{equation}\label{eq:joint_laws}
	P_{X,Y}^{1,N}:=\text{Law}(X^{1,N},Y^{1}),
	\end{equation}
	for $N\geq 1$. By the proof of Proposition \ref{lem:tightness}, the sequence $(X^{1,N})_{N\geq 1}$ is tight on $(D_{\mathbb{R}},\text{M1})$, so the joint laws \eqref{eq:joint_laws} are also tight. Passing to a subsequence, still indexed by $N$, we thus have  $\mathbf{{P}}^{N}\Rightarrow\mathbf{{P}}^{*}$ and, in the sense of weak convergence of measures, $P_{X,Y}^{1,N}\rightarrow P_{X,Y}^{1,*}$
	for a limit point $P_{X,Y}^{1,*}\in\mathcal{{P}}(D_{\mathbb{{R}}}\times C_{\mathbb{{R}}})$.
Let $P_{X}^{1,*}$ and $P_{Y}^{1,*}$ denote the marginal laws. Since the particles $X^{i,N}$, for $i=1,\ldots,N$, have the same law, we get
	$\mathbb{{E}}[\langle\mathbf{{P}}^{N},f\rangle]=\langle P_{X,Y}^{1,N},f\rangle$ 
	for any bounded continuous $f:(D_{\mathbb{R}},\text{M1}) \rightarrow\mathbb{{R}}$.
	For any such $f$, the linear functional $\mu\mapsto\langle\mu,f\rangle$ is continuous on $\mathcal{
	P}(D_{\mathbb{R}})$ for the induced topology of weak convergence, so the convergence of $\mathbf{{P}}^{N}$ and $P_{X,Y}^{1,N}$, and uniqueness of limits, allows us to conclude that
	\[
	\mathbb{{E}}[\mathbf{{P}}^{*}(E)]=P_{X}^{1,*}(E)=P_{X,Y}^{1,*}(E\times C_{\mathbb{{R}}})
	\]
	for all events $E\in\mathcal{{B}}(D_{\mathbb{{R}}})$. Noting that the set of times
	\[
	\mathbb{{T}}:=\{t\in[0,\bar{T}]:\mathbb{{E}}[\mathbf{{P}}^{*}(\eta_{t}=\eta_{t-})]=1\}
	\]
	is co-countable, we can consider the Borel set $E^*\in\mathcal{{B}}(D_{\mathbb{{R}}}\times C_{\mathbb{{R}}})$
	given by
	\[
	E^*:=\bigcap_{q\leq r,\,q,r\in\mathbb{{Q}}\cap\mathbb{{T}}}\{(\eta,y):\eta_{r}-\eta_{q}\leq y_{r}-y_{q}+C\sup_{s\leq \bar{T}} |\eta_s|(r-q) \},
	\]
	where $C>0$ is the Lipschitz constant of the drift $b$. Since the marginal projections and the running supremum are continuous for the M1 topology at continuity
	points, we see that $E^*$ is closed with probability $1$
	under $P_{X,Y}^{1,*}$. Therefore, using the upper semi-continuity of $(\eta,w)\mapsto\mathbf{{1}}_{E^*}(\eta,w)$
	with probability $1$ under $P_{X,Y}^{1,*}$, we conclude that
	
	\[
	P_{X,Y}^{1,*}(E^*)\geq\limsup_{N\rightarrow\infty}P_{X,Y}^{1,N}(E^*)=1.
	\]
Write $(Z^{*},Y^{*})(\eta,w)=(\eta,w)$. Since $(Z^{*},Y^{*})$ belongs to $E^*$ with probability 1 under $P_{X,Y}^{1,*}$, the right-cotinuity of $Z^{*}$ and continuity
	of $Y^*$ gives that, with probability 1, we have $Z^*_{t}-Z^*_{s}\leq Y^*_{t}-Y^*_{s}+C\sup_{r\leq \bar{T}} |Z^*_r|(t-s)$
	for all $s\leq t$ in $[0,\bar{T}]$. Moreover, weak convergence of the marginals implies that, under $P_{X,Y}^{1,*}$, the proccess $Y^*$ has the same law as $Y^{1}$. Using again the weak convergence, we thus deduce that, with respect to the filtration generated by the pair $(Z^{*},Y^{*})$, $Y^{*}$ is a Brownian motion time-changed by $t\mapsto \int_0^t \sigma(s)^2ds$. In particular, the restarted process
	$ Y^*_{\cdot +\tau(Z^*)}-Y^*_{\tau(Z^*)}$
	also has the law of a time-changed Brownian. Applying e.g.~the law of the iterated logarithm to $Y^*_{\cdot +\tau(Z^*)}-Y^*_{\tau(Z^*)}$, we readily deduce from the definition of the M1 topology that
	\[
	\mathbb{{E}}[\mathbf{{P}}^{*}(\{\eta:\hat{\tau}\;\text{continuous at}\;\eta\})]=P_{X}^{1,*}(\{\eta:\hat{\tau}\;\text{continuous at}\;\eta\})=1.
	\]
	Thus, $\hat{\tau}$ is almost everywhere continuous under almost all realisations of $\mathbf{P}^*$.
	
	Now suppose $\mu^{n}\rightarrow\mu$, where $\mu$ is in the support	of $\mathbb{{Q}}^{*}$. By the above, it holds for $\mathbb{{Q}}^{*}$-a.e.~$\mu$ that $\hat{\tau}$ is continuous at $\mu$-almost
	every $\eta$. From here on, we thus restrict to $\mu$ with this property. Using Skorokhod's representation theorem, we can write
	\[
	\mu^{n}(t\geq\hat{{\tau}})=\mathbb{{P}}(t\geq\hat{{\tau}}(Z^{n}))\quad\text{and}\quad\mu(t\geq\hat{{\tau}})=\mathbb{{P}}(t\geq\hat{{\tau}}(Z)),
	\]
	where we have almost sure M1 convergence $Z^{n}\rightarrow Z$ and $\hat{{\tau}}$ is continuous at almost all paths of $Z$. Take $t\in\mathbb{{T}}^{\mu}$. By definition of $\mathbb{{T}}^{\mu}$,  dominated convergence then gives
	\[
	\mathbb{{P}}(t=\hat{{\tau}}(Z))=\mu(t\geq\hat{{\tau}})-\lim_{s\uparrow t}\mu(s\geq\hat{{\tau}})=0,
	\]
	and hence it holds for $\mathbb{{P}}$-almost all $\omega\in\Omega$ that
	$\mathbf{{1}}_{t\geq\hat{{\tau}}(Z^{n}(\omega))}\rightarrow\mathbf{{1}}_{t\geq\hat{{\tau}}(Z(\omega))}$,
	since this convergence can only fail if $t=\hat{{\tau}}(Z(\omega))$. Consequently, another application of dominated convergence yields
	\[
	\mu^{n}(t\geq\hat{{\tau}})= \mathbb{{E}}[\mathbf{{1}}_{t\geq\hat{{\tau}}(Z^{n})}]\rightarrow\mathbb{{E}}[\mathbf{{1}}_{t\geq\hat{{\tau}}(Z)}]= \mu(t\geq\hat{{\tau}}),
	\]
which completes the proof.
\end{proof}
Using the tightness of $\mathbf{P}^{N}$ from Proposition \ref{lem:tightness}, we can fix a limit point $\mathbb{P}^{*}$ of $\mathbb{P}^{N}:=\text{Law}(\mathbf{P}^{N},B^{0})$
on $\mathcal{P}(D_{\mathbb{R}})\times C_{\mathbb{R}}$. From now on, we shall work with this particular limit point (keeping in mind that the arguments apply to any limit point) and, despite having passed to a subsequence, we shall simply write $\mathbb{P}^{N}\Rightarrow\mathbb{P}^{*}$. For concreteness,
we define $\Omega^{*}:=\mathcal{P}(D_{\mathbb{R}})\times C_{\mathbb{R}}$
and introduce the random variables $\mathbf{P}^{*}(\mu,w):=\mu$
and $B^{0}(\mu,w):=w$ on the background space $(\Omega^{*},\mathbb{P}^{*},\mathcal{B}(\Omega^{*}))$.
Note that the joint law of $(\mathbf{P}^{*},B^{0})$ is $\mathbb{P}^{*}$
with $\mathcal{B}(\Omega^{*})=\sigma(\mathbf{P}^{*},B^{0})$.  Given this, we define $L^{*}:=\mathbf{P}^{*}(\, \cdot  \geq\hat{\tau})$ along with the co-countable set of times
\begin{equation}\label{blackboard_t}
\mathbb{T}:=\bigl\{ t\in[0,\bar{T}]:\mathbb{P}^{*}(L_{t}^{*}=L_{t-}^{*})=1,\:\mathbb{E}^{*}[\mathbf{P}^{*}(\hat{\eta}_{t}=\hat{\eta}_{t-})]=1\bigr\}.
\end{equation}

By Proposition \ref{prop: Loss process nice at 0} and the proof of Proposition \ref{lem:tightness}, we knew from the outset that $L^N$ is tight in $D_\mathbb{R}$ on $[0,\bar{T}]$. By Lemma \ref{prop:loss_continuity} and the continuous mapping theorem \cite[Thm.~2.7]{billingsley}, we can deduce that the finite dimensional distributions of its limit (along our fixed subsequence) agree with those of $L^*$, and hence we identify $L^*$ as the limit. Of course, we are actually interested in the dynamics on $[0,T]$, so we note that we also have convergence at the process level in $D_\mathbb{R}$ on $[0,T_0]$, for any $T_0\in[0,T]\cap\mathbb{T}$, since the restriction is continuous so long as $T_0$ is $\mathbb{P}^*$-almost surely a continuity point of $L^*$.

We now proceed to derive some further properties of the limiting law $\mathbb{P}^{*}$, which will ultimately enable us to construct a probability space that supports a solution to (\ref{eq:Relaxed McKean-Vlasov}). First of all, we define the map $\mathcal{M}:\mathcal{P}(D_{\mathbb{R}})\times D_{\mathbb{R}} \rightarrow D_{\mathbb{R}}$ by
\begin{equation}
\label{eq:mathcal_M}
\mathcal{M}(\mu,\eta):=\eta-\eta_{0}-\int_{0}^{\cdot}b(s,\eta_{s})ds-\alpha\mu(\, \cdot \geq\hat{\tau}).
\end{equation}
Fix an arbitrary choice of times $s_{0},t_{0}\in\mathbb{T}\cap[0,T)$ with $s_0<t_0$ and $s_{1},\ldots,s_{k}\in[0,s_{0}]\cap \mathbb{T}$, for $\mathbb{T}$ as defined in (\ref{blackboard_t}), and let $F:D_{\mathbb{R}}\rightarrow\mathbb{R}$ be given by
\begin{equation}\label{eq:F}
F(\zeta):=(\zeta_{t_{0}}-\zeta_{s_{0}})\prod_{i=1}^{k}f_{i}(\zeta_{s_{i}})
\end{equation}
for arbitrary $f_{1},\ldots,f_{k}\in C_{b}(\mathbb{R})$. Based on this, we define the functionals
\[
\begin{cases}
\Psi(\mu):=\bigl\langle\mu,F(\mathcal{M}(\mu, \boldsymbol{\cdot}\hspace{1pt}))\bigr\rangle, \\[3pt] \Upsilon(\mu):=\bigl\langle\mu,F\big(\mathcal{M}(\mu, \boldsymbol{\cdot}\hspace{1pt})^2 - {\textstyle\int_{0}^{\cdot}} \sigma(s)^2ds\big)\bigr\rangle, \;\; \text{and}\\[3pt] \Theta(\mu,w):=\bigl\langle\mu,F\big(\mathcal{M}(\mu, \boldsymbol{\cdot}\hspace{1pt})\times w -{\textstyle\int_{0}^{\cdot}} \sigma(s)\rho(s)ds\big)\bigr\rangle.
\end{cases}
\]
 As an application of Lemma \ref{prop:loss_continuity}, we obtain the following continuity result.
\begin{lem}[Functional continuity]
\label{lem:The-functional-continuity} For $\mathbb{P}^{*}$-almost every~$\mu$, we have $\Psi(\mu^n)$, $\Upsilon(\mu^n)$ and $\Theta(\mu^n,w^n)$ converging to $\Psi(\mu)$, $\Upsilon(\mu)$ and $\Theta(\mu,w)$ whenever $(\mu^{n},w^{n})\rightarrow(\mu,w)$ in $(\mathcal{P}(D_{\mathbb{R}}),\mathfrak{T}_{\emph{M1}}^{wk})\times (C_{\mathbb{R}},\left\Vert \cdot\right\Vert _{\infty})$ along a sequence for which $\langle\mu^{n},|\mathcal{M}_t(\mu^n,\boldsymbol{\cdot}\hspace{1pt})|^p\rangle$ is bounded uniformly in $n\geq1$, for some $p>2$, at any $t\geq0$.
\end{lem}

\begin{proof}
By definition of $\mathbb{T}$, it holds for $\mathbb{P}^{*}$-almost
every $\mu$ that $\mu(t\geq\hat{\tau})=\mu(t-\geq\hat{\tau})$ and
	$\mu(\eta_{t}=\eta_{t-})=1$ for all $t\in\mathbb{T}$. Appealing to Lemma \ref{prop:loss_continuity},
	we can thus restrict to a set of $\mu$'s with probability one under $\mathbb{P}^{*}$ such that $\mu^{n}(t\geq\hat{\tau})$
	converges to $\mu(t\geq\hat{\tau})$ for $t=t_0$ and $t=s_0,\ldots,s_k$. Fix any one of these $\mu$'s and suppose $(\mu^{n},w^{n})\rightarrow(\mu,w)$ in $(\mathcal{P}(D_{\mathbb{R}}),\mathfrak{T}_{\text{\text{M1}}}^{wk})\times (C_{\mathbb{R}},\left\Vert \cdot\right\Vert _{\infty})$ with $\mu^n$ satisfying the above integrability assumption. We start by showing that $\Psi(\mu^n)\rightarrow\Psi(\mu)$. Invoking  Skorokhod's representation theorem \cite[Thm.~6.7]{billingsley}, we can write
	\[
	\Psi(\mu^n)=\mathbb{E}[F(\mathcal{M}(\mu^n,Z^n))] \quad \text{and} \quad \Psi(\mu)=\mathbb{E}[F(\mathcal{M}(\mu,Z))],
	\]
	where $Z^n\rightarrow Z$ almost surely in $(D_\mathbb{R},\text{M1})$ with $\text{Law}(Z^n)=\mu^n$, $\text{Law}(Z)=\mu$, and $\mathbb{P}(Z_t=Z_{t-})$ for $t=t_0,s_0,\ldots,s_k$. 
	By standard properties of M1 convergence \cite[Thm.~12.4.1]{whitt}, we have pointwise convergence at the continuity points of $Z$ and, in particular, $Z^n_s\rightarrow Z_s$ for almost every $s\in[0,T]$. As $Z^n$ is a convergent and hence bounded sequence, we deduce from the continuity of $b(s,\cdot)$ and its linear growth bound that, almost surely,
\begin{equation}
\label{eq:Zn_pointwise}
Z^n_t-Z^n_0-\int_{0}^{t}b(s,Z^n_{s})ds \,\rightarrow\,	Z_t-Z_0-\int_{0}^{t}b(s,Z_{s})ds,\quad \text{for}\;\; t=t_0,s_0,\ldots,s_k,
\end{equation}
by dominated convergence. In turn, we conclude that $F(\mathcal{M}(\mu^n,Z^n))$ converges almost surely to $F(\mathcal{M}(\mu,Z))$ in $\mathbb{R}$. It remains to observe that $F(\mathcal{M}(\mu^n,Z^n))$ is uniformly integrable, since \[
|F(\mathcal{M}(\mu^n,Z^n)|\leq |\mathcal{M}_{t_0}(\mu^n,Z^n)|+|\mathcal{M}_{s_0}(\mu^n,Z^n)|,
\]
and $ \E[|\mathcal{M}_{t}(\mu^n,Z^n)|^p]$ equals $\langle \mu^{n},|\mathcal{M}_t(\mu^n,\boldsymbol{\cdot}\hspace{1pt})|^p\rangle $, which is uniformly bounded for some $p>1$, by assumption. Hence the convergence of $\Psi(\mu^n)$ to $\Psi(\mu^n)$ follows from Vitali's convergence theorem. The proof is analogous for $\Upsilon$ and $\Theta$, using the boundedness assumption for $p>2$.
\end{proof}


Define a new background space $\bar{\Omega} := \Omega^* \times D_{\mathbb{R}} = \mathcal{P}(D_{\mathbb{R}})\times C_{\mathbb{R}}\times D_{\mathbb{R}}$ equipped with its Borel sigma algebra $\mathcal{B}(\bar{\Omega})$ and the probability measure $\bar{\mathbb{P}}$ given by
\begin{equation}
\label{P_bar}
\bar{\mathbb{P}}(A):=\int_{\mathcal{P}(D_{\mathbb{R}})\times C_{\mathbb{R}}}\mu( \{ \eta:(\mu,w,\eta)\in A\})d\mathbb{P}^{*}(\mu,w), \qquad \forall A\in \mathcal{B}(\bar{\Omega}).
\end{equation}

For simplicity of presentation, we shall not distinguish notationally between random variables defined on $\Omega^*$ and their canonical extensions to $\bar{\Omega}$.

\begin{prop}
	\label{prop:martingales}
Let $\mathcal{M}$ be given by (\ref{eq:mathcal_M}). Then $\mathcal{M}_{\cdot}$, $\mathcal{M}_{\cdot}^2-{\textstyle\int_{0}^{\cdot}}\sigma(s)^2ds$ and $\mathcal{M}_{\cdot}\times B^0_{\cdot} -{\textstyle\int_{0}^{\cdot}}\sigma(s)\rho(s)ds$ are all continuous martingales on $(\bar{\Omega}, \bar{\mathbb{P}}, \mathcal{B}(\bar{\Omega}))$ under Assumption \ref{Assumptions}.	\end{prop}

\begin{proof}  Recall that $\mathbb{P}^N\Rightarrow\mathbb{P}^*$, where $\mathbb{P}^N$ is the law of $(\mathbf{P}^N,B^0)$. Using Skorokhod's representation theorem, we can find an almost surely convergent sequence $(\mathbf{Q}^N,B^N)\rightarrow(\mathbf{Q}^*,B^*)$ in $(\mathcal{P}(D_\mathbb{R}),\mathfrak{T}_{\text{\emph{M1}}}^{wk})\times (C_\mathbb{R},\Vert\cdot\Vert_{\infty})$ with \[
	\mathbb{P}^N=\text{Law}(\mathbf{Q}^N,B^N) \quad \text{and} \quad \mathbb{P}^*=\text{Law}(\mathbf{Q}^*,B^*).
	\]
	Notice that we simply have
		\[
		\langle \mathbf{P}^N,|\mathcal{M}_t(\mathbf{P}^N,\boldsymbol{\cdot}\hspace{1pt})|^p\rangle
		=\frac{1}{N}\sum_{i=1}^N \Bigl|  \int_0^t \sigma(s)\rho(s)dB_s^0 + \int_0^t \sigma(s)\sqrt{1-\rho(s)^2}dB_s^i   \Bigr|^p,
		\]
		for any given power $p>0$. This converges with probability one to a functional $G$ of the given $B^0$, where $G(\cdot)$ is bounded in terms of its argument, e.g.~by the Birkhoff--Khinchin ergodic theorem. Since the laws of $(\mathbf{P}^N,B^0)$ and $(\mathbf{Q}^N,B^N)$ agree, for each $N\geq1$, we deduce that $\langle \mathbf{Q}^N,|\mathcal{M}_t(\mathbf{Q}^N,\boldsymbol{\cdot}\hspace{1pt})|^p\rangle - G(B^N)$ converges to zero in probability, and using also that $B^N\rightarrow B^*$ almost surely, with $G(B^N)$ bounded in terms of $B^N$, we can thus find a subsequence for which $\langle \mathbf{Q}^N,|\mathcal{M}_t(\mathbf{Q}^N,\boldsymbol{\cdot}\hspace{1pt})|^p\rangle$ is bounded in $N\geq1$ with probability one (the bound being random). After passing to this subsequence, $(\mathbf{Q}^N,B^N)$ satisfies the assumptions of Lemma \ref{lem:The-functional-continuity} with probability one, so the functional continuity in Lemma \ref{lem:The-functional-continuity} gives that $\Psi(\mathbf{P}^N)$ converges in law to $\Psi(\mathbf{P}^*)$ along a subsequence (still indexed by $N$). Analogous arguments apply to $\Upsilon(\mathbf{P}^N)$ and $\Theta(\mathbf{P}^N,B^0)$. Next, we can observe that, for all $N\geq1$,
\[
\mathbb{E}\bigl[\Psi(\mathbf{P}^{N})\bigr]=\mathbb{E}\Bigl[F \Bigl( \int_{0}^{\cdot}\sigma(s)\rho(s)dB_{s}^{0}+ \int_{0}^{\cdot}\sigma(s)\sqrt{1-\rho(s)^{2}}dB_{s}^{1}\Bigr)\Bigr]=0.
\]
Furthermore, $\E[|\Psi(\textbf{P}^N)|^p]$ is clearly uniformly bounded in $N\geq1$, for any given $p>1$, so we have uniform integrability. Since $\Psi(\mathbf{P}^N)$ converges weakly to $\Psi(\mathbf{P}^*)$, as we argued above, the uniform integrability gives convergence of the means  \cite[Thm.~3.5]{billingsley} and so it holds by construction of $\bar{\mathbb{P}}$ that
\[
	\bar{\mathbb{E}}[F(\mathcal{M})]=\mathbb{E}^*[\Psi(\mathbf{P}^*)]=\lim_{N\rightarrow\infty}\mathbb{E}\bigl[\Psi(\mathbf{P}^{N})\bigr]=0.
\]
By definition of $F$ we can thus deduce that $\mathcal{M}$ is indeed a martingale under $\bar{\P}$. For the pathwise continuity of $t \mapsto \mathcal{M}_t$, note that
\begin{align*}
	\bar{\mathbb{E}}[|\mathcal{M}_t - \mathcal{M}_s|^{4}]
		 \leq \limsup_{N \to \infty} \mathbb{E}[ \langle \mathbf{P}^N, |\mathcal{M}_t - \mathcal{M}_s|^{4} \rangle]  = \mathbb{E}\Big[ \Big| \int^t_s \sigma(r) dW_r  \Big|^4 \Big] = O(|t-s|^2),
\end{align*}
	where $W$ is a standard Brownian motion and the last equality follows from Burkholder--Davis--Gundy. By Kolmogorov's continuity criterion \cite[Chp.~3, Thm.~8.8]{ethier_kurtz_1986}, we conclude that $\mathcal{M}$ has a continuous version. The proof is similar for the two other processes, using the convergence in law (along a subsequence) of $\Upsilon(\mathbf{P}^N)$ to $\Upsilon(\mathbf{P}^*)$ and of $\Theta(\mathbf{P}^N,B^0)$ to $\Theta(\mathbf{P}^*,B^0)$, respectively.
\end{proof}

Based on the previous proposition, we can now finalise the proof of Theorem \ref{Thm:Existence}.
\begin{proof}[Proof of Theorem \ref{Thm:Existence}]
Fix the probability space $(\bar{\Omega}, \bar{\mathbb{P}}, \mathcal{B}(\bar{\Omega}))$ as introduced in (\ref{P_bar})
	and recall that this includes fixing a limit point $(\mathbf{P}^{*},B^{0})$ of $(\mathbf{P}^{N},B^{0})$. Now define a c\`adl\`ag process $X$ on $\bar{\Omega}$  by $X(\mu,w,\eta):=\eta$.
Then it holds by construction of $\bar{\mathbb{P}}$ that
\[
	\bar{\mathbb{P}}\bigl(X\in A,\,(\mathbf{P}^{*},B^{0})\in S\bigr) =\int_{S}\mu(A)d\mathbb{P}^{*}(\mu,w),
\]
where we recall that $\mathbb{P}^{*}$ is the joint law of $(\mathbf{P}^{*},B^{0})$. Consequently, we have
\[
	\bar{\mathbb{P}}(X \in A \,| \, \mathbf{P}^*, B^0) = \mathbf{P}^*(A),\qquad \forall A\in \mathcal{B}(D_{\mathbb{R}}),
\]
as desired. Next, Proposition \ref{prop:martingales} gives that 
\[
	\mathcal{M}_t = X_t - X_0 - \int^t_0 b(s,X_s)ds - \alpha \mathbf{P}^*(\hat{\tau} \leq t)	
\]
is a continuous martingale with
\[
\langle \mathcal{M} , \mathcal{M} \rangle_t = \int^t_0 \sigma(s)^2ds 
	\qquad \text{and}\qquad
	\langle \mathcal{M}, B^0 \rangle_t = \int^t_0 \sigma(s) \rho(s) ds.
\]
By L\'evy's characterisation theorem, we deduce that there exists a Brownian motion, $B$, that is independent of $B^0$ and for which
\[
	\mathcal{M}_t = \int^t_0 \sigma(s) \rho(s) dB^0_s + \int^t_0 \sigma(s) \sqrt{1 - \rho(s)^2}dB_s.
\]
By the law of large numbers and continuity of the projection $\eta \mapsto \eta_0$, one readily verifies that  $X_0$ is distributed according to $\nu_0$ under $\bar{\mathbb{P}}$. However, we need Lemma \ref{lem:independence} below to ensure that $(B,B^0,\mathbf{P}) \perp B$ and $X_0\perp (B,B^0,\mathbf{P})$. By virtue of this lemma, we conclude that  $(X, \mathbf{P}^*, B^0, B)$ is a solution to \eqref{eq:Relaxed McKean-Vlasov}.

Finally, by passing to a limit as $N \to \infty$ in Lemma \ref{Minimal_jumps_particle} and then sending $\delta \downarrow 0$, we obtain ($\bar{\mathbb{P}}$-a.s.)
\begin{equation}
\label{eq:Existence_UpperBoundForPhysicalCondition}
	\Delta \hat{L}_t \leq \inf\{ x > 0 : \nu_{t-}^*( \hspace{0.4pt} [0,\alpha x]\hspace{0.4pt} ) < x \}.
\end{equation}
	Since $\hat{L}$ is c\`adl\`ag by construction, Proposition \ref{Existence_Prop_MinimalityForCommonNoise} ensures that we have equality in (\ref{eq:Existence_UpperBoundForPhysicalCondition}) and thus the existence proof is complete.
\end{proof}

It remains to verify that the limit point $\mathbf{P}^*$ of $\mathbf{P}^N$ really does become independent of the idiosyncratic Brownian motion in the large population limit.
\begin{lem}
	\label{lem:independence} In the proof of Theorem \ref{Thm:Existence}, we have mutual independence of $X_0$, $ B$, and $( B^0,\mathbf{P}^*)$ under $\bar{\mathbb{P}}$.
\end{lem}

\begin{proof}
By Assumption \ref{Assumptions}, $t\mapsto \sigma(t) \rho(t)$ is in $\mathcal{C}^{\kappa}(0,T)$ for some $\kappa>1/2$. Hence the map
	\[
	w \mapsto \int_{0}^{\cdot}\sigma(s)\rho(s)dw_s
	\]
	makes sense as a Young integral for any Brownian path $w$ \cite[Thm.~6.8]{Friz_Victoir} and it agrees almost surely with the corresponding It\^o integral against $B^0$ under $\mathbb{P}^*$. Moreover, if a sequence $(w^n)$ in $\cap_{\kappa<1/2} \hspace{2pt}\mathcal{C}^{\kappa}(0,T)$ converges uniformly to $w\in\cap_{\kappa<1/2}\hspace{2pt}\mathcal{C}^{\kappa}(0,T)$, then the pathwise integrals also converge \cite[Prop.~6.11-6.12]{Friz_Victoir}. 
	
	Given an arbitrary bounded continuous function $g\in \mathcal{C}_b(\mathbb{R})$, let
	\[\tilde{F}(\zeta , \eta ):=F(\zeta)g(\eta_0), \qquad \text{for all } (\zeta,\eta)\in C_\mathbb{R}\times D_\mathbb{R},
	\]
	where $F$ is defined as in \eqref{eq:F}.Using the above observations on Young integrals, by analogy with Lemma \ref{lem:The-functional-continuity} we then get $\Lambda(\mu^n,w^n)\rightarrow\Lambda(\mu,w)$ where
	\[
	\Lambda(\mu,w):=\Bigl\langle \mu, \tilde{F}\Bigl( \mathcal{M}(\mu,\boldsymbol{\cdot} \,)-\int_{0}^{\cdot}\!\sigma(s)\rho(s)dw_s) , \boldsymbol{\cdot} \Bigr)  \Bigr \rangle^{\!2},
	\]
	whenever $(\mu^n,w^n)\rightarrow (\mu,w)$, for $\mathbb{P}^*$-almost every $(\mu,w)$, provided $\mu^n$ satisfies the integrability assumption of Lemma \ref{lem:The-functional-continuity} and provided $w^n$ is a sequence of Brownian paths.
	By the same reasoning as in the proof of Proposition \ref{prop:martingales}, a Skorokhod representation argument gives that it suffices to only have continuity of $\Lambda$ along such sequences, in order for us conclude that $\Lambda(\mathbf{P}^N,B^0)$ converges in law to $\Lambda(\mathbf{P}^*,B^0)$.

	Next, we define the process
	\[
	Y_t:=\mathcal{M}_t-\int_{0}^{t}\sigma(s)\rho(s)dB^0_s=\int^t_0 \sigma(s) \sqrt{1 - \rho(s)^2}dB_s.
	\]
	By definition, we then have
	
	\[
	\int_{\Omega^{*}}\bigl\langle \mathbf{P}^{*}(\omega),\tilde{F}\bigl(Y(\omega,\boldsymbol{\cdot}\,),\boldsymbol{\cdot}\, \bigr)\bigr\rangle^{2}d\mathbb{P}^{*}(\omega) = \mathbb{E}^*[\Lambda(\textbf{P}^{*},B^0)].
	\]
It is immediate that $\Lambda(\mathbf{P}^N,B^0)$ is uniformly integrable, so the convergence in law of $\Lambda(\mathbf{P}^N,B^0)$ to $\Lambda(\mathbf{P}^*,B^0)$, as established above, yields convergence of the means. Hence we get
	\begin{equation}\label{eq:idio_BM}
	\int_{\Omega^{*}}\bigl\langle \mathbf{P}^{*}(\omega),F\bigl(Y(\omega,\boldsymbol{\cdot}\,), \boldsymbol{\cdot}\,\bigr)\bigr\rangle^{2}d\mathbb{P}^{*}(\omega)=\lim_{N\rightarrow\infty}\mathbb{E}\bigl[\Lambda(\textbf{P}^N,B^0)\bigr]=0,
	\end{equation}
	where we have used that
	\[
	\mathbb{E}[\Lambda(\mathbf{P}^{N},B^0)] \leq \frac{C}{N^{2}}\sum_{i=1}^{N}\mathbb{E}\biggl[\Bigl(\int_{s_{0}}^{t_{0}}\sigma(s)\sqrt{1-\rho(s)^2}dB_{s}^{i}\Bigr)^{\!2}\biggr]\rightarrow 0,
	\]
	as $N\rightarrow0$, by the properties of the $B^i$'s and the definition of $\Lambda$. Indeed, since the $B^i$'s are independent Brownian motions, and since they are independent of the initial conditions, the definition of $\tilde{F}$ and the tower law yield
	\begin{equation*}
 \mathbb{E}\biggl[ \tilde{F} \Bigl(\int_0^{\cdot}\sigma(s)\sqrt{1-\rho(s)^2}dB_{s}^{i}, X_0^i\Bigr)  \tilde{F} \Bigl(\int_0^{\cdot}\sigma(s)\sqrt{1-\rho(s)^2}dB_{s}^{j}, X_0^j\Bigr) \biggr]=0,
	\end{equation*}
	for all $i\neq j$, which kills the cross terms when writing out the expression for $\mathbb{E}[\Lambda(\mathbf{P}^{N},B^0)]$.
	
	From \eqref{eq:idio_BM} and the definition of $\tilde{F}$, we can conclude that, for $\P^*$-a.e.~$\omega \in \Omega^{*}$, the process $\eta \mapsto Y(\omega,\eta)$ is a martingale under $\textbf{P}^*(\omega)$ conditional on the random variable  $\eta\mapsto \eta_0$.
	
	A similar argument, this time for
	\[
	\Gamma(\mu,w):=\Bigl \langle \mu, \tilde{F}\Bigl( \Bigl(  \mathcal{M}(\mu,\boldsymbol{\cdot} \,)-\int_{0}^{\cdot}\!\sigma(s)\rho(s)dw_s \Bigr)^{\!2} - \int_0^\cdot \! \sigma(s)^2(1-\rho(s)^2)ds, \boldsymbol{\cdot} \Bigr) \Bigr \rangle^{\!2},
	\]
	shows that the quadratic variation of $\eta \mapsto Y(\omega,\eta)$ is $\int_{0}^{\cdot}\sigma(s)^2(1-\rho(s)^2)ds$ conditionally on $\eta \mapsto \eta_0$. Therefore, as in the proof of Theorem \ref{Thm:Existence}, Levy's characterisation theorem allows us to deduce that, for $\P^*$-a.e. $\omega \in \Omega^{*}$, the process  $\eta \mapsto Y(\omega,\eta)$ is a time-changed Brownian motion under $\textbf{P}^*(\omega)$ conditional on $\eta \mapsto \eta_0$. Furthermore, by the strong law of large numbers, one easily sees that $\eta \mapsto \eta_0$ is distributed according to $\nu_0$ under $\textbf{P}^*(\omega)$, and so is $X_0$ under $\bar{\mathbb{P}}$. In particular, we can thus deduce that
	\begin{align*}
	\bar{\mathbb{P}}\bigl( X_0 \in I , \, Y \in A, \, (\mathbf{P}^{*},B^{0})\in S \bigr)  &=\int_{S}(\mathbf{P}^{*}(\omega)\hspace{-0.8pt})(\hspace{0.6pt}\{\eta:Y(\omega,\eta)\in A , \eta_0\in I\}\hspace{0.2pt})d\mathbb{P}^{*}(\omega)\\
	& =\int_{S}\bar{\mathbb{P}}(Y\in A)\bar{\mathbb{P}}(X_0 \in I)d\mathbb{P}^{*}(\omega) \\
	&=\bar{\mathbb{P}}(X_0 \in I) \bar{\mathbb{P}}(Y\in A) \bar{\mathbb{P}}\bigl(\hspace{-0.5pt}(\mathbf{P}^{*},B^{0})\in S \bigr).
	\end{align*}
	Finally, we can observe that $B=(B_t)_{t\in[0,T]}$ is given as the uniform limit in probability (under $\bar{\mathbb{P}}$) of Riemann sums of $\int_{0}^{\cdot} (\sigma(s) \sqrt{1-\rho(s)^2})^{-1} d Y_s$. By the above result, we have that $X_0$, $(\mathbf{P}^{*},B^{0})$, and any one of these finite sums are mutually independent  under $\bar{\mathbb{P}}$, so the same is true for $B$ in the limit. This finishes the proof.
\end{proof}

\section{Filtrations and a remark on numerics}\label{subsec:filtration} 

This final section is concerned mainly with the construction of the filtrations discussed at the start of Section 3. First, we prove Proposition \ref{filtration} concerning the overall filtration  $\mathcal{F}_t$, and we then derive verify the properties of the subfiltration $\mathcal{F}^{B^0,\boldsymbol{\nu}}_t$ used in Section \ref{subsec:SPDE} above. Next, we dedicate Section \ref{sec:numerics} to a brief outline of the numerical scheme used to generate the simulations presented in Figure \ref{fig:theonlyfigure} in the introduction.

\subsection{Construction of filtrations and their properties}\label{subsec:filtrations}

As in the proof of Theorem \ref{Thm:Existence} above, we continue to work with the background space $(\bar{\Omega},\mathcal{F},\mathbb{P})$ from (\ref{P_bar}), and the co-countable set of times $\mathbb{T}$ from (\ref{blackboard_t}), for a fixed limit point $(\mathbf{P}^*,B^0)$. Associated to this limit point, we have the  absorbing marginal flow
\[
t\mapsto\boldsymbol{\nu}^*_t:=\mathbf{P}^*(\hat{\eta}_t\in \hspace{-1pt}\cdot \, , \, t<\hat{\tau}),
\]
where, as in the previous subsections, $\hat{\eta}$ denotes the canonical process on $D_\mathbb{R}$ and $\hat{\tau}$ is its first hitting time of zero.

Recall that $\bar{\Omega}=\mathcal{P}(D_\mathbb{R})\times C_\mathbb{R} \times D_\mathbb{R}$ with its product Borel $\sigma$-algebra $\mathcal{F}$. Recall also that $X_t(\mu,w,\zeta)=\zeta_t$ and $B^0_t(\mu,w,\zeta)=w_t$, and note that $\boldsymbol{\nu}^*_t(\mu,w,\zeta)=\mu(\hat{\eta}_t \in \hspace{-1pt}\cdot \, , \, t<\hat{\tau})$. Finally, we can observe that, by virtue of these processes satisfying  (\ref{eq:Relaxed McKean-Vlasov}), $B_t$ will be adapted to any filtration that makes $X_t$, $B^0_t$, and $ \langle \boldsymbol{\nu}^*_t , 1 \rangle$ adapted.

Let us begin by constructing the desired filtration for just $B^0$ and $\boldsymbol{\nu}^*=(\boldsymbol{\nu}^*_t)_{t\geq0}$. As formulated in Proposition \ref{filtration}, we want each marginal $\boldsymbol{\nu}_t^*$ to be measurable as an $\text{M}(\mathbb{R})$-valued random variable, where  the Borel-sigma algebra for $\text{M}(\mathbb{R})$ comes from the duality with (a closed subspace of) the bounded continuous functions $\mathcal{C}_b(\mathbb{R})$ in the uniform topology. Therefore, we will work with the filtration which, at time $t$, is generated by the projections $B^0\mapsto B^0_s$ and $\boldsymbol{\nu}^*\mapsto \langle \boldsymbol{\nu}^*_s , \phi \rangle$ for all $s\leq t$ and all $\phi\in \mathcal{C}_b(\mathbb{R})$. That is, we define
\begin{equation*}
\mathcal{F}^{B^0\!\!,\hspace{1pt}\boldsymbol{\nu}^*}_t :=\sigma\Bigl( (\mu,w) \mapsto \bigl( \langle \mu, \phi(\hat{\eta}_s)\mathbf{1}_{ s<\hat{\tau}  } \rangle, w_s\bigr) :\phi\in \mathcal{C}_b(\mathbb{R}),\,s\leq t  \Bigr)
\end{equation*}
The critical point is now to ensure that $B^0$ remains a Brownian motion in this filtration. The next lemma will allow us to deduce just that.

\begin{prop}[Convergence of the absorbing marginal flow]\label{marginal_flow_weak_conv}
	Let $\boldsymbol{\nu}^N$ be given by \eqref{eq:finite_flow} under Assumption \ref{Assumptions}. For any family of bounded continuous functions $\phi_1,\ldots,\phi_k \in{\mathcal{C}_b(\mathbb{R})}$ we have
	\[
	\bigl( \langle \boldsymbol{\nu}^N_{t_1} , \phi_1 \rangle , \ldots, \langle \boldsymbol{\nu}^N_{t_k} , \phi_k \rangle \bigr) \Rightarrow \bigl( \langle \boldsymbol{\nu}^*_{t_1} , \phi_1 \rangle, \ldots, \langle \boldsymbol{\nu}^*_{t_k} , \phi_k \rangle \bigr),
	\]
	as $N\rightarrow \infty$, for any finite set of times $t_1\ldots,t_k\in \mathbb{T}$.
\end{prop}
\begin{proof}
	Fix an arbitrary $t\in\mathbb{T}$ and any $\phi\in \mathcal{C}(\mathbb{R})$. We can then observe that
	\[
	\langle \boldsymbol{\nu}^N_t , \phi \rangle = \langle \boldsymbol{\nu}^N_t , \phi_0 \rangle, \quad \text{where} \quad \phi_0(x):=\begin{cases}
	\phi(x) & x\geq0\\
	\phi(0) & x<0
	\end{cases},
	\]
	 and likewise for $\boldsymbol{\nu}^*_t$. Noting that $\phi_0(\hat{\eta}_{t\land\hat{\tau}})=\phi_0(\hat{\eta}_t)\mathbf{1}_{t<\hat{\tau}}+\phi_0(\hat{\eta}_{\hat{\tau}})\mathbf{1}_{t\geq\hat{\tau}}$ with $\phi_0(\hat{\eta}_{\hat{\tau}})={\phi}(0)$, we therefore have the decomposition
	\[
	\langle \boldsymbol{\nu}^N_t , \phi \rangle = \langle \mathbf{P}^N , \phi_0(\hat{\eta}_{t\land \hat{\tau}}) \rangle - \phi(0) \mathbf{P}^N(t\geq \hat{\tau}),
	\]
and, again, the same is true for the limit $\boldsymbol{\nu}^*_t$. Hence the proposition will follow from the continuous mapping theorem if we can show that, given any $t\in\mathbb{T}$, for $\mathbb{P}^*$-almost every $\mu$,
	\begin{equation}\label{absorbing_flow}
	\langle \mu^n , \phi_0(\hat{\eta}_{t\land \hat{\tau}}) \rangle \rightarrow \langle \mu , \phi_0(\hat{\eta}_{t\land \hat{\tau}}) \rangle \quad \text{and} \quad  \mu^n(t\geq \hat{\tau}) \rightarrow \mu(t\geq \hat{\tau})
	\end{equation}
	whenever $\mu^n\rightarrow\mu$ in $(\mathcal{P}(D_{\mathbb{R}}),\mathfrak{T}_{\text{M1}}^{wk})$. Since $t\in\mathbb{T}$, we have $\mu(t\geq \hat{\tau})=\mu(t-\geq \hat{\tau})$ for $\mathbb{P}^*$-almost every $\mu$, so the claim is true for the second part of (\ref{absorbing_flow}) by Lemma \ref{prop:loss_continuity}.
	
	Noting that $t\in\mathbb{T}$ also implies $\mu(\hat{\eta}_t=\hat{\eta}_{t-})$ for $\mathbb{P}^*$-almost every $\mu$, Skorokhod's representation theorem allows us to rewrite the first part of (\ref{absorbing_flow}) as
	$\mathbb{E}[\phi(Z^n_{t\land\tau^{n}})] \rightarrow \mathbb{E}[\phi(Z_{t\land\tau})]$
	whenever $Z^n\rightarrow Z$ a.s.~in $(D_\mathbb{R},\text{M1})$ with $\mathbb{P}(Z_t=Z_{t-})=1$.
	
	Since $t$ is almost surely a continuity point of $Z$, there is an event $\Omega_t$ of probability one, for which $Z_t^n\rightarrow Z_t$ and $\inf_{s\leq t}Z_s^n \rightarrow \inf_{s\leq t}Z_s $, by \cite[Thms.~12.4.1 \& 13.4.1]{whitt}. On the event $\Omega_t \cap \{ \inf_{s \leq t} Z_s >0 \}$, we eventually have $\phi_0(Z^n_{t\land \tau^n}) = \phi_0(Z^n_t)$, which converges to $\phi_0(Z_t)=\phi_0(Z_{t \land \tau})$. Similarly, on the event $\Omega_t\cap \{ \inf_{s\leq t}Z_s<0 \}$, we eventually have  $\phi_0(Z^n_{t\land \tau^n}) = \phi(0)$, which agrees with $\phi_0(Z_{t\land \tau})=\phi(0)$. For the remaining event $\Omega_t \cap \{ \inf_{s \leq t} Z_s = 0 \}$, it follows as in the proof of Theorem \ref{prop:loss_continuity} that, up to a null set, $Z_t=0$ and hence $\phi_0(Z_{t \land \tau})=\phi_0(Z_t)=\phi(0)$. In principle, $\phi_0(Z^n_{t\land \tau^n})$ may oscillate between the values $\phi(0)$ and $\phi_0(Z^n_t)$, but the latter converges to $\phi_0(Z_t)=\phi(0)$. Combining the above, we have $\phi_0(Z^n_{t \land \tau^n})\rightarrow \phi_0(Z_{t \land \tau})$ almost surely, so   $\mathbb{E}[\phi_0(Z_{t\land\tau})]$ is indeed the limit of $\mathbb{E}[\phi_0(Z^n_{t\land\tau^{n}})]$, by dominated convergence.
\end{proof}

Let $G$ be an arbitrary bounded continuous function $G:C_\mathbb{R}\rightarrow \mathbb{R}$. Fixing any $t_1,\ldots,t_n \in \mathbb{T}\cap[0,t]$ and $f_1,\ldots,f_n\in \mathcal{C}_b(\mathbb{R})$, the previous proposition gives that
\begin{align}\label{filtration_independ}
\bar{\mathbb{E}}\Bigl[ G(B^0_{t+\cdot}-B^0_t) \prod_{i=1}^n f_i(\langle \boldsymbol{\nu}^*_{t_i}, \phi_i \rangle) \Bigr] & = \lim_{N\rightarrow\infty} \mathbb{E}\Bigl[ G(B^0_{t+\cdot}-B^0_t)  \prod_{i=1}^n f_i(\langle \boldsymbol{\nu}^N_{t_i}, \phi_i \rangle)  \Bigr] \nonumber\\
& = \bar{\mathbb{E}}\Bigl[ G(B^0_{t+\cdot}-B^0_t)\Bigr] \bar{\mathbb{E}}\Bigl[ \prod_{i=1}^n f_i(\langle \boldsymbol{\nu}^*_{t_i}, \phi_i \rangle)  \Bigr].
\end{align}
Here the second equality follows from the simple observation that $(\boldsymbol{\nu}^N_s)_{s\leq t}$ is independent of $(B^0_{s+t}-B^0_t)_{s\geq 0}$, since the finite particle system has a unique strong solution with $(B^1,\ldots,B^N)$ independent of $B^0$. We deduce from (\ref{filtration_independ}) that the future increments of $B^0$ are independent of $\mathcal{F}^{B^0\!\!,\hspace{1pt}\boldsymbol{\nu}^*}_t$ and so $B^0$ remains a Brownian motion in this filtration. Furthermore, the filtration is independent of $B$, since both $B^0$ and $\boldsymbol{\nu}^*$ are independent of $B$, as ensured by Lemma \ref{lem:independence}.

 Finally, we need a larger filtration $\mathcal{F}_t$ for which the full solution $(X,B,B^0,\boldsymbol{\nu}^*)$ is adapted. This is achieved by also including the pre-images of the projections $X\mapsto X_s$, and we therefore define
\[
\mathcal{F}_t := \sigma\Bigl( (\mu,w,\zeta) \mapsto \bigl( \langle \mu, \phi(\hat{\eta}_s)\mathbf{1}_{ s<\hat{\tau}  } \rangle, w_s,\zeta_s\bigr) :\phi\in \mathcal{C}(\mathbb{R}),\,s\leq t  \Bigr).
\]
From our construction of the solution $(X,B,B^0,\boldsymbol{\nu}^*)$ on $(\bar{\Omega},\bar{\mathbb{P}},\mathcal{F})$, we note that we can write $B_t(\mu,w,\zeta)$ in terms of $\zeta_s$, $w_s$, and $\langle \mu, \phi(\hat{\eta}_s)\mathbf{1}_{ s<\hat{\tau}}  \rangle$ with $\phi\equiv1$, for $s\leq t$, so $B_t$ is $\mathcal{F}_t$-measurable. By analogy with (\ref{filtration_independ}), at any time $t\geq0$, the future increments of $B$ and $B^0$ are independent of $\mathcal{F}_t$. Thus, they are both Brownian motions with respect to this filtration. This completes the proof of Proposition \ref{filtration}.

\subsection{Outline of the numerical scheme}\label{sec:numerics}
The density approximations shown in Figure \ref{fig:theonlyfigure} were generated according to the following quadrature-based algorithm. Begin with an initial probability density function, $V_0$, and fix step sizes $\delta t$ and $\delta x$ for creating a mesh on the time-space grid $[0,T] \times [0,u]$. Here, the value $u$ is an upper bound chosen large enough so that truncating the solution above the level $u$ in the spatial component is an acceptable approximation to the process on the positive half-line.

At each point on the spatial slice of the grid at time $\delta t$, approximate the density at that point by evolving the solution according heat equation with drift term given by the common noise increment $B^0_{\delta t} - B^0_0$ (that is, numerically evaluate the integral corresponding to this transformation at all points on this slice of the space grid). This procedure gives a candidate density $U_{\delta t}$.

Now estimate the candidate loss
\[
L^{(0)}_{\delta t} = 1 - \int^\infty_0 U_{\delta t}(x)dx
\]
numerically, and then iterate the process
\[
L^{(n+1)}_{\delta t} = L^{(0)}_{\delta t} + \int^{\alpha L^{(n)}_{\delta t}}_0 U_{\delta t}(x)dx
\]
until some stopping rule is reached. The rule we used here is to stop the iterations when $L^{(n+1)}_{\delta t} - L^{(n)}_{\delta t} < \varepsilon$ for the first time, for some threshold $\varepsilon > 0$. 

The final value in the sequence $n\mapsto L^{(n)}_{\delta t}$ is our approximation, $\hat{L}_{\delta t}$, to the loss over the first time step. We then translate the solution by $\alpha \hat{L}_{\delta t}$ in the direction of the boundary (by re-interpolating $U_{\delta t}$ onto the shifted grid) to give our approximation $\hat{V}_{\delta t}$ of the true density, $V_{\delta t}$, after the first time-step. This process is then repeated across all time-steps sequentially. 

We conjecture that this algorithm gives the correct approximation in the limit as the mesh-sizes vanish. Indeed it would be interesting to attempt a rigorous proof, modifying the problem by assuming that quadrature and interpolation steps can be carried out to infinite-precision (i.e.~to prove that the repeated discrete steps of approximating by the heat equation together with the contagion approximation yields the correct limiting behaviour). It is worth emphasising that this sequential approximation to the loss due to contagion corresponds exactly to the construction in \cite[Equation (2.1)]{hambly_ledger_sojmark_2018}. Finally, we note that the number of operations in the above procedure scales as
\[
N_\textrm{time-steps} \times N_\textrm{quadrature points} \times ( N_\textrm{space-steps} + N_\textrm{contagion-steps} ).
\]
The precise properties of the proposed numerical scheme are left for future research.


\paragraph*{Acknowledgements.}

 This project was initiated while A. S{\o}jmark was at the University of Oxford, supported by the EPSRC award EP/L015811/1. We thank Ben Hambly for initial discussions and advice at that time. Moreover, we thank two anonymous referees and the associate editor for their encouragement to expand on certain results in Sections 2 and 3 as well as their suggestions for improvement of the overall presentation.

\bibliographystyle{alpha}

\begin{thebibliography}{99}


\bibitem{billingsley}
P.~Billingsley.
\newblock {\em Convergence of probability measures}.
\newblock Wiley Series in Probability and Statistics: Probability and
  Statistics. John Wiley \& Sons Inc., New York, second edition, 1999.
\newblock A Wiley-Interscience Publication.

\bibitem{brunel2000}
N.~Brunel.
\newblock Dynamics of Sparsely Connected Networks of Excitatory and Inhibitory Spiking Neurons.
\newblock \emph{Journal of Computational Neuroscience} 8(3): 183--208, 2000.

\bibitem{brunel-hakim_1998}
N.~Brunel and V.~Hakim.
\newblock{Fast Global Oscillations in Networks of Integrate-and-Fire
Neurons with Low Firing Rates}. \emph{Neural Computation} 11(7): 1621--1671, 1999

\bibitem{carmona2018}
R.~Carmona and F.~Delarue.
\newblock {\em Probabilistic Theory of Mean Field Games with Applications II. Mean Field Games with Common Noise and Master Equations}.
\newblock Probability Theory and Stochastic Modelling 84,
Springer, 2018.


\bibitem{carmona2016}
R.~Carmona, F.~Delarue and D.~Lacker.
\newblock Mean field games with common noise.
\newblock \emph{Annals of Probability} 44(6): 3740--3803, 2016.


\bibitem{carrillo2011}
M.J.~C\'{a}ceres, J.A.~Carrillo and B.~Perthame.
\newblock Analysis of non-linear noisy integrate \& fire neuron models: blow-up
  and steady states.
\newblock \emph{The Journal of Mathematical Neuroscience} 1(7), 2011.


\bibitem{carrillo2013}
J.A.~Carrillo, M.D.M.~Gonz\'{a}lez, M.P.~Gualdani  and  M.E.~Schonbek.
\newblock Classical solutions for a nonlinear Fokker--Planck equation arising in computational neuroscience.
\newblock \emph{Comm. Partial Differential Equations} 38: 385--409, 2013.

\bibitem{carrillo2015}
J.A.~Carrillo,  B.~Perthame, D.~Salort and  D.~Smets.
\newblock Qualitative properties of solutions for the noisy integrate and fire model in computational neuroscience.
\newblock \emph{Nonlinearity} 28: 3365--3388, 2015.

\bibitem{dawson} D.~Dawson and J.~Vaillancourt.
\newblock Stochastic McKean--Vlasov equations.
\newblock  \emph{NoDEA} 2: 199--229, 1995

\bibitem{DIRT_AAP}
F.~Delarue, J.~Inglis, S.~Rubenthaler and E.~Tanr\'{e}.
\newblock {Global solvability of a networked integrate-and-fire model of
  McKean--Vlasov type}, {\em Annals of Applied Probability} 25(4): 2096--2133, 2015.


\bibitem{DIRT_SPA}
F.~Delarue, J.~Inglis, S.~Rubenthaler and E.~Tanr\'{e}.
\newblock Particle systems with singular mean-field self-excitation.
  {A}pplication to neuronal networks, \emph{Stochastic Processes and their Applications} 125(6): 2451--2492, 2015.

\bibitem{delarue_sergey_shkolni}
F.~Delarue, S.~Nadtochyi, and M.~Shkolnikov 
\newblock Global solutions to the supercooled Stefan problem with blow-ups: regularity and uniqueness. Preprint  available at \texttt{arXiv:1902.05174}, 2019.

\bibitem{ethier_kurtz_1986}
S.~Ethier and T.~Kurtz.
\newblock \emph{Markov Processes: Characterization and Convergence.}
\newblock Wiley, 1986.


\bibitem{Friz_Victoir}
P.~Friz and N.~Victoir.
\emph{Multidimensional Stochastic Processes as Rough Paths}.
Cambridge Studies in Advanced Mathematics, Cambridge University Press, 2010.

\bibitem{bouchaud}
S.~Gualdi, J.-P.~Bouchaud, G. Cencetti, M. Tarzia and F. Zamponi.
\newblock Endogenous crisis waves: a stochastic model with synchronized collective behavior.
\newblock \emph{Physical Review Letters}, 114, 088701: 1--5, 2015.

\bibitem{hambly_ledger_2017}
B.~Hambly and S.~Ledger.
\newblock A stochastic McKean--Vlasov equation for absorbing diffusions on the half-line.
\newblock \emph{Annals of Applied Probability}, 27: 2698--2752, 2017.


\bibitem{hambly_ledger_sojmark_2018}
B.~Hambly, S.~Ledger and A.~S{\o}jmark.
\newblock A McKean--Vlasov equation with positive feedback and blow-ups.
\newblock \emph{Annals of Applied Probability} 29(4): 2338--2373, 2019.


\bibitem{hambly_sojmark_2017}
B.~Hambly and A.~S{\o}jmark.
\newblock An SPDE model for systemic risk with endogenous contagion. \emph{Finance and Stochastics} 23(3): 535--594, 2019.


\bibitem{ledger_2016}
S.~Ledger.
\newblock Skorokhod's {M1} topology for distribution-valued processes.
\newblock {\em Electronic Communications of Probability} 21(1): 1--11, 2016.

\bibitem{LS_unique}
S.~Ledger and A. S{\o}jmark
\newblock Uniqueness of contagious McKean--Vlasov systems in the weak feedback regime.
\newblock {\em Bulletin of the London Mathematical Society},  52(3): 448--463, 2020.

\bibitem{mattia_2002}
M.~Mattia and P.~Del Giudice.
\newblock Population dynamics of interacting spiking neurons. \emph{Physical review E: Statistical, nonlinear, and soft matter physics} 66(5 Pt1): 051917, 2002.


\bibitem{Moreno-Bote_2010}
R.~Moreno-Bote and N.~Parga.
\newblock Response of integrate-and-fire neurons to noisy inputs filtered by
synapses with arbitrary time scales: firing rate and correlations.
\newblock  {\em Neural Computation} 22(6), 1528--1572, 2010.



\bibitem{nadtochiy_shkolnikov_2017}
S.~Nadtochiy and M.~Shkolnikov.
\newblock Particle systems with singular interaction through hitting times: Application in systemic risk modeling. \emph{Annals of Applied Probability} 29(1), 89--129, 2019.


\bibitem{nadtochiy_shkolnikov_2018}
S.~Nadtochiy and M.~Shkolnikov.
\newblock Mean field systems on networks, with singular interaction through hitting times. \emph{Annals of Probability} 48(3), 1520--1556, 2020.


\bibitem{sznitman}
A.-S. Sznitman.
\newblock {Topics in propagation of chaos}.
\newblock In Paul-Louis Hennequin, editor, {\em Ecole d'Et\'{e} de
  Probabilit\'{e}s de Saint-Flour {XIX} -- 1989}, volume 1464 of {\em Lecture
  Notes in Mathematics}, chapter~3, pages 165--251. Springer Berlin Heidelberg,
  1991.
  
  \bibitem{vadim_reisinger}
 V.~Kaushansky and C.~Reisinger.
 \newblock Simulation of particle systems interacting through hitting times. \emph{Discrete and Continuous Dynamical Systems --- Series B}
 24(10): 5481--5502, 2019.
  
  \bibitem{torcini}
 E.~Ullner, A.~Politi and A.~Torcini.
  \newblock Ubiquity of collective irregular dynamics in balanced networks of spiking neurons \emph{Chaos: An Interdisciplinary Journal of Nonlinear Science}, 28(8), 2018.


\bibitem{whitt}
W.~Whitt.
\newblock {\em Stochastic-Process Limits: An Introduction to Stochastic-Process
  Limits and Their Application to Queues}.
\newblock Springer Series in Operations Research and Financial Engineering.
  Springer, 2002.


\end{thebibliography}

\end{document}